\newif\ifital  \italfalse  
\newif\ifdraft  \draftfalse  
\newif\ifelsevier  \elsevierfalse  
\ifelsevier
\documentclass{elsarticle}
\else
\documentclass[a4paper,11pt]{article}
\fi
\newcommand{\JSLandscape}%
   {\ifital \thispagestyle{empty}\mbox{ }\clearpage%
       \addtocounter{page}{-1}\fi}
\ifelsevier
\else
\newlength{\augwidth}
\newlength{\augheight}
\fi
\ifelsevier
\else
\usepackage[mathlines,displaymath]{lineno}
\ifdraft
\linenumbers
\fi
\fi
\usepackage{vaucanson-g}
\usepackage[dvips]{graphicx}
\usepackage{psfrag}
\usepackage[loose]{subfigure}
\usepackage{xspace}
\usepackage{amsmath,amsfonts,amscd,amssymb,amsthm,amstext}
\theoremstyle{plain}
\newtheorem{theorem}{Theorem}[section]
\newtheorem{corollary}[theorem]{Corollary}
\newtheorem{lemma}[theorem]{Lemma}
\newtheorem{proposition}[theorem]{Proposition}

\newtheorem{definition}[theorem]{Definition}
\theoremstyle{definition}
\newtheorem{remark}[theorem]{Remark}
\newtheorem{assumption}[theorem]{Assumption}

\newtheorem*{notation*}{Notation}
\newtheorem{example}[theorem]{Example}
\makeatletter
\newenvironment{proofss}[1][\proofname]{\par
  \normalfont \topsep6\p@\@plus6\p@\relax
  \trivlist
  \item[\hskip\labelsep
        \itshape
    #1\@addpunct{.}]\ignorespaces
}{%
  \endtrivlist\@endpefalse
}
\makeatother
\usepackage{calc}
\newcounter{hours}\newcounter{minutes}
\newcommand\printtime%
  {\setcounter{hours}{\time/60}%
   \setcounter{minutes}{\time-\value{hours}*60}%
  \thehours\,h\,\,\theminutes}

\newif\ifnote  \notefalse  
\newcommand{\NEM}[1]%
   {\ifnote%
    \marginpar[\begin{flushright}%
               {\sl {\scriptsize #1}}%
               \end{flushright}]%
              {\begin{flushleft}%
               {\sl {\scriptsize #1}}%
               \end{flushleft}}%
	\fi
	}%
\newcommand{\NeM}[1]{\NEM{#1}}
\newcommand{\CMT}[1]%
   {
    \par
    \noindent
    XXXXXX\PushLine  
{\small\textsl{Commentaire trop grand pour rentrer dans la marge...}} 
    \PushLine XXXXXX
    \par
    {\small{\sl#1}}
    \par
    \noindent
    XXXXX\PushLine 
{\small\textsl{...mais ce n'est pas une raison pour ne rien \'ecrire}} 
    \PushLine XXXXXX\\
    }%
\newcommand{\cmt}[1]%
   {
    \smallskip
    \par
    \noindent
	XXX\PushLine \textsc{#1}
	\PushLine XXX
	}%
\newcommand{\assum}[1]{Assumption~\ref{ass.#1}}

\newcommand{\corol}[1]{Corollary~\ref{c.#1}}
\newcommand{\defin}[1]{Definition~\ref{d.#1}}
\newcommand{\equat}[1]{Equation~(\ref{q.#1})}
\newcommand{\equnm}[1]{(\ref{q.#1})\xspace}

\newcommand{\examp}[1]{Example~\ref{e.#1}}
\newcommand{\exemp}[1]{Example~\ref{e.#1}}
\newcommand{\figur}[1]{Figure~\ref{f.#1}}
\newcommand{\lemme}[1]{Lemma~\ref{l.#1}}
\newcommand{\propo}[1]{Proposition~\ref{p.#1}}

\newcommand{\remar}[1]{Remark~\ref{r.#1}}
\newcommand{\theor}[1]{Theorem~\ref{t.#1}}
\newcommand{\secti}[1]{Sec.~\ref{s.#1}}
\newcommand{\EoP}{\hbox{}\hfill\qedsymbol\hbox{}}%
\newcommand{\dex}[1]{(#1)}
\newcommand{\smallO}[1]%
   {\ensuremath{\mathop{}\mathopen{}o\mathopen{}\left(#1\right)}}
\def\le{\leqslant}
\def\leq{\leqslant}
\def\ge{\geqslant}
\def\geq{\geqslant}

\newcommand{\medskipneg}{\vspace*{-2ex}} 
\newcommand{\PushLine}{\hbox{}\hfill\hbox{}}
\newcommand{\Defi}[2]{\left\{#1\xmd\xmd\middle|\xmd\xmd#2\right\}}    
\newlength{\retraita}
\newlength{\listespa}
\newcommand{\EnumLbl}[1]{\rm (#1)}%
\newcommand{\jsListe}[1]%
    {\noindent\makebox[\retraita][r]{\EnumLbl{#1}}%
     \hspace*{\listespa}\ignorespaces}
\newcommand{\tha}{\jsListe{a}}
\newcommand{\thb}{\jsListe{b}}
\newcommand{\thc}{\jsListe{c}}

\newcommand{\thi}{\jsListe{i}}
\newcommand{\thii}{\jsListe{ii}}
\newcommand{\thiii}{\jsListe{iii}}
\newcommand{\thiv}{\jsListe{iv}}
\newcommand{\fa}{\forall}

\newcommand{\bk}{\mathrel{\backslash }}

\newcommand{\e}{\text{\quad}}                 
\newcommand{\ee}{\text{\qquad}}               
\newcommand{\eee}{\text{\qquad \qquad}} 
\newsavebox{\InterSymbolSpace}
\savebox{\InterSymbolSpace}{\hspace{0.125em}}
\newsavebox{\SideFormulaSpace}
\savebox{\SideFormulaSpace}{\hspace{0.2em}}
\newcommand{\msp}{\usebox{\SideFormulaSpace}} 
\newcommand{\xmd}{\usebox{\InterSymbolSpace}} 
\newcommand{\eqpnt}{\makebox[0pt][l]{\: .}}

\newcommand{\eqvrg}{\makebox[0pt][l]{\: ,}}
\newcommand{\EqVrgInt}{\: , \e }

\newcommand{\EqVrg}{\: ,}
\newcommand{\EqPnt}{\: .}

\newcommand{\quantvrg}{\, , \;}
\newcommand{\quantsp}{\ee }
\newcommand{\quantsmsp}{\e }
\newcommand{\x}{\! \times \!}
\newcommand{\LatinLocution}[1]{{\itshape #1}\xspace}
\newcommand{\cf}{\LatinLocution{cf.}}
\newcommand{\eg}{\LatinLocution{e.g.}}
\newcommand{\etc}{\LatinLocution{etc.}} 
\newcommand{\ie}{{that is, }}
\newcommand{\via}{via\xspace}
\newcommand{\Cmbb}{\mathbb{C}} 

\newcommand{\Nmbb}{\mathbb{N}}
\newcommand{\Qmbb}{\mathbb{Q}}
\newcommand{\Rmbb}{\mathbb{R}}
\newcommand{\Zmbb}{\mathbb{Z}}
\newcommand{\Ac}{\mathcal{A}}

\newcommand{\Ec}{\mathcal{E}}
\newcommand{\Fc}{\mathcal{F}}
\newcommand{\Ic}{\mathcal{I}}
\newcommand{\Kc}{\mathcal{K}}
\newcommand{\Tc}{\mathcal{T}}
\newcommand{\ShiftInd}[1]{\raisebox{-0.3ex}{$\scriptstyle{#1}$}}
\newcommand{\Rat}{\mathrm{Rat}\,}
\newcommand{\jsAutUn}[1]%
   {\mbox{$\left\langle \thinspace #1 \thinspace \right\rangle $}}
\newcommand{\aut}[1]{\jsAutUn{#1}} 
\newcommand{\auta}{\jsAutUn{Q,A,E,I,T}}
\newcommand{\jsleq}{\leqslant }

\newcommand{\act}{\mathbin{\boldsymbol{\cdot}}}
\newcommand{\matmul}{\mathbin{\cdot}}
\DeclareMathOperator{\petitoop}{\mathrm{o}}
\DeclareMathOperator{\grandoop}{\mathrm{O}}
\newcommand{\petito}[1]{\petitoop\!\left(#1\right)}
\newcommand{\grando}[1]{\grandoop\!\left(#1\right)}
\newcommand{\eulnum}{\mathrm{e}}
\renewcommand{\epsilon}{\varepsilon}
\newcommand{\CompAuto}[1]{L(#1)}
\newcommand{\C}{\Cmbb}

\newcommand{\N}{\Nmbb}
\newcommand{\Q}{\Qmbb}
\newcommand{\R}{\Rmbb}
\newcommand{\Z}{\Zmbb}
\newcommand{\lit}[1]{`$#1$'}
\newcommand{\Ae}{A^{*}}
\newcommand{\Ap}{A_{p}}
\newcommand{\Ape}{\Ap^{*}}

\newcommand{\Ab}{A_{\beta}}

\newcommand{\lexp}[1]{\vphantom{a}^{#1}}
\newcommand{\lom}[1]{\lexp{\omega}#1\xmd}
\newcommand{\lomA}{\lom{\!\!A}}
\newcommand{\lomz}{\lom{0}}
\newcommand{\Lp}{L_{p}}
\newcommand{\Lq}{L_q}
\newcommand{\Lb}{L_{\beta}}
\newcommand{\LG}{L_{G}}
\newcommand{\Tp}{\Tc_{p}}

\newcommand{\Tcp}{\Tp}
\newcommand{\Tcpl}{\Tp^{(\ell)}}
\newcommand{\TL}{\Tc_{L}}
\newcommand{\TcL}{\TL}
\newcommand{\TcLl}{\TL^{(\ell)}}
\newcommand{\KcG}{\Kc_{G}}
\newcommand{\KcL}{\Kc_{L}}
\newcommand{\Gb}{G_{\beta}}
\newcommand{\pce}{\textsc{pce}\xspace}
\newcommand{\dev}{\textsc{dev}\xspace}
\newcommand{\adev}{\textsc{adev}\xspace}
\DeclareMathOperator{\CaRd}{\mathsf{card}}
\newcommand{\jsCard}[1]{\CaRd\left(#1\right)}
\DeclareMathOperator{\Charsym}{\mathbf{\chi}}
\newcommand{\Charf}[1]{\Charsym_{#1}}
\newcommand{\Char}[2]{\Charf{#1}\!\left(#2\right)}
\DeclareMathOperator{\rdiffop}{\Delta}
\newcommand{\rdiff}[1]{\rdiffop(#1)}
\DeclareMathOperator{\distop}{\mathrm{d}}
\newcommand{\dist}[1]{\distop(#1)}
\DeclareMathOperator{\ecarop}{\mathrm{e}}
\newcommand{\ecar}[1]{\ecarop(#1)}
\newcommand{\cyli}[1]{[#1]}
\newcommand{\cyliw}{\cyli{w}}
\newcommand{\lex}{\preccurlyeq}
\newcommand{\rad}{\sqsubseteq}
\newcommand{\lcp}{\wedge}
\newcommand{\Pre}[1]{\mathsf{Pre}(#1)}
\newcommand{\lght}[1]{|#1|}
\newcommand{\Inte}[1]{\left\lfloor#1\right\rfloor}
\newcommand{\sequ}[1]{\big(\!#1\big)_{\ell\in\N}}
\newcommand{\prew}[1]{w_{[#1]}}
\renewcommand{\ShiftInd}[1]{\raisebox{-0.4ex}{$\scriptstyle{#1}$}}
\newcommand{\repr}[2][L]{\langle#2\rangle_{\ShiftInd{#1}}}
\newcommand{\reprp}[1]{\repr[p]{#1}}
\newcommand{\reprpq}[1]{\repr[\pq]{#1}}
\newcommand{\reprG}[1]{\repr[G]{#1}}
\DeclareMathOperator{\valop}{\pi}
\newcommand{\val}[2][L]{\valop_{#1}\!\left(#2\right)}
\renewcommand{\auta}{\jsAutUn{A,Q,I,E,T}}
\newcommand{\ImuT}{\jsAutUn{I,\mu,T}}
\DeclareMathOperator{\genop}{\mathsf{g}}
\newcommand{\gen}[2][L]{\genop_{#1}(#2)}
\newcommand{\genz}[1][z]{\gen{#1}}
\newcommand{\adj}[1][\Ac]{M_{#1}}
\DeclareMathOperator{\polop}{\mathsf{P}}
\newcommand{\pol}[1][L]{\polop_{#1}}
\DeclareMathOperator{\degop}{\mathsf{deg}}
\renewcommand{\deg}[1]{\degop{#1}}
\newcommand{\QPz}{\Q[z]}
\DeclareMathOperator{\Loneop}{L^{1}}
\newcommand{\Lone}[1]{\Loneop\!\left(#1\right)}

\DeclareMathOperator{\Succop}{\mathsf{Succ}}
\newcommand{\Succ}[2][L]{\Succop_{#1}\!\left(#2\right)}

\newcommand{\Succf}[1][L]{\Succop_{#1}}
\newcommand{\Succp}[1]{\Succ[p]{#1}}
\renewcommand{\succ}[2][L]{\Succ[#1]{#2}}   
\newcommand{\succf}[1][L]{\Succop_{#1}}
\newcommand{\succw}{\Succ{w}}               
\DeclareMathOperator{\odoop}{\tau}
\newcommand{\odo}[2][L]{\odoop_{\!\ShiftInd{#1}}\!\left(#2\right)}
\newcommand{\odof}[1][L]{\odoop_{\!\ShiftInd{#1}}}
\newcommand{\odofG}{\odof[G]}

\newcommand{\odoG}[1]{\odo[G]{#1}}
\DeclareMathOperator{\uop}{\mathbf{u}}
\newcommand{\bfu}[2][L]{\uop_{#1}(#2)}
\newcommand{\uL}[1]{\bfu{#1}}
\newcommand{\uLl}{\uL{\ell}}

\newcommand{\uLp}[1]{\bfu[\Lp]{#1}}
\newcommand{\up}[1]{\bfu[p]{#1}}
\newcommand{\wlg}[2][L]{\uop_{#1}(#2)} 
\newcommand{\wlgl}{\uLl}               
\DeclareMathOperator{\vop}{\mathbf{v}}
\newcommand{\bfv}[2][L]{\vop_{#1}(#2)}
\newcommand{\vL}[1]{\bfv{#1}}
\newcommand{\vLp}[1]{\bfv[\Lp]{#1}}
\newcommand{\bfvp}[1]{\bfv[p]{#1}}
\newcommand{\vLl}{\vL{\ell}}

\DeclareMathOperator{\bfwop}{\mathbf{w}}
\newcommand{\bfw}[2][L]{\bfwop_{#1}(#2)}
\DeclareMathOperator{\zop}{\mathbf{z}}
\newcommand{\bfz}[2][q]{\zop_{#1}(#2)}
\newcommand{\wlgsl}{\vLl}               
\DeclareMathOperator{\wop}{\mathrm{w}}
\newcommand{\wea}[1]{\wop(#1)}
\DeclareMathOperator{\cpop}{\mathsf{cp}}
\newcommand{\capr}[2][L]{\cpop_{\ShiftInd{#1}}(#2)} 
\newcommand{\cpf}[1][L]{\cpop_{\ShiftInd{#1}}} 
\newcommand{\cpL}[1]{\capr{#1}}
\newcommand{\cpG}[1]{\capr[G]{#1}}
\newcommand{\cpfG}[1][L]{\cpf[G]} 

\DeclareMathOperator{\scpop}{\mathsf{scp}}
\newcommand{\scp}[2][L]{\scpop_{#1}(#2)} 
\newcommand{\scpL}[1]{\scp{#1}}
\newcommand{\CP}[1][L]{\mathsf{CP}_{#1}} 
\newcommand{\CPL}{\CP}
\newcommand{\CPG}{\CP[G]}
\newcommand{\CAPR}[1][L]{\mathsf{CP}_{#1}} 

\DeclareMathOperator{\FCPop}{\mathsf{FCP}}
\newcommand{\FCP}[1][L]{\FCPop_{#1}}
\newcommand{\FCPL}{\FCP}
\newcommand{\LFCAPR}[1][L]{\FCPop_{#1}}        
\newcommand{\muG}{\mu_{\ShiftInd{G}}}
\newcommand{\mub}{\mu_{\ShiftInd{\beta}}}

\newcommand{\mup}{\mu_{\ShiftInd{\psi}}}
\DeclareMathOperator{\LBop}{\mathsf{LB}}
\DeclareMathOperator{\RBop}{\mathsf{RB}}
\newcommand{\LB}[2][L]{\LBop_{#1}\left(#2\right)}
\newcommand{\RB}[2][L]{\RBop_{#1}\left(#2\right)}
\DeclareMathOperator{\Maxlgop}{\mathsf{Maxlg}}
\newcommand{\Maxlg}[1][L]{\Maxlgop(#1)}
\newcommand{\rhth}[1][{r}]{{\mathbf{#1}}}
\newcommand{\rhthtp}{(r_0,r_1,\ldots,r_{\iqmu})}

\newcommand{\sigs}[1][s]{{\boldsymbol{#1}}}

\newcommand{\igrr}[1][L]{\gamma_{#1}}
\newcommand{\dmrm}{\mathrm{d}}
\newcommand{\bge}[1][\beta]{\dmrm_{#1}}
\newcommand{\qge}[1][\beta]{\bge[#1]^{*}(1)}

\newcommand{\THFracs}[2]{\frac{#1}{#2}}%
\newcommand{\pq}{\THFracs{p}{q}}
\newcommand{\td}{\THFracs{3}{2}}
\newcommand{\Indpq}[1]{#1_{\pq}}%
\newcommand{\Indtd}[1]{#1_{\td}}%
\newcommand{\Lpq}{\Indpq{L}}
\newcommand{\Tpq}{\Indpq{\Tc}}
\newcommand{\Ttd}{\Indtd{\Tc}}

\newcommand{\frct}[1]{\{#1\}}
\newcommand{\ntgr}[1]{\left\lfloor#1\right\rfloor}
\newcommand{\degh}[2][G]{\partial_{\ShiftInd{#1}}(#2)}
\newcommand{\deghN}{\degh{N}}
\newcommand{\quot}[2]{#1\div#2}
\newcommand{\rest}[2]{#1\!\mathbin{\texttt{\%}}#2} 
\newcommand{\rfog}{\mathbin{>_{\mathsf{rf}}}}

\newcommand{\meanE}[3]
   {\frac{1}{#2}\sum_{#3=0}^{#2-1}#1\left(\tau^{#3}(0)\right)}
\newcommand{\meanEi}[2]{\meanE{#1}{#2}{i}}

\newcommand{\sumTE}[3]{\sum_{#3=0}^{#2-1}\!#1(#3)\xmd}
\newcommand{\sumTEi}[2]{\sumTE{#1}{#2}{i}}
\newcommand{\sumTEj}[2]{\sumTE{#1}{#2}{j}}

\newcommand{\Digmu}[1]{#1\psscalebox{.9}{-1}}
\newcommand{\Digmui}[1]{#1\psscalebox{.9}{{\scriptstyle -1}}}

\newcommand{\pmu}{\Digmu{p}}

\newcommand{\rmu}{\Digmu{r}}

\newcommand{\iqmu}{\Digmui{q}}
\newcommand{\Digpu}[1]{#1\psscalebox{.9}{+1}}

\newcommand{\dpu}{\Digpu{d}}
\newlength{\lga}\newlength{\lgb}
\newcounter{nodelbl}

\ifdraft
\linenumbers

\fi
\begin{document}
\JSLandscape
\title{%
\ifelsevier
\else
\vspace*{-2cm}
\fi
The carry propagation of the successor function}

\author{{Val\'erie Berth\'e}
\thanks{IRIF, CNRS/Universit\'e de Paris}
\and
{Christiane Frougny}
\footnotemark[1]
\and
{Michel Rigo}
\thanks{Universit\'e de Li\`ege, D\'epartement de Math\'ematiques}
\and
{Jacques Sakarovitch}
\thanks{IRIF, CNRS/Universit\'e de Paris 
        and
		LTCI, Telecom, Institut Polytechnique de Paris}
\footnotemark[4]
}
\maketitle
\renewcommand{\thefootnote}{\fnsymbol{footnote}}
\footnotetext[4]{Corresponding author}
\renewcommand{\thefootnote}{\arabic{footnote}}


\begin{abstract}
Given any numeration system, we call \emph{carry propagation} at a
number~$N$ the number of digits that are changed when going from the
representation of~$N$ to the one of~$N+1$, and \emph{amortized carry
propagation} the limit of the mean of the carry propagations at the
first~$N$ integers, when~$N$ tends to infinity, if this limit exists.

In the case of the usual base~$p$ numeration system, it can be shown 
that the limit indeed exists and is equal to $p/(p-1)$.
We recover a similar value for those numeration systems we consider 
and for which the limit exists.

We address the problem of the existence of the amortized carry
propagation in non-standard numeration systems of various kinds:
abstract numeration systems, rational base numeration systems, greedy
numeration systems and beta-numeration.
We tackle the problem with three different types of techniques:
combinatorial, algebraic, and ergodic.
For each kind of numeration systems that we consider, the relevant
method allows for establishing sufficient conditions for the existence
of the carry propagation and examples show that these conditions are
close to being necessary conditions.
\end{abstract}

\vspace*{1cm}

\ifelsevier
\else
\renewcommand{\secti}[1]{Section~\ref{s.#1}}
\fi

\section{Introduction}

The \emph{carry propagation} is a nightmare for schoolchildren and a
headache for computer engineers: not only could the addition of two 
digits produce a carry, but this carry itself, when added to the 
next digit on the left\footnote{%
   We write numbers under MSDF (Most Significant Digit First) 
   convention.} 
could give rise to another carry, and so on, 
\NeM{cp~length, cp~time, one has to choose}%
and this may happen arbitrarily many times.
Since the beginnings of computer science, the evaluation of the carry
propagation length has been the subject of many works and it is known
that the average carry propagation length for the addition of two
uniformly distributed $n$-digit binary numbers is: 
$\msp\log_2(n)+\grando{1}\msp$ (see \cite{BurkEtAl47,Knut78,Pipp02}).

Many published works address the design of numeration systems in which
the carry does not indeed propagate --- through the use of
supplementary digits --- which allow the design of circuits where
addition is performed `in parallel' for numbers of large, but fixed,
length \cite{Aviz61, ChowRobe78}.

We consider here the problem of carry propagation from a more 
theoretical perspective and in an seemingly elementary case.
We investigate the amortized carry propagation of the 
\emph{successor function} in various numeration systems.
The central case of integer base numeration system 
is a clear example of the issue.
Let us take an integer~$p$ greater than~$1$ as a base.
In the representations of the succession of the integers --- which is
exactly what the successor function achieves --- the least digit changes
at every step, the penultimate digit changes every~$p$ steps, the
ante-penultimate digit changes every~$p^2$ steps, and so on.
Consequently, the average carry propagation of the successor
function, computed over the first~$N$ integers, should tend to the 
quantity: 
\begin{equation}
1 + \frac{1}{p} + \frac{1}{p^{2}} +\frac{1}{p^{3}} + \cdots =
\frac{p}{p-1} 
\eqvrg
\label{q.int-cp-int}
\end{equation} 
when~$N$ tends to infinity.
It can be shown that it is indeed the case.
Following on our previous works on various non-standard numeration
systems, we investigate here the questions of evaluating and computing
the amortized carry propagation in those systems.
We thus consider several such numeration systems which are different
from the classical integer base numeration systems: the greedy
numeration systems and the $\beta$-numeration systems
(see~\cite{FrouSaka10}) which are a specific case of the former, the
rational base numeration systems (introduced in~\cite{AkiyEtAl08})
which are not greedy numeration systems, and the abstract numeration
systems (defined in~\cite{LecoRigo10}) which are a generalization of
the classical \emph{positional numeration systems}.

In~\cite{BertEtAl07}, we already reported that the approach of
\emph{abstract numeration systems} of~\cite{LecoRigo01}, namely the
study of a numeration system \via the properties of the set of
expansions of the natural integers is appropriate for this problem.
Such systems consist of a totally ordered alphabet~$A$ --- hence, 
without loss of generality, an initial section 
$\msp\{0,1,\ldots,p-1\}\msp$
of the non-negative integers~$\N$ --- and a language~$L$ of~$\Ae$, 
ordered by the \emph{radix order} deduced from the ordering on~$A$.
The representation of an integer~$n$ is then the $(n+1)$-th 
word\footnote{%
   The `$+1$' gives room for the representation of~$0$ by the first 
   word of~$L$.} 
of~$L$ in the radix order.
This definition is consistent with every classical standard and
non-standard numeration system, that is, the representation of~$n$ in
such a system is the $(n+1)$-th word (in the radix order) of the set
of representations of all integers in the system.

Given a numeration system defined by a language~$L$ ordered by radix
order, we denote by~$\cpL{i}$ the carry propagation in the computation
from the representation of~$i$ in~$L$ to that of~$i+1$.
The \emph{(amortized) carry propagation} of~$L$, which we
denote by~$\CPL$, is the limit, \emph{if it exists}, of the mean
of the carry propagation at the first~$N$ words of~$L$: 
\begin{equation}
\CPL = \lim_{N \to \infty} \frac{1}{N}\sum_{i=0}^{N-1}\cpL{i}
\eqpnt
\label{q.cp-int}
\end{equation}
This quantity, introduced by Barcucci, Pinzani and Poneti
in~\cite{BarcEtAl05}, is the main object of study of the present paper
whose aim is to investigate cases where the carry propagation exists
or not, and suggests ways to compute~it.

A common further hypothesis is to consider \emph{prefix-closed} and
\emph{right-extendable} languages, called `\pce' languages in the
sequel: every left-factor of a word of~$L$ is a word of~$L$ and every
word of~$L$ is a left-factor of a longer word of~$L$.
Hence, $L$ is the branch language of an infinite labeled tree~$\TL$ 
and, once again, every classical standard and non-standard numeration 
system meets that hypothesis.

We move on to prove two simple properties of the carry propagation of
\pce languages.
First, $\CPL$ does not depend upon the labeling of~$\TL$, but only on
its `shape' which is completely defined by the infinite sequence of
the degrees of the nodes visited in a breadth-first traversal
of~$\TL$, and which we call the \emph{signature} of~$\TL$ (or
of~$L$)~\cite{MarsSaka17b}.
For instance, the signature of the language of the representations of 
the integers in base~$p$ is the constant sequence~$p^{\omega}$.
Next, let us denote by~$\uLl$ the number of words of~$L$ of 
length~$\ell$. 
We call the limit, \emph{if it exists}, of the ratio
$\msp\uL{\ell+1}/\uLl\msp$ the \emph{local growth rate} of a
language~$L$, and we denote it by~$\igrr$.
And we show (\corol{car-pro-fil}) that \emph{if~$\CPL$ exists,
then~$\igrr$ exists} and
\begin{equation}
    \CPL=\frac{\igrr}{\igrr-1}
    \label{q.agg-int}
\end{equation}
holds, which is an obvious generalization of~\equnm{int-cp-int}.
On the other hand, an example shows that~$\igrr$ may exist 
while~$\CPL$ does not (\examp{un-bal}).

By virtue of~\equnm{agg-int}, the \emph{computation} of~$\CPL$ is 
usually not an issue.
The problem lies in proving its \emph{existence}.
We develop \emph{three different methods} for the proofs of existence,
whose domains of application are pairwise incomparable, that is to say, we
have examples of numeration systems for which the existence of~$\CPL$
is established by one method and not by the other two.
These methods: \emph{combinatorial}, \emph{algebraic}, and
\emph{ergodic}, are built upon very different mathematical
backgrounds.

\ifelsevier\else\medskip\fi

We first show by a combinatorial method that
languages with an \emph{eventually periodic signature} have a carry
propagation (\theor{per-sig}).
It is known that these languages are essentially the rational base
numeration systems, possibly with non-canonical alphabets of
digits~\cite{MarsSaka17b}.

We next consider the \emph{rational abstract numeration systems}, that 
is, those systems which are defined by languages accepted by 
\emph{finite automata}.\footnote{%
   In this context where we deal with both languages and formal power 
   series, we say \emph{rational} rather than \emph{regular} for 
   languages accepted by finite automata.}
Examples of such systems are the Fibonacci numeration system and, more
generally, $\beta$-numeration systems where $\beta$ is a \emph{Parry
number} \cite{FrouSaka10}, and of course many other systems which
greatly differ from $\beta$-numeration.
\theor{cp-rat-lan} states that \emph{if a rational \pce language~$L$
has a local growth rate, and if all its quotients also have a local
growth rate, then~$L$ therefore has a carry propagation}.
The proof is based on a property of rational power series with 
positive coefficients which is reminiscent of the Perron-Frobenius 
Theorem.
A tighter sufficient condition may even be established
(\theor{a-dev-cp}) but a remarkable fact is that the existence of the
local growth rate \emph{is not} a sufficient condition for the
existence of carry propagation \emph{even for rational} \pce languages
(\examp{ctr-exp-js}).

The definition of carry propagation by \equat{cp-int} inevitably 
brings to mind the Ergodic Theorem.
Finally, we consider the so-called \emph{greedy numeration
systems} \cite{Frae85} --- $\beta$-numeration systems, with
any~$\beta>1$, are one example but they can be much more general.
The language of greedy expansions in such a system is embedded into a
compact set, and the successor function is extended as an action,
called the odometer, on that compactification.
In this setting, the odometer is just the addition of~$1$.
This gives a dynamical system, introduced
in~\cite{GrabEtAl95,BaraEtAl02}.
Tools from ergodic theory developed only recently
(in~\cite{BaraGrab16}) allow us to prove the existence of the carry
propagation for greedy systems with exponential growth
(\theor{cp-erg}), and thus for $\beta$-numeration in general.
The difficulty is that the odometer is not continuous in general
and the Ergodic Theorem does not directly apply.

\ifelsevier\else\bigskip\fi

The substential length of the paper is due to the fact that it borrows
results from different chapters of mathematics (in relation with
formal language theory) which we had to present as we wished the paper
to be as self-contained as possible.
It is organized as follows.

In \secti{pre-not}, after reviewing some definitions on words, we 
present the notion of \emph{abstract numeration systems}.
\secti{car-pro-num} is devoted to the combinatorial point of view.
Here we more precisely define the notion of carry propagation, present
its relationship with the local growth rate, and give, as mentioned
above, a first example of a language with local growth rate which does
not have a carry propagation.
We then define the signature of a language and establish the
aforementioned quoted result for languages with eventually periodic
signature.
Note that neither the algebraic nor the ergodic methods apply to 
these languages (except of course for the integer-base numeration 
systems).

In \secti{car-pro-rat}, we study the carry propagation of 
\emph{rational} abstract numeration systems by means of algebraic 
methods.
We first recall the definitions of \emph{generating function}, of
\emph{modulus} of a language, of \emph{languages with dominating
eigenvalue} (\dev languages), and give the description, due to
Berstel, of the `leading terms' of generating functions of rational
languages.
We are then able to introduce the notion of \emph{languages with almost
dominating eigenvalue}s (\adev languages) and to show that it is a
necessary and sufficient condition for a rational language to have a
local growth rate (\theor{lgr-a-dev}).
As already said, it is not a sufficient condition for the existence 
of the carry propagation.
But the counter-example directly leads to a sufficient condition for a
rational language to have a carry propagation (\theor{a-dev-cp}).

\secti{erg-poi-vie} is devoted to the study of the question of the 
carry propagation of a language by means of tools from ergodic theory.
Even though it seems to be quite a natural approach, it requires some 
elaborate new results and is, so far, applicable to the family of 
greedy numeration systems only.
We first recall Birkhoff's Ergodic Theorem and 
follow~\cite{GrabEtAl95} for the description of 
a framework 
in which we can turn a numeration system and its successor function 
into a dynamical system.
We then focus on greedy numeration systems that have been studied by 
Barat and Grabner \cite{BaraGrab16}.
The carry propagation in these systems is not the uniform limit of its
truncated 
approximations but the properties of greedy numeration systems allow
us to establish that it is regular enough to be in the scope of the
Ergodic Theorem.
We end with some examples of $\beta$-numeration systems which are at the 
crossroads of algebraic and ergodic methods, thus allowing two 
different ways for the computation of the carry propagation.

\ifelsevier\else\bigskip\fi

It should be noted that the inspiration for this current work was
initiated by a paper where the \emph{amortized algorithmic complexity}
of the successor function for some $\beta$-numeration systems was
studied \cite{BarcEtAl05}.
Whatever the chosen computation model, the (amortized) algorithmic
complexity, that is, the limit of the mean of the \emph{number of
operations} necessary to compute the successor of the first~$N$
integers, is greater than the (amortized) carry propagation, hence can
be seen as the sum of two quantities: the carry propagation itself and
an \emph{overload}.
The study of carry propagation leads to quite unexpected and winding
developments that form a subject on its own and that we present here.
But this paper is only the first step in solving the original problem 
which consists in describing the \emph{complexity} 
of the successor function.

Addressing complexity implies the definition of a computational model 
and ours is based on the use of sequential transducers.
This explains the particular attention we pay in this paper to
\emph{rational} abstract numeration systems.
The sequel of this work~\cite{CCSF-Part2} is in the preparation phase
and will hopefully be completed in a not too distant future.

\section*{Acknowledgments}

We are grateful to Christophe Reutenauer for helpful discussions on
the proof of \theor{cp-rat-lan}, and for pointing us to results from
his treatise on rational series~\cite{BersReut11}.
We would like to thank Peter Grabner who drew our attention to the
ergodic nature of the notion of carry propagation and advised us on
using the results of his recent work~\cite{BaraGrab16}.
We are also grateful to our colleagues David Madore and Hugues 
Randriam for the numerous helpful and fruitful discussions we had 
with them during the long development process of this paper.
We finally thank the referee for his precise reading and his many 
corrections in order to improve the paper quality.

We are also pleased to acknowledge support from the French Agence
Nationale de la Recherche through the ANR projects {\tt DynA3S}
(ANR-13-BS02-0003) and {\tt CODYS} (ANR-18-CE40-0007).

\ifelsevier
\else
\renewcommand{\secti}[1]{Sec.~\ref{s.#1}}
\fi



\section{Preliminary notions}
\label{s.pre-not}

We review more or less classical basic notions on languages that we 
will use throughout this work.
More specific notions and notation will be introduced at the point 
they are needed, even when classical.

\subsection{Words on ordered alphabets}

In this paper, $A$ denotes a \emph{totally ordered} finite alphabet, 
and the order is denoted by~$<$.
Without loss of generality, we can always assume that~$A$ consists of
consecutive integers starting with~$0$ and naturally ordered:
$\msp A=\{0,1,\ldots,\rmu\}\msp$.
The set of all words over~$A$ is denoted by~$\Ae$.  
The empty word is denoted by~$\varepsilon$.  
The length of a word $w$ of $\Ae$ is denoted by~$|w|$.
The set of words of length less than or equal to~$n$ is denoted 
by~$A^{\le n}$.

If $w=u\xmd v$, $u$ is a \emph{prefix} (or a \emph{left-factor})
of~$w$, \emph{strict prefix} if~$v$ is non-empty, 
and~$v$ is a \emph{suffix} (or a \emph{right-factor}) of~$w$, 
\emph{strict suffix} if~$u$ is non-empty. 
The set of prefixes of~$w$ is denoted by~$\Pre{w}$.

The \emph{lexicographic order}, denoted by~$\lex$, extends the order 
on~$A$ onto~$\Ae$ and is defined as follows.
Let $v$ and $w$ be two words in $\Ae$ and~$u$ their longest common 
prefix.
Then, $\msp v\lex w\msp$ if~$v=w$ or, if $v=u\xmd a\xmd s$, 
$w=u\xmd b\xmd t$  with~$a$ and~$b$ in~$A$, and~$a<b$.
The \emph{radix order} (also called the \emph{genealogical order} or
the \emph{short-lex order}), denoted by~$\rad$, is defined as follows:
\NeM{Is notation $\rad$ really necessary?  not used but once, I think.
JS 190504}%
$\msp v\rad w\msp$ if~$|v|<|w|$ or~$|v|=|w|$ and $v\lex w$
(\ie for two words of same length, the radix order
coincides with the lexicographic order).
In contrast with lexicographic order, radix order is a
\emph{well-order}, that is, every non-empty subset has a minimal
element.
For instance, the set $\msp a^+b=\Defi{a^nb}{n>0}\msp$ has no minimal 
element for the lexicographic order. 

\subsection{Languages and Abstract Numeration Systems}
\label{s.lan-ans}%

In all what follows, $L$ denotes a \emph{language over~$A$}, that is, 
any subset of~$\Ae$.
A language $L$ is said to be \emph{prefix-closed} if every prefix of a
word of~$L$ is in~$L$.
A language~$L$ is said to be \emph{(right) extendable} if every word 
of~$L$ is a strict prefix of another word of~$L$.

\begin{definition}
A language~$L$ is called a \emph{\pce language} if it is both
\emph{prefix-closed} and \emph{right extendable}.
\end{definition}

\begin{definition}
Every infinite language~$L$ over~$A$ is totally ordered by the radix
order on~$\Ae$.
The \emph{successor of a word}~$w$ of~$L$ is the least of all words
of~$L$ greater than~$w$, a well-defined word since radix order is a
well-order, and denoted by~$\succw$.
\end{definition}

Hence~$\succf$ is a map from~$\Ae$ into itself, whose domain is~$L$ 
and image is~$L\setminus\{w_{0}\}$, where~$w_{0}$ is the least word 
in~$L$ for the radix order. 

Languages over totally ordered alphabets have been called 
\emph{Abstract Numeration Systems} (ANS for short) and studied, for 
instance, in~\cite{LecoRigo01} or~\cite{LecoRigo10}.\footnote{%
   To tell the truth, ANS are supposed to be \emph{rational} (or 
   regular) languages in these references~\cite{LecoRigo01,LecoRigo10}.
   Although it will be met in most instances in this work, this 
   hypothesis of being rational is not necessary for the basic 
   definitions in ANS and we indeed also consider ANS which are 
   \emph{not} rational.}
Of course, such a language~$L$ can be totally ordered:
$w_{0}$ is the least word of~$L$, 
$w_{1}$ is the least word of~$L\bk\{w_{0}\}$, 
$w_{2}$ the least word of~$L\bk\{w_{0},w_{1}\}$, 
$w_{i+1}$ the least word of~$L\bk\{w_{0},w_{1},\ldots,w_{i}\}$, 
and so on:
\begin{equation}
L=\{w_{0}\rad w_{1}\rad w_{2}\rad \cdots \rad w_{i} \rad \cdots \}
\eqpnt
\notag
\end{equation}
\emph{By definition}, $w_{i}$, the $(i+1)$-th word of~$L$ in that
enumeration, is \emph{the $L$-repre\-sen\-ta\-tion} of the integer~$i$ and
is denoted by~$\repr{i}$ --- hence~$w_{0}$ is the representation 
of~$0$.
Conversely, we let~$\val{w}$ denote the integer represented by the 
word~$w$ of~$L$:
$\repr{\val{w}}=w$.
In this setting, the successor function behaves as expected, that is, 
for every non-negative integer~$i$,
\begin{equation}
	\succ{\repr{i}}=\repr{i+1}
	\eqpnt
	\notag
\end{equation}

The notion of ANS is consistent with that of \emph{positional}
numeration systems in the sense that the language of representations
of integers in such systems, seen as an ANS, gives the same
representation for every integer.

\begin{example}
\label{e.int-bas}%
\textbf{The integer base numeration systems.}
Let~$p$ be an integer, $p>1$, taken as a base.
We write~$\repr[p]{n}$ for the representation of~$n$ in base~$p$ (the 
$p$-representation of~$n$).
Let~$\msp\Ap=\{0,1,\ldots,\pmu\}\msp$ be the alphabet of digits used 
to write integers in base~$p$
and~$\msp\Lp=\{\epsilon\}\xmd\cup\xmd\{1,\ldots,\pmu\}\Ape\msp$
the set of $p$-representations of the integers.\footnote{%
   For consistency with the whole theory we present here, the 
   integer~$0$ is represented by~$\epsilon$ even though, in reality, its 
   $p$-representation is~`$0$'.}
The consistency claimed above reads:
\begin{equation}
    \fa n\in\N\quantsp
    \reprp{n}=\repr[\Lp]{n}
    \eqpnt
    \eee
    \notag
\end{equation}
\end{example}
 
Other examples such as \emph{rational base numeration systems} are
presented in \examp{rat-bas} and \emph{greedy numeration systems} in
\secti{gre-num-sys}.

\subsection{The language tree}
\label{s.lan-tre}%

A prefix-closed language~$L$ of~$\Ae$ is the \emph{branch language} 
of a labeled tree~$\TL$, that we call the \emph{language tree} 
of~$L$.

The nodes of~$\TL$ are indifferently seen as labeled by the words 
of~$L$ or by the non-negative integers:
the root of~$\TL$ is associated with~$\epsilon$ and 
with~$0=\repr{\epsilon}$;
a node labeled by~$w$ (and by~$n=\repr{w}$) has as many children as 
there are letters~$a_{1}$, $a_{2}$, \ldots, $a_{k}$ in~$A$ such 
that~$w\xmd a_{1}$, $w\xmd a_{2}$, \ldots, $w\xmd a_{k}$ are words 
in~$L$ and the edge between the node~$w$ (or~$n$) and the 
node~$w\xmd a_{i}$ (or~$m=\repr{w\xmd a_{i}}$) is labeled 
by~$a_{i}$. 
It follows that the tree~$\TL$ is naturally an \emph{ordered tree} in
the sense that the children~$w\xmd a_{1}$, $w\xmd a_{2}$, \ldots, 
$w\xmd a_{k}$  are ordered by~$a_{1}<a_{2}< \cdots<a_{k}$. 

The \emph{breadth-first traversal} of the ordered tree~$\Tc_{L}$ 
amounts to enumerating the words of~$L$ in the radix order. 
We come back to this fact in \secti{sig-car-pro}.

If~$L$ is (right) extendable, then~$\TL$ has no leaf and every branch 
of~$\TL$ is infinite.

\begin{example}[\examp{int-bas} continued]
\label{e.int-bas-1}%
The first nodes of the tree~$\Tc_{\Lp}$, that we rather write~$\Tp$,
are represented in \figur{lan-tre}\dex{a} for the case~$p=3$.
\end{example}

\begin{example}
\label{e.Fib-1}%
\textbf{The Fibonacci numeration system.} 
The Fibonacci numeration
system is a positional numeration system based on the sequence of
Fibonacci numbers, that is, the linear recurrence sequence
$(F_n)_{n\ge0}$ where~$F_0=1$, $F_1=2$ and $F_{n+2}=F_{n+1}+F_n$ for
all~$n\ge 0$.
The set of representations of the natural integers in that system is 
\NeM{Attn, j'ai changŽ la notation pour le langage, de $F$ ˆ $L_{F}$. 
JS 190507}%
known to be the set of words of~$\{0,1\}^*$ that do not contain two 
consecutive~$1$'s, that is, 
$L_{F}=\{\varepsilon\}\cup 1\{0,1\}^*\setminus \{0,1\}^*11\{0,1\}^*$
or simply $L_{F}=\{\varepsilon\}\cup 1\{0,01\}^*$.
The first nodes of the language tree~$\Tc_{F}$ are represented in  
\figur{lan-tre}\dex{b}.
For a general reference on non-standard numeration systems, see \eg
\cite{FrouSaka10}.
\end{example}

\begin{example}
\label{e.Fi-na}%
\textbf{The Fina numeration system.}
\NeM{Mauvaise idŽe (de moi je crois) de noter $E$ pour la suite.
JS 190504}%
The sequence of Fibonacci numbers of \emph{even rank} is also a
linear recurrence sequence, defined by~$E_0=1$, $E_1=3$
and $E_{n+2}=3\xmd E_{n+1}-E_n$ for all $n\ge 0$.
The positional numeration system based on the
sequence~$(E_n)_{n\ge0}$, which we call \emph{Fina}, is known to give
the integers representations that are the words of~$\{0,1,2\}^*$ which
do not contain factors in the language~$2\xmd1^{*}2$.
The first nodes of the language tree~$\Tc_{E}$ of~$E$ are represented in 
\figur{lan-tre}\dex{c}.
\end{example}

\ifelsevier
\else
\subsection{Automata}
\label{s.aut-oma}%

We essentially follow the definitions and notation of
\cite{FrouSaka10,Saka09} for automata.

An \emph{automaton over~$A$}, $\Ac=\auta$, is a directed graph with
edges labeled by elements of~$A$.
The set of vertices, traditionally called \emph{states}, is denoted
by~$Q$, $I \subset Q$ is the set of \emph{initial} states, 
$T \subset Q$ is the set of \emph{terminal} states and 
$E \subset Q \times A\times Q$ is the set of labeled \emph{edges}.
If $(p,a,q) \in E$, we write $p\stackrel{a}{\to}q$.
The automaton is \emph{finite} if $Q$ is finite.
The automaton~$\Ac$ is \emph{deterministic} if $E$ is the graph
of a (partial) function from $Q \times A$ to $Q$, and if there is a
unique initial state.
It is \emph{trim} if every state is \emph{accessible} and
\emph{co-accessible}.
\fi

\setcounter{nodelbl}{0}
\begin{figure}[htbp]
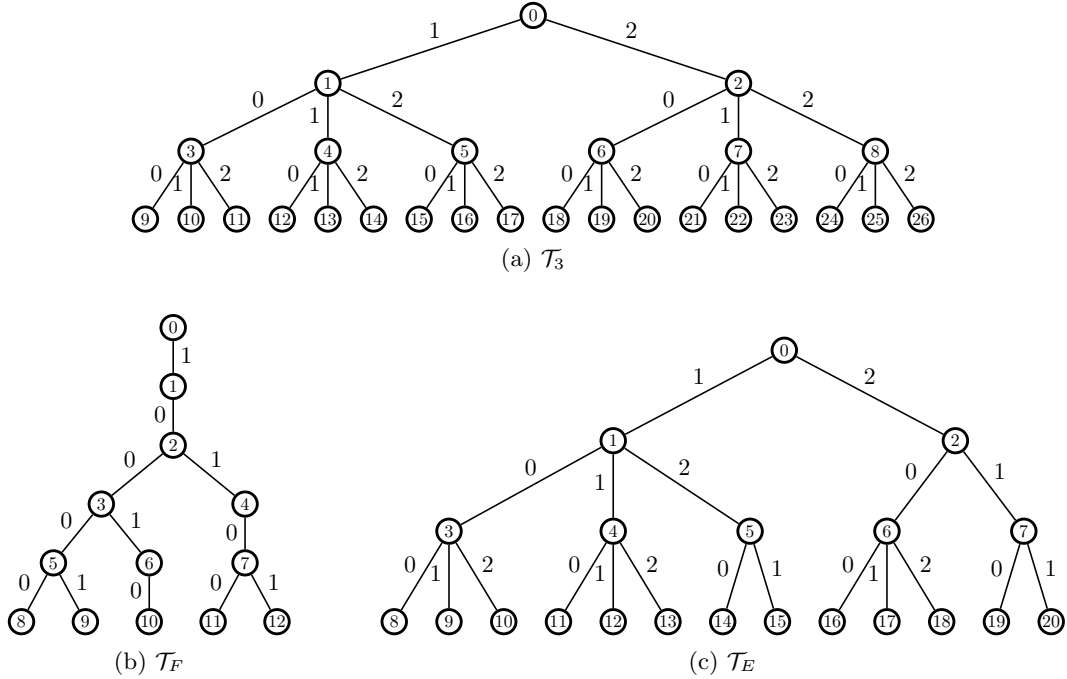

\centering%
\subfigure[$\Tc_{3}$]{%
\setlength{\lga}{1cm}%
\setlength{\lgb}{1.5cm}%
\VCDraw{%
\begin{VCPicture}{(0,0)(17\lga,4.2)}
\SmallState
\ChgStateLabelScale{.6}
\VCPut{(0,0)}{\multido{\i=0+1}{18}{\FPadd{\i}{9}{\nodelbl}%
                                   \State[\nodelbl]{(\i\lga,0)}{A\i}}}%
\VCPut{(0,\lgb)}{\multido{\i=0+1}{6}{\FPadd{\i}{\i}{\jj}%
                                     \FPadd{\jj}{\i}{\jj}%
				     \FPadd{\jj}{1}{\jj}%
                                     \FPadd{\i}{3}{\nodelbl}%
				     \State[\nodelbl]{(\jj\lga,0)}{B\i}}}%
\VCPut{(0,2\lgb)}{\State[1]{(4\lga,0)}{C0}\State[2]{(13\lga,0)}{C1}}%
\VCPut{(0,3\lgb)}{\State[0]{(8.5\lga,0)}{D0}}%
\ChgEdgeArrowStyle{-}
\ChgEdgeLabelScale{.8}
\EdgeR{D0}{C0}{1}
\EdgeL{D0}{C1}{2}
\EdgeR{C0}{B0}{0}
\EdgeR{C0}{B1}{1}
\EdgeL{C0}{B2}{2}
\EdgeR{C1}{B3}{0}
\EdgeR{C1}{B4}{1}
\EdgeL{C1}{B5}{2}
\EdgeR{B0}{A0}{0}
\EdgeR{B0}{A1}{1}
\EdgeL{B0}{A2}{2}
\EdgeR{B1}{A3}{0}
\EdgeR{B1}{A4}{1}
\EdgeL{B1}{A5}{2}
\EdgeR{B2}{A6}{0}
\EdgeR{B2}{A7}{1}
\EdgeL{B2}{A8}{2}
\EdgeR{B3}{A9}{0}
\EdgeR{B3}{A10}{1}
\EdgeL{B3}{A11}{2}
\EdgeR{B4}{A12}{0}
\EdgeR{B4}{A13}{1}
\EdgeL{B4}{A14}{2}
\EdgeR{B5}{A15}{0}
\EdgeR{B5}{A16}{1}
\EdgeL{B5}{A17}{2}
\end{VCPicture}}%
}

\subfigure[$\Tc_{F}$]{%
\setlength{\lga}{1.4cm}%
\setlength{\lgb}{1.3cm}%
\VCDraw{%
\begin{VCPicture}{(0,0)(4\lga,6.5)}
\SmallState
\ChgStateLabelScale{.6}
\VCPut{(0,0)}{\State[8]{(0,0)}{A0}%
              \State[9]{(1\lga,0)}{A1}%
              \State[10]{(2\lga,0)}{A2}%
              \State[11]{(3\lga,0)}{A3}%
              \State[12]{(4\lga,0)}{A4}}%
\VCPut{(0,\lgb)}{\State[5]{(.5\lga,0)}{B0}%
                 \State[6]{(2\lga,0)}{B1}%
                 \State[7]{(3.5\lga,0)}{B2}}%
\VCPut{(0,2\lgb)}{\State[3]{(1.25\lga,0)}{C0}
                  \State[4]{(3.5\lga,0)}{C1}}%
\VCPut{(2.375\lga,3\lgb)}{\State[2]{(0,0)}{D0}}%
\VCPut{(2.375\lga,4\lgb)}{\State[1]{(0,0)}{E0}}%
\VCPut{(2.375\lga,5\lgb)}{\State[0]{(0,0)}{F0}}%
\ChgEdgeArrowStyle{-}
\ChgEdgeLabelScale{.8}
\EdgeL{F0}{E0}{1}
\EdgeR{E0}{D0}{0}
\EdgeR{D0}{C0}{0}
\EdgeL{D0}{C1}{1}
\EdgeR{C0}{B0}{0}
\EdgeL{C0}{B1}{1}
\EdgeR{C1}{B2}{0}
\EdgeR{B0}{A0}{0}
\EdgeL{B0}{A1}{1}
\EdgeR{B1}{A2}{0}
\EdgeR{B2}{A3}{0}
\EdgeL{B2}{A4}{1}
\end{VCPicture}}
}
\eee
\subfigure[$\Tc_{E}$]{%
\setlength{\lga}{1.2cm}%
\setlength{\lgb}{2cm}%
\VCDraw{%
\begin{VCPicture}{(0,0)(12\lga,6.5)}
\SmallState
\ChgStateLabelScale{.6}
\VCPut{(0,0)}{\State[8]{(0,0)}{A0}%
              \State[9]{(\lga,0)}{A1}%
              \State[10]{(2\lga,0)}{A2}%
              \State[11]{(3\lga,0)}{A3}%
              \State[12]{(4\lga,0)}{A4}
              \State[13]{(5\lga,0)}{A5}%
              \State[14]{(6\lga,0)}{A6}%
              \State[15]{(7\lga,0)}{A7}%
              \State[16]{(8\lga,0)}{A8}
              \State[17]{(9\lga,0)}{A9}%
              \State[18]{(10\lga,0)}{A10}%
              \State[19]{(11\lga,0)}{A11}%
              \State[20]{(12\lga,0)}{A12}}%
\VCPut{(0,\lgb)}{\State[3]{(\lga,0)}{B0}%
                 \State[4]{(4\lga,0)}{B1}%
                 \State[5]{(6.5\lga,0)}{B2}
                 \State[6]{(9\lga,0)}{B3}%
                 \State[7]{(11.5\lga,0)}{B4}
		 }%
\VCPut{(0,2\lgb)}{\State[1]{(4\lga,0)}{C0}
                  \State[2]{(10.25\lga,0)}{C1}}%
\VCPut{(0,3\lgb)}{\State[0]{(7.125\lga,0)}{D0}}%
\ChgEdgeArrowStyle{-}
\ChgEdgeLabelScale{.8}
\EdgeR{D0}{C0}{1}
\EdgeL{D0}{C1}{2}
\EdgeR{C0}{B0}{0}
\EdgeR{C0}{B1}{1}
\EdgeL{C0}{B2}{2}
\EdgeR{C1}{B3}{0}
\EdgeL{C1}{B4}{1}
\EdgeR{B0}{A0}{0}
\EdgeR{B0}{A1}{1}
\EdgeL{B0}{A2}{2}
\EdgeR{B1}{A3}{0}
\EdgeR{B1}{A4}{1}
\EdgeL{B1}{A5}{2}
\EdgeR{B2}{A6}{0}
\EdgeL{B2}{A7}{1}
\EdgeR{B3}{A8}{0}
\EdgeR{B3}{A9}{1}
\EdgeL{B3}{A10}{2}
\EdgeR{B4}{A11}{0}
\EdgeL{B4}{A12}{1}
\end{VCPicture}}
}
\caption{First levels of three language trees.}
\label{f.lan-tre}%
\end{figure}

A language~$L$ of~$\Ae$ is said to be \emph{recognizable by a finite
automaton} or \emph{rational} if there exists a finite automaton~$\Ac$
such~$L$ is equal to the set~$\CompAuto{\Ac}$ of labels of paths
starting in an initial state and ending in a terminal state.
The set of rational languages over the alphabet~$A$ is denoted
by~$\Rat\Ae$.
Note that the automata defined below implicitly read words \emph{from
left to right}.

\begin{figure}[h]
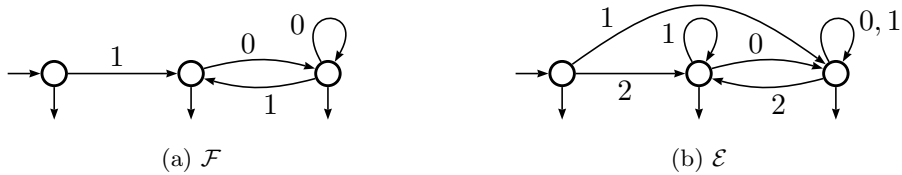

\setlength{\lga}{3cm}%
\centering
\subfigure[${\Fc}$]{%
\VCDraw{%
\begin{VCPicture}{(0,-1)(2\lga,1.5)}
\SmallState
\State{(0,0)}{A}\State{(\lga,0)}{B}\State{(2\lga,0)}{C}
\Initial[w]{A}\Final[s]{A}\Final[s]{B}\Final[s]{C}
\EdgeL{A}{B}{1}
\ArcL{B}{C}{0}
\ArcL{C}{B}{1}
\LoopN{C}{0}
\end{VCPicture}}%
}
\eee\eee
\subfigure[${\Ec}$]{%
\VCDraw{%
\begin{VCPicture}{(0,-1)(2\lga,1.5)}
\SmallState
\State{(0,0)}{A}\State{(\lga,0)}{B}\State{(2\lga,0)}{C}
\Initial[w]{A}\Final[s]{A}\Final[s]{B}\Final[s]{C}
\EdgeR{A}{B}{2}
\ArcL{B}{C}{0}
\ArcL{C}{B}{2}
\ChgLArcAngle{36}%
\ChgLArcCurvature{1.05}%
\LArcL[.15]{A}{C}{1}
\LoopN[.2]{B}{1}
\LoopN[.75]{C}{0,1}
\end{VCPicture}}%
}
\caption{Two automata for representation languages.}
\label{f.Fib-Fi-na}
\end{figure}

\subsection{A calculus classic}
\label{s.cal-cla}

The general following statement will be used several times in the 
paper. 

\begin{lemma}
\label{l.cal-cla}
Let~$\msp\big(x(n)\big)_{n\in\N}\msp$ be an increasing sequence of 
positive 
numbers\linebreak
and~$\msp\big(y(n)\big)_{n\in\N}\msp$ the sequence of the 
sums of 
initial segments: $y(n)=\sum_{i=0}^n x(i)$ for every~$n$.
Then, the following statements are equivalent:\\[.7ex]
\thi $\displaystyle{\lim_{n\to\infty} \xmd \frac{x(n+1)}{x(n)}}$
exists and is equal to~$\gamma>1$;\\[.7ex]
\thii $\displaystyle{\lim_{n\to\infty} \xmd \frac{y(n+1)}{y(n)}}$ 
exists and is equal to~$\gamma>1$;\\[.7ex]
\thiii $\displaystyle{\lim_{n\to\infty} \xmd \frac{y(n)}{x(n)}}$ 
exists and is equal to~$\frac{\gamma}{\gamma-1}$.
\end{lemma}

\begin{proof}
Let us recall a classical result \cite[Chap.~V.4, Prop.~2]{Bourbaki}. 
Let~$(u_n)_{n\in\N}$ and~$(v_n)_{n\in\N}$ be two sequences of 
non-negative numbers. 
If the series $\sum_{j=0}^{+\infty} v_j$ is divergent, 
then $u_n\sim v_n$ implies that 
$\sum_{j=0}^{n} u_j\sim \sum_{j=0}^{n} v_j$. 
This result is sometimes referred to as Stolz--Ces\`aro Theorem.

\medskip
\noindent
(i) implies (ii): by the ratio test, the series $\sum_{j=0}^{+\infty} 
x(j)$ is divergent. 
Apply the above result with $\msp(u_n)_{n\in\N}=(x(n+1))_{n\in\N}\msp$ 
and $\msp(v_n)_{n\in\N}=(\gamma x(n))_{n\in\N}\msp$.
We obtain that
\begin{equation}
\sum_{j=0}^n u_j=\sum_{j=1}^{n+1}x(j)=y(n+1)-x(0)\sim
\sum_{j=0}^n v_j=\gamma \sum_{j=0}^nx(j)=\gamma y(n)
\notag
\end{equation}
and the conclusion follows.

\medskip
\noindent
(ii) implies (i): from (ii) we have:
\begin{equation}
\lim_{n\to\infty} \xmd \frac{y(n)+x(n+1)}{y(n)}=\gamma
\notag
\end{equation} 
and thus $\displaystyle{\frac{x(n+1)}{y(n)}}\to\gamma-1$.
Since
\begin{equation}
\frac{x(n+1)}{x(n)}=\frac{x(n+1)}{y(n)} \xmd
                    \frac{y(n)}{y(n-1)} \xmd\frac{y(n-1)}{x(n)}
\eqvrg
\notag
\end{equation}
the result follows.

\medskip
\noindent
(ii) implies (iii): since $y(n)=y(n-1)+x(n)$ dividing both sides by 
$y(n-1)$ and letting $n$ tends to infinity, leads to 
\begin{equation}
\gamma=1+\lim_{n\to\infty} \frac{x(n)}{y(n)}\frac{y(n)}{y(n-1)}
\eqpnt
\notag
\end{equation} 
We conclude that
$\displaystyle{\frac{x(n)}{y(n)}}\to\frac{\gamma-1}{\gamma}$.
 
\medskip
\noindent
(iii) implies (ii): again since $y(n)=y(n-1)+x(n)$, observe that 
\begin{equation}
\lim_{n\to\infty} \xmd \frac{y(n)}{x(n)}=\frac{\gamma}{\gamma-1}
\ee\text{if and only if}\ee
\lim_{n\to\infty} \xmd \frac{y(n-1)}{x(n)}=\frac{1}{\gamma-1}
\eqpnt
\notag
\end{equation} 
Since 
$\displaystyle{\frac{y(n)}{y(n-1)}=
               \frac{y(n)}{x(n)}\xmd\frac{x(n)}{y(n-1)}}$, 
the result follows.
\end{proof}



\section{The carry propagation of a language: 
         \protect\\ \eee \eee \e 
	 a combinatorial point of view}
\label{s.car-pro-num}%

We first define the carry propagation of a language~$L$ and show that it 
does not always exist. 
Sufficient conditions for its existence, and for its computation are
then investigated in terms, first, of growth rates of the language
and, then, of the notion of signature associated with the language
tree of~$L$, by stressing the fact that what counts for carry
propagation is the shape of the tree and not its labeling.

\subsection{First definitions for the carry propagation}

We write~$u\lcp v$ for the \emph{longest common left factor} of two 
words~$u$ and~$v$ of~$\Ae$.
If two words~$u$ and~$v$ of~$\Ae$ have the \emph{same length}, we 
write~$\rdiff{u,v}$ for the (common) length of the left quotient of~$u$ 
(or~$v$) by~$u\lcp v$:
\begin{equation}
	\rdiff{u,v} = |u| - |u\lcp v| = |v| - |u\lcp v|     
	\eqpnt
	\notag
\end{equation}
If~$u$ and~$v$ do \emph{not} have the same length, we set
\begin{equation}
	\rdiff{u,v} = \max\{|u|,|v|\}
	\eqvrg
	\notag
\end{equation}
which is the same as~$\rdiff{u',v'}$
where~$u'$ and~$v'$ are obtained from~$u$ and~$v$ by padding the 
shorter word \emph{on the left} with a symbol which is not in~$A$ and so 
that~$|u'|=|v'|$.

\begin{definition}
\label{d.car-pro}%
Let~$A$ be an ordered alphabet and~$L$ a language of~$\Ae$, ordered 
by radix order.
The \emph{carry propagation} at a word~$w$ of~$L$, and with respect 
to~$L$, is the quantity:
\begin{equation}
\capr{w} = \rdiff{w,\succ{w}}
\eqpnt
\notag
\end{equation}
We naturally consider a language over an ordered alphabet as 
an abstract numeration system and we also write, for every integer~$i$,
\begin{equation}
\capr{i} = \capr{\repr{i}} = \rdiff{\repr{i},\repr{i+1}}
\eqpnt
\notag
\end{equation}
\end{definition}

\begin{example}[\examp{Fib-1} continued]
\label{e.Fib-2}%
In the Fibonacci numeration system, 
$\repr[F]{9}=1\xmd0\xmd0\xmd0\xmd1$ and
$\repr[F]{1\xmd0}=1\xmd0\xmd0\xmd1\xmd0$,
hence $\msp\capr[F]{9} = 2\msp$. 
We also have $\repr[F]{12}=1\xmd01\xmd01$ 
and  $\repr[F]{13}=1\xmd0\xmd0\xmd0\xmd0\xmd0$. 
Thus $\msp\capr[F]{12} = 6\msp$.
\end{example}

From the definition of the carry propagation at a word, we derive the
carry propagation \emph{of a language}.
We first denote by~$\scpL{N}$ \emph{the sum of the carry propagations
at the first~$N$ words of the language~$L$}:
\begin{equation}
\scpL{N}= \sum_{i=0}^{N-1}\cpL{i}
\eqpnt
\label{q.scp}
\end{equation}

\begin{definition}[\cite{BarcEtAl05}]
\label{d.car-pro-lan}%
The \emph{carry propagation} of a language~$L\subseteq\Ae$, which we
denote by~$\CPL$, is the \emph{amortized} carry propagation at the words
of the language, that is, the limit, \emph{if it exists}, of the mean
of the carry propagation at the first~$N$ words of the language:
\begin{equation}
\CPL = \lim_{N \to \infty} \frac{1}{N}\scpL{N}
\eqpnt
\notag
\end{equation}
\end{definition}

\subsection{The language tree and the carry propagation}
\label{s.car-pro-tri}%

We denote by~$\wlgl$ (resp.~$\wlgsl$) the number of words 
of~$L$ of length~$\ell$ (resp. of length less than, or equal 
to,~$\ell$):
\begin{equation}
    \uLl =\jsCard{L\cap A^{\ell}}
    \e\text{and}\e
    \vLl =\jsCard{L\cap A^{\leq\ell}}=\sum_{i=0}^{\ell}\uL{i}
    \eqpnt
    \notag
\end{equation}
The set of words of $L$ of each length that are maximal in the radix 
(or lexicographic) order is denoted by~$\Maxlg$.
We have:
\begin{gather}
    \Maxlg =\Defi{\repr{\vLl-1}}{\ell\in\N}
    \ee\text{and}
    \notag\\
    L\cap A^{\ell} =\Defi{u\in L}{\vL{\ell-1}\leq\val{u}<\vLl}
    \eqpnt
    \notag
\end{gather}

The carry propagation~$\CPL$ is more easily evaluated when the
terms~$\capr{i}$ of the sum~$\scp{N}$ are first \emph{aggregated} in
\emph{partial sums} corresponding to \emph{words of fixed length}.
In particular, we can state:

\begin{proposition} 
\label{p.car-pro-agg}%
If~$L$ is a \pce language, then, for every integer~$\ell$,
\begin{equation}
    \sum_{\substack{w\in L\\ |w|=\ell}}\cpL{w}
	      = \sum_{i=\vL{\ell-1}}^{\vLl-1}\cpL{i}
	      = \vLl
	\eqpnt
\label{q.agg-1}
\end{equation}
\end{proposition}

This proposition is indeed an instance of the more precise
\theor{lb-car-pro} that will be used in subsequent proofs and that
requires a definition.

Let~$\TcL$ be the language tree of~$L$ and~$w$ a word of~$L$ of 
length~$\ell$.
Let us denote by~$\TcLl$ the part of~$\TcL$ which consists of 
words of~$L$ of length less than, or equal to,~$\ell$.
And let us see~$\Pre{w}$, the set of prefixes of~$w$ that form the
unique path from~$w$ to~$\epsilon$, as a river that flows from~$w$
to~$\epsilon$; it determines two subsets of~$\TcLl$: the `left bank'
of~$w$,~$\LB{w}$, and the `right bank' of~$w$,~$\RB{w}$, which
consists respectively of the nodes on the left and on the right
of~$\Pre{w}$ as depicted in \figur{bank}.
Together, $\LB{w}$, $\Pre{w}$ and~$\RB{w}$ form a partition 
of~$\TcLl$ and we have:
\begin{equation}
    \LB{w} = \Defi{u\in L}{|u|\jsleq|w| \e\text{and}\e u\lex w} \bk\Pre{w}
    \eqpnt
    \notag
\end{equation}

\begin{figure}[h!tb]
\centering
\scalebox{1}{
\psfrag{LB}{$\LB{w}$}
\psfrag{RB}{$\RB{w}$}
\psfrag{e}{$\varepsilon$}
\psfrag{w}{$w$}
\psfrag{T}{$\TcLl$}
\includegraphics{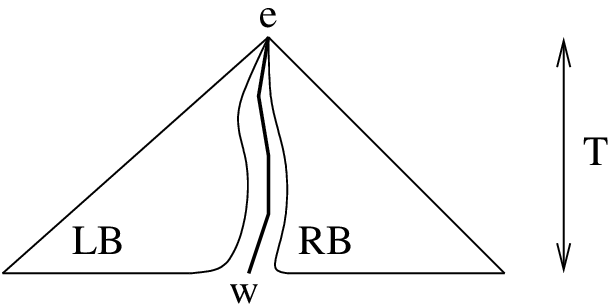}}
\caption{The tree~$\TcLl$ and the three sets $\LB{w}$, $\Pre{w}$ 
and~$\RB{w}$.} 
\label{f.bank}
\end{figure}

\begin{example}[\examp{Fib-1} continued]
\label{e.Fib-3}
In~$\Tc_{F}$, we have:
\begin{align}
\LB[F]{10010} &= \{1000,10000,10001\}
\e\text{and}\e
\notag\\
\LB[F]{10100} &=\LB[F]{10010} \cup \{100,1001,10010\}
\eqpnt
\eee\eee
\notag
\end{align}
\end{example}

The statement we are aiming at gives the sum of the carry propagation 
at all words of the same length~$\ell$ as a word~$u$ and less than or 
equal to~$u$ in the lexicographic order.
We recall that $\Maxlg =\Defi{\repr{\vLl-1}}{\ell\in\N}$.

\begin{theorem} 
\label{t.lb-car-pro}%
Let~$L$ be a \pce language, $u$ in~$L$ of length~$\ell$ 
and~$\msp N=\val{u}\msp$. Then, we have:
\begin{equation}
\sum_{i=\vL{\ell-1}}^{N}\cpL{i} = 
\begin{cases}
        \; \jsCard{\LB{\succ{u}}} 
	          & \e\text{if}\e u\not\in\Maxlg\eqvrg\\
        \; \vLl & \e\text{if}\e u\in\Maxlg \eqpnt 
\end{cases}
\label{q.agg-2}
\end{equation}
\end{theorem}
\begin{figure}[b!htp]
    \centering
    \scalebox{1}{
\psfrag{LB}{$\LB{u}$}
\psfrag{e}{$\varepsilon$}
\psfrag{w}{$w$}\psfrag{u}{$u$}\psfrag{v}{$v$}
\psfrag{x}{$x$}\psfrag{y}{$y$}\psfrag{z}{$z$}
\includegraphics{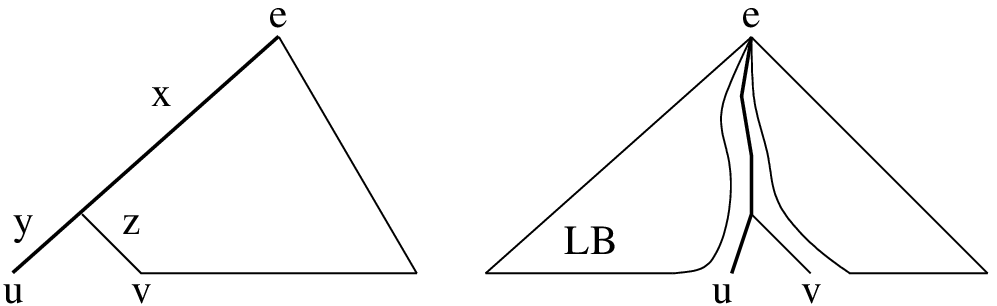}}
\caption{Illustrations of the first two cases of the proof of 
\theor{lb-car-pro}.}
\label{f.bank2}
\end{figure}

\ifelsevier
\begin{proof}
\else
\begin{proofss}
\fi
The proof works by  induction on~$N$.
Let~$u=\repr{N}$ and~$\ell=|u|$.
Hence $\vL{\ell-1}\leq N\leq\vLl-1$. 

\thi~We first assume $N=\vL{\ell-1}$.
The word~$u$ is the smallest word of~$L$ of length~$\ell$ (in the
lexicographic order).
Let~$v=\Succ{u}$.
We first suppose that~$u$ is not in~$\Maxlg$.
Then (as depicted on the left part of \figur{bank2}):
\begin{equation}
    u = x\xmd y\EqVrgInt
    v = x\xmd z
    \e\text{and}\e
    \cpL{N} = \rdiff{u,v} = |y| =|z|
    \eqpnt
    \notag
\end{equation}
But, since~$L$ is prefix-closed, $|y|$~is exactly the number of nodes 
in~$\LB{v}$.

Now, if~$u=\repr{\vL{\ell-1}}$ is 
in~$\Maxlg$, then there is only one word in~$L$ for every length~$k$,
$0\leq k\leq\ell$ and we have at the same time~$\cpL{u}=\ell+1$ and
$\vLl=\ell+1$, hence \equnm{agg-2} still holds in this case.

\thii
We now assume  $\vL{\ell-1}< N<\vLl-1$.
Hence~$v=\Succ{u}$ is of length~$\ell$ and as above there exist~$x$, 
$y$ and~$z$ such that~$u=x\xmd y$ and~$v = x\xmd z$.
The same reasoning as above applies (see the right part 
of~\figur{bank2}): 
\begin{multline}
    \sum_{i=\vL{\ell-1}}^{N}\cpL{i} =
    \sum_{i=\vL{\ell-1}}^{N-1}\cpL{i} + \cpL{N} =
    \\
    \jsCard{\LB{u}} + \cpL{u} = \jsCard{\LB{u}} + |y|
    \EqPnt
    \notag
\end{multline}
But, since~$L$ is prefix-closed, $|y|$~is  the number of nodes 
in~$\LB{v}\bk\LB{u}$.

\thiii~ Finally, assume $N=\vLl-1$.
Then~$u=\repr{N}$ belongs to~$\Maxlg$.
In this case,
\begin{gather}
    \cpL{u}=\ell+1\EqVrgInt
    \LB{u}\cup\Pre{u}=\TcLl
    \e\text{and}
    \notag\\
    \sum_{i=\vL{\ell-1}}^{N}\cpL{i} =  \jsCard{\LB{u}} + (\ell+1) =
    \jsCard{\TcLl} = \vLl
    \eqpnt
    \notag
\ifelsevier\else\tag*{\EoP}\fi
\end{gather}
\ifelsevier
\end{proof}
\else
\end{proofss}
\fi

\propo{car-pro-agg} is the instance of \theor{lb-car-pro} 
when~$\msp N=\vLl-1\msp$.
By grouping the sum of carry propagations by words of the same length,
\theor{lb-car-pro} also yields an evaluation of the sum \equnm{scp}
of the carry propagations at the first~$N$ words of a \pce language:

\begin{corollary} 
\label{c.sum-cp}%
Let~$L$ be a \pce language, $u$ in~$L$ of length~$\ell$ 
and~$\msp N=\val{u}\msp$.
We then have:
\begin{equation}
\scp{N} = 
\sum_{i=0}^{\ell-1} \vL{i} + \jsCard{\LB{\Succ{u}}}
\eqpnt
\notag
\end{equation}
\end{corollary}

\begin{example}[\examp{int-bas} continued]
\label{e.int-bas-2}%
Let~$p$ be an integer, $p>1$, and~$\Lp$ the set of $p$-representations
of the integers.
In order to lighten the notation, we write~$\up{\ell}$ instead
of~$\uLp{\ell}$, $\bfvp{\ell}$ instead of~$\vLp{\ell}$, $\Succp{u}$
instead of~$\Succ[\Lp]{u}$, $\CP[p]$ instead of~$\CP[\Lp]$, \etc As a
first application of \theor{lb-car-pro}, Proposition~\ref{p.int-bas}
below allows one to recover the value $\displaystyle{\frac{p}{p-1}}$
for the carry propagation.
\end{example}

\begin{proposition} 
\label{p.int-bas}%
Let~$p$ be an integer, $p>1$, 
and~$\Lp$
the set of $p$-representations of the integers. 
The carry propagation $\CP[p]$  exists and is equal to 
$\displaystyle{\frac{p}{p-1}}$. 
\end{proposition}

\begin{proof}
Let~$N$ in~$\N$ and~$u=\repr[p]{N}$; 
we have: $\msp\bfvp{\ell-1}\leq N<\bfvp{\ell}\msp$ with~$\ell=|u|$.
And then, by Corollary~\ref{c.sum-cp}:
\begin{equation}
\scp[p]{N} = 
\sum_{i=0}^{\ell-1} \bfvp{i} + \jsCard{\LB[p]{\Succ[p]{u}}}
\eqpnt
\label{q.int-bas-1}
\end{equation}

\noindent 
First, and since~$\bfvp{k}=p^{k}$, for every~$k$, then, by 
\lemme{cal-cla},
\begin{gather}
\lim_{\ell\to\infty}\frac{1}{\bfvp{\ell-1}}
\sum_{i=0}^{\ell-1}\bfvp{i}
=\frac{p}{p-1}
\notag\\
\intertext{which can be written as:}
\sum_{i=0}^{\ell-1}\bfvp{i} =
\bfvp{\ell-1}\left(\frac{p}{p-1}+\epsilon(\ell)\right)
\ee\text{with }
\lim_{\ell\to\infty}\xmd\epsilon(\ell)=0
\eqpnt 
\label{q.int-bas-2}
\end{gather}

\noindent
Second, we turn to the evaluation of~$\jsCard{\LB[p]{\Succ[p]{u}}}$.
Let~$v=\Succ[p]{u}$;
we exclude the case where~$N=\bfvp{\ell}-1$ and~$|v|=\ell+1$
(which corresponds to $\scp[p]{N} = \sum_{i=0}^{\ell} \bfvp{i}$)
and we write 
\begin{equation}
N=\bfvp{\ell-1} + (M-1)
    \e\text{with}\e
    1\leq M<\up{\ell}
\eqpnt
\notag
\end{equation}
We write~$\Tcp$ for the language tree of~$\Lp$, $\Tcpl$ for its 
truncation to length~$\ell$.
Every node of~$\Tcp$, every internal node of~$\Tcpl$,
is of degree~$p$, but the root, which is of 
degree~$p-1$.

Since~$\reprp{\bfvp{\ell-1}}$ is the smallest word of~$\Lp$ of 
length~$\ell$, $v$~is the $(M+1)$-th word of~$\Lp$ of length~$\ell$ 
and~$\jsCard{\LB[p]{v}\cap A^{\ell}}=M$. 
We suppose that~$\ell>1$ (which is not a restriction since we
want~$\ell$ to tend to infinity).
Since every internal node of~$\Tcpl$ at level~$\ell-1$ is of
degree~$p$:
\begin{equation}
\jsCard{\LB[p]{v}\cap A^{\ell-1}}=\Inte{\frac{M}{p}}
\eqpnt
\notag
\end{equation}
By induction on~$k$, $1\leq k<\ell-1$, and with the same argument: 
\begin{equation}
\jsCard{\LB[p]{v}\cap A^{\ell-k}}=\Inte{\frac{M}{p^{k}}}
\eqpnt
\notag
\end{equation}
From the inequalities
$\msp\frac{M}{p^{k}}-1
   \leq\Inte{\frac{M}{p^{k}}}
   \leq\frac{M}{p^{k}}\msp$,  
we first get 
$$M+\sum_{k=1}^{\ell-1}\Inte{\frac{M}{p^{k}}}\geq 
M+\sum_{k=1}^{\ell-1}\frac{M}{p^k}-(\ell-1)
=M\frac{p}{p-1}-\ell+1-\frac{M}{p^{\ell-1}(p-1)}$$
and since~$M<p^{\ell}$, 
$1-M/(p^{\ell-1}(p-1))\geq -1$, 
it leads to the lower and upper bounds  
\begin{equation}
M\frac{p}{p-1}-(\ell+1)\leq\jsCard{\LB[p]{v}}\leq M\frac{p}{p-1}
\eqpnt
\notag
\end{equation}
Together with \equnm{int-bas-2}, they yield the bounds
\begin{equation}
(N+1)\frac{p}{p-1}+\bfvp{\ell-1}\xmd\epsilon(\ell)-(\ell+1)
\leq\scp[p]{N}\leq 
(N+1)\frac{p}{p-1}+\bfvp{\ell-1}\xmd\epsilon(\ell)
\EqPnt
\notag
\end{equation}
If we divide by~$N$, both the lower and upper bounds tend 
to~$\frac{p}{p-1}$ when~$N$ tends to infinity, 
hence~$\frac{1}{N}\scp[p]{N}$ has a limit, and this limit 
is~$\frac{p}{p-1}$.
\end{proof}

After \propo{car-pro-agg}, it is natural to extract from the sequence
of means of carry propagations up to the first~$N$ words of~$L$, those
that correspond to the first~$\vLl$ words of~$L$.

\begin{definition}
\label{d.car-pro-fil} 
For a language~$L$, we call the limit, \emph{if it exists}, of the
mean of the carry propagation at the first~$\vLl$ words of~$L$ the
\emph{length-filtered carry propagation} of~$L$, and we denote it
by~$\FCPL$:
\begin{equation}
\FCPL = \lim_{\ell\to\infty}\xmd 
	\frac{1}{\vLl}\scp{\vLl} 
      = \lim_{\ell\to\infty}\xmd 
	\frac{1}{\vLl}\sum_{i=0}^{\ell} \vL{i}
    \eqpnt
    \label{q.agg-3}
\end{equation}
\end{definition}

\begin{remark}\label{r.fcpl-cpl}
Of course, 
if $\CPL$ exists, then $\FCPL$ exists and $\CPL=\FCPL$ but the 
converse does not hold as we shall see with \exemp{un-bal}.
On the other hand, an easy way for showing that~$\CPL$ does not 
exist is to prove that~$\FCPL$ does not exist.
\end{remark}

\subsection{The local growth rate and the carry propagation}
\label{s.car-pro-gro}%

From \propo{car-pro-agg} it also follows that the carry propagation of
a language~$L$ is closely related to other growth measures of~$L$.
It is the case in particular of the growth rates.

First, the \emph{global growth rate}~$\eta_L$ of a language~$L$
(called \emph{growth rate} in~\cite{Shur08} for instance)
is classically defined by:
\begin{equation}
    \eta_L= \limsup_{\ell\to\infty}\xmd \sqrt[\ell]{\uL{\ell}}
    \eqpnt
    \notag
\end{equation}
A language~$L$ is said to have \emph{exponential growth}
if $\eta_L>1$ and \emph{polynomial growth} 
if~$\uL{\ell}\leq P(\ell)$ for some polynomial~$P$ and all large 
enough~$\ell$.

\begin{example}
\label{e.pol-gro}%
\textbf{Languages with polynomial growth.}
Let $L$ be a \pce language such that $\wlgl=P(\ell)$ for
some polynomial~$P$ of degree~$d$. 
Then~$\vLl$ is a polynomial of degree~$d+1$ and by 
\propo{car-pro-agg}, 
$\scp{\vLl}$ is a polynomial of degree~$d+2$.
Hence
$\lim_{\ell\to\infty}\frac{1}{\vLl}\scp{\vLl}=+\infty$
and~$\FCPL$ does not exist.
\end{example}

\begin{definition}
\label{d.ins-gro-rat}%
We call the limit, \emph{if it exists}, of the ratio between the
number of words of a language~$L$ of length~$\ell$ and the number of
words of~$L$ of length~$\ell+1$, when~$\ell$ tends to infinity, the
\emph{local growth rate} of~$L$, and we denote it by~$\igrr$:
\begin{equation}
\igrr=\lim_{\ell\to +\infty}\xmd\frac{\uL{\ell+1}}{\uLl}
\eqpnt
\notag
\end{equation}
\end{definition}

\begin{remark}
\label{r.igr}%
Observe that the quantity~$\eta_L$ always exists since it is defined 
by an upper limmit. 
If the local growth rate~$\igrr$ exists, then $\igrr=\eta_L$.
\end{remark}

The definition of length-filtered carry propagation (see
\defin{car-pro-fil}) together with \propo{car-pro-agg} directly
implies the following.

\begin{proposition}
\label{p.car-pro-fil}%
Let~$L$ be a \pce language with exponential growth. 
Then, $\FCPL$ exists if and only if~$\igrr$ exists and, in this 
case, $\displaystyle{\FCPL =\frac{\igrr}{\igrr-1}}$ holds.
\end{proposition}

\ifelsevier
\begin{proof}
\else
\begin{proofss}
\fi
Using \lemme{cal-cla}, if
$\lim_{\ell\to\infty}\frac{\uL{\ell+1}}{\uLl}=\igrr$,
then $\lim_{\ell\to\infty}\frac{\vL{\ell+1}}{\vLl}=\igrr$. 
Using again \lemme{cal-cla}, the latter limit exists if and only if 
\begin{equation}
\FCPL=\lim_{\ell\to\infty}\xmd\frac{1}{\vLl}\sum_{i=0}^{\ell}\vL{i}
     =\frac{\igrr}{\igrr-1}
\eqpnt
\ifelsevier
\else
    \tag*{\EoP}
\fi
\end{equation} 
\ifelsevier
\end{proof}
\else
\end{proofss}
\fi

From \remar{fcpl-cpl}, the following holds, which extends the case of
numeration in base~$p$ described in \propo{int-bas}.

\begin{corollary}
\label{c.car-pro-fil}%
If the carry propagation~$\CPL$ exists, then the local growth~$\igrr$
exists and $\displaystyle{\CPL=\frac{\igrr}{\igrr-1}}$.
\end{corollary}

However, the existence of~$\igrr$, and hence of~$\LFCAPR$, does not
imply in general the existence of the carry propagation~$\CAPR$ of a
language, as witnessed by the following example.

\begin{example}
\label{e.un-bal}%
\textbf{A language with an unbalanced tree.}
Let~$A=\{a,b,c\}$.
The \pce language~$H$ we build will be such that
$\bfu[H]{\ell}=2^\ell$, for every~$\ell$.
We denote by~$H_{\ell}$ the set~$H\cap A^\ell$ and by~$H'_{\ell}$ 
(resp.~$H''_{\ell}$) the first (resp. the last) $2^{\ell-1}$ words of 
length~$\ell$ in the radix ordered language $H_{\ell}$.
Set $H_1 =\{a,c\}$. 
For all $\ell > 0$, 
$H_{\ell+1} = \{H'_{\ell}\} A \cup\{H''_{\ell}\}b$. 
Thus we get $H_2 = \{aa, ab, ac, cb\}$, 
$H_3 =\{aaa, aab, aac, aba, abb, abc, acb, cbb\}$ 
and it is clear that
$\bfu[H]{\ell}=2^\ell$ and $\bfv[H]{\ell}=2^{\ell+1}-1$.
Hence~$\igrr[H]=2$ and~$\FCP[H]=2$ by \propo{car-pro-fil}.  

\begin{figure}[htbp]
\setlength{\lga}{1cm}
\setlength{\lgb}{2.1cm}
\centering
\VCDraw{%
\begin{VCPicture}{(0,0)(18\lga,3.3\lgb)}
\VCPut{(0,0)}{\multido{\i=0+1}{13}{\VSState{(\i\lga,0)}{A\i}}%
              \VSState{(14\lga,0)}{A13}%
              \VSState{(16\lga,0)}{A14}%
              \VSState{(18\lga,0)}{A15}}%
\VCPut{(0,\lgb)}{\multido{\i=1+3}{4}{\VSState{(\i\lga,0)}{B\i}}
                 \VSState{(12\lga,0)}{B12}%
                 \VSState{(14\lga,0)}{B13}%
                 \VSState{(16\lga,0)}{B14}%
                 \VSState{(18\lga,0)}{B15}%
		 }%
\VCPut{(0,2\lgb)}{\VSState{(4\lga,0)}{C0}%
              \VSState{(12\lga,0)}{C1}%
	      \VSState{(16\lga,0)}{C2}%
	      \VSState{(18\lga,0)}{C3}}%
\VCPut{(0,3\lgb)}{\VSState{(12\lga,0)}{D0}%
	          \VSState{(18\lga,0)}{D1}}%
\VCPut{(0,3.3\lgb)}{\VSState{(15\lga,0)}{E0}}%
\ChgEdgeArrowStyle{-}
\ChgEdgeLabelScale{.8}
\EdgeR{E0}{D0}{a}
\EdgeL{E0}{D1}{c}
\EdgeR{D0}{C0}{a}
\EdgeR{D0}{C1}{b}
\EdgeL{D0}{C2}{c}
\EdgeL{D1}{C3}{b}
\EdgeR{C0}{B1}{a}
\EdgeR{C0}{B4}{b}
\EdgeL{C0}{B7}{c}
\EdgeR{C1}{B10}{a}
\EdgeR{C1}{B12}{b}
\EdgeL{C1}{B13}{c}
\EdgeL{C2}{B14}{b}
\EdgeL{C3}{B15}{b}
\EdgeR{B1}{A0}{a}
\EdgeR{B1}{A1}{b}
\EdgeL{B1}{A2}{c}
\EdgeR{B4}{A3}{a}
\EdgeR{B4}{A4}{b}
\EdgeL{B4}{A5}{c}
\EdgeR{B7}{A6}{a}
\EdgeR{B7}{A7}{b}
\EdgeL{B7}{A8}{c}
\EdgeR{B10}{A9}{a}
\EdgeR{B10}{A10}{b}
\EdgeL{B10}{A11}{c}
\EdgeL{B12}{A12}{b}
\EdgeL{B13}{A13}{b}
\EdgeL{B14}{A14}{b}
\EdgeL{B15}{A15}{b}
\end{VCPicture}
}
\caption{The first $5$ levels of~$\Tc_{H}$.}
\label{f.trie-H}
\end{figure}

\noindent Let, for every~$\ell$, $M(\ell)=2^{\ell+1}-1+2^\ell$ and let
us evaluate~$\scp[H]{M(\ell)}$.

\thi The contribution to~$\scp[H]{M(\ell)}$ of the words of length 
less than, or equal to,~$\ell$ is equal to
$C=\sum_{i=0}^\ell\bfv[H]{i}=\sum_{i=0}^\ell(2^{i+1}-1)
  =2^{\ell+2}-\ell-2$.

\thii By construction of~$H$, the elements of~$H'_{\ell+1}$ are the
leftmost $2^\ell$ leaves of a ternary tree of height~$k$ such that
$3^k>2^\ell$.  Since the carry propagation in base~$3$ is equal
to~$3/2$ (as seen in the proof of \propo{int-bas}), the contribution
of the elements of $H'_{\ell+1}$ to~$\scp[H]{M(\ell)}$ is less than
$D=2^\ell \times 3/2=2^{\ell-1} \times 3$.
We have:
\begin{equation}
    \frac{1}{M(\ell)}\scp[H]{M(\ell)} <
    \frac{C+D}{2^{\ell+1}+2^\ell}=
    \frac{2^{\ell+2}-\ell-2+2^{\ell-1} \times 3}{2^{\ell+1}+2^\ell}
    \eqpnt
    \notag
\end{equation}
Hence
\begin{equation}
\lim_{\ell\to\infty}\frac{1}{M(\ell)}\scp[H]{M(\ell)}
    \le\frac{11}{6}\neq
\lim_{\ell\to\infty}\frac{1}{\bfv[H]{\ell}}\scp[H]{\bfv[H]{\ell}}
    =\FCP[H]=2
\eqpnt
\notag
\end{equation}
The quantity~$\frac{1}{N}\scp[H]{N}$ has no limit when~$N$ tends to 
infinity and~$\CP[H]$ does not exist.
\end{example}

In view of \secti{car-pro-rat} where we prove that the existence 
of the local growth rate of a language and its quotients (see 
\theor{cp-rat-lan}) is a sufficient condition for a \emph{rational} 
language to have a carry propagation, let us add that this 
language~$H$ is easily seen not to be rational. 
Indeed let~$m=m(n)$ be the smallest integer such that~$a^nb^ma$ is 
not in~$H$. 
Then for every~$n'>n$, $a^{n'}b^{m}a$ is an element of $H$, thus 
the words~$a^{n}$ have all distinct sets of right
contexts for~$H$.

\subsection{The signature and the carry propagation}
\label{s.sig-car-pro}%

\theor{lb-car-pro} and its proof make clear that the actual words of a
language~$L$, that is, the \emph{labeling} of the language
tree~$\TcL$, have no impact on the carry propagation of~$L$, its
existence or its value, but what only counts is the \emph{shape}
of~$\TcL$ or, in one more precise word, its \emph{signature},
introduced in \cite{MarsSaka17a,MarsSaka17b}, and that we now define.

First, we introduce a slightly different look at trees that proves to 
be technically fit to the description and study of language trees 
associated with languages seen as abstract numeration systems (see 
\secti{lan-ans} and~\ref{s.lan-tre}).

Given a tree, we consider that in addition to all edges, the root is 
also a \emph{child of itself}, that is, bears a loop onto 
itself.\footnote{%
    This convention is sometimes taken when implementing tree-like
    structures (for instance in the Unix/Linux file system).}
We call such a structure an \emph{i-tree}.\footnote{%
   The terminology comes indeed from the terminology for 
   \emph{inodes} in the Unix/Linux file system.} 
It is so close to a tree that we pass from one to the other
with no further ado.
When a tree is usually denoted by~$\Tc_{x}$ for some index~$x$,
the associated i-tree is denoted by~$\Ic_{x}$, and conversely.

If the tree~$\Tc_{x}$ is labeled by letters of an ordered
alphabet~$A$, we want the loop on the root of the i-tree~$\Ic_{x}$ to
be labeled by a letter less than the labels of all other edges going
out of the root in~$\Tc_{x}$.
Either there exists a letter in~$A$ which meets the condition and it 
can be chosen as label for the loop, or such a letter does not exist 
in~$A$ and we enlarge  the alphabet~$A$ with a new symbol, less  
than all letters of~$A$.
\figur{Fib-tre-itr} shows the language tree of the representation 
language in the Fibonacci numeration system and the associated 
i-tree. 

\begin{figure}[htbp]
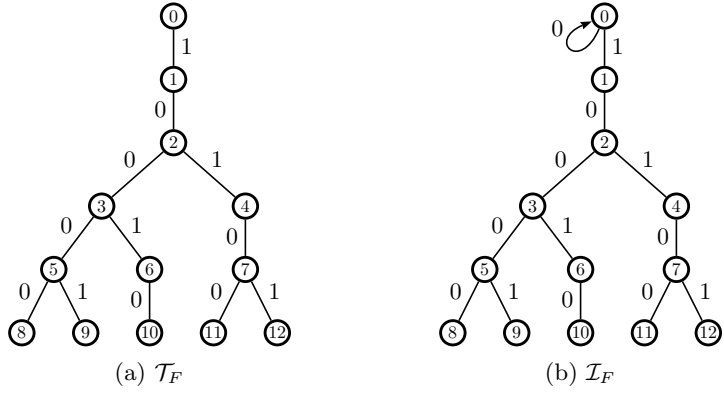

\centering
\subfigure[$\Tc_{F}$]{%
\setlength{\lga}{1.4cm}%
\setlength{\lgb}{1.4cm}%
\VCDraw{%
\begin{VCPicture}{(0,0)(4\lga,7)}
\SmallState
\ChgStateLabelScale{.6}
\VCPut{(0,0)}{\State[8]{(0,0)}{A0}%
              \State[9]{(1\lga,0)}{A1}%
              \State[10]{(2\lga,0)}{A2}%
              \State[11]{(3\lga,0)}{A3}%
              \State[12]{(4\lga,0)}{A4}}%
\VCPut{(0,\lgb)}{\State[5]{(.5\lga,0)}{B0}%
                 \State[6]{(2\lga,0)}{B1}%
                 \State[7]{(3.5\lga,0)}{B2}}%
\VCPut{(0,2\lgb)}{\State[3]{(1.25\lga,0)}{C0}
                  \State[4]{(3.5\lga,0)}{C1}}%
\VCPut{(2.375\lga,3\lgb)}{\State[2]{(0,0)}{D0}}%
\VCPut{(2.375\lga,4\lgb)}{\State[1]{(0,0)}{E0}}%
\VCPut{(2.375\lga,5\lgb)}{\State[0]{(0,0)}{F0}}%
\ChgEdgeArrowStyle{-}
\ChgEdgeLabelScale{.8}
\EdgeL{F0}{E0}{1}
\EdgeR{E0}{D0}{0}
\EdgeR{D0}{C0}{0}
\EdgeL{D0}{C1}{1}
\EdgeR{C0}{B0}{0}
\EdgeL{C0}{B1}{1}
\EdgeR{C1}{B2}{0}
\EdgeR{B0}{A0}{0}
\EdgeL{B0}{A1}{1}
\EdgeR{B1}{A2}{0}
\EdgeR{B2}{A3}{0}
\EdgeL{B2}{A4}{1}
\end{VCPicture}}
}
\eee\ee
\subfigure[$\Ic_{F}$]{%
\setlength{\lga}{1.4cm}%
\setlength{\lgb}{1.4cm}%
\VCDraw{%
\begin{VCPicture}{(0,0)(4\lga,7)}
\SmallState
\ChgStateLabelScale{.6}
\VCPut{(0,0)}{\State[8]{(0,0)}{A0}%
              \State[9]{(1\lga,0)}{A1}%
              \State[10]{(2\lga,0)}{A2}%
              \State[11]{(3\lga,0)}{A3}%
              \State[12]{(4\lga,0)}{A4}}%
\VCPut{(0,\lgb)}{\State[5]{(.5\lga,0)}{B0}%
                 \State[6]{(2\lga,0)}{B1}%
                 \State[7]{(3.5\lga,0)}{B2}}%
\VCPut{(0,2\lgb)}{\State[3]{(1.25\lga,0)}{C0}
                  \State[4]{(3.5\lga,0)}{C1}}%
\VCPut{(2.375\lga,3\lgb)}{\State[2]{(0,0)}{D0}}%
\VCPut{(2.375\lga,4\lgb)}{\State[1]{(0,0)}{E0}}%
\VCPut{(2.375\lga,5\lgb)}{\State[0]{(0,0)}{F0}}%
\ChgEdgeLabelScale{.8}
\CLoopSW[.75]{F0}{0}
\ChgEdgeArrowStyle{-}
\EdgeL{F0}{E0}{1}
\EdgeR{E0}{D0}{0}
\EdgeR{D0}{C0}{0}
\EdgeL{D0}{C1}{1}
\EdgeR{C0}{B0}{0}
\EdgeL{C0}{B1}{1}
\EdgeR{C1}{B2}{0}
\EdgeR{B0}{A0}{0}
\EdgeL{B0}{A1}{1}
\EdgeR{B1}{A2}{0}
\EdgeR{B2}{A3}{0}
\EdgeL{B2}{A4}{1}
\end{VCPicture}}
}
\caption{Tree and i-tree associated with the Fibonacci numeration 
system.}
\label{f.Fib-tre-itr}
\end{figure}

The \emph{degree} of a node in a tree, or in an i-tree, is the number
of its children.
The \emph{signature~$\sigs_{x}$ of a tree~$\Tc_{x}$} is the sequence
of the degrees of the nodes of the associated i-tree~$\Ic_{x}$ in the
breadth-first traversal.
For instance, the signature of~$\Tc_{F}$ is
$\msp\sigs_{F}= 2\xmd 1\xmd 2\xmd 2\xmd 1\xmd 2\xmd 1\xmd 2\xmd 
\cdots\msp$,
the signature of~$\Tc_{p}$ for the numeration in base $p$ is the 
constant sequence 
$\msp\sigs_{p}= p^{\omega}\msp$ for any base~$p>1$.

Conversely, we call \emph{signature} any sequence~$\sigs$ of 
non-negative integers:\\ 
$\msp\sigs=s_0\xmd s_1\xmd s_1\cdots\msp$ and a 
signature is \emph{valid} if the following condition holds:
\begin{equation}
\forall j\in\N \quantsp 
\sum_{\smash{i=0}}^{j} s_{i} > j + 1
\eqpnt
\notag
\end{equation}
Infinite trees and valid signatures are then in a 1-to-1 
correspondence as expressed by the following.

\begin{proposition}[\cite{MarsSaka17a}]
\label{p.sig}%
The signature of an infinite tree is valid and a valid signature is 
the signature of a unique (i-)tree (up to the labeling).
\end{proposition}

By extension, the signature of a (prefix-closed) language~$L$ is the 
signature of the language tree~$\TL$.
The language~$L$ is extendable (or~$\TL$ has no finite branch) if and 
only if its signature contains no \lit{0}.
As said above, \emph{the carry propagation of a \pce language~$L$ is
entirely determined by its signature} which determines the `shape'
of~$\TL$.
In view of the next statements, we have to give two further
definitions.

\begin{definition}
\label{d.rhy-gro}%
Let~$p$ and~$q$ be two integers with~$p>q\geq 1$.

\thi We call a \emph{$q$-tuple}~$\rhth$ of non-negative integers
\emph{whose sum is~$p$} a \emph{rhythm} of directing
parameter~$(q,p)$:
\begin{equation}
  \rhth = \rhthtp 
  \ee\text{and}\ee
  \sum_{\smash{i=0}}^{\smash{\iqmu}} r_i = p
  \eqpnt
  \notag
\end{equation}

\thii
A signature~$\sigs$ is \emph{periodic} if there exists a rhythm~$\rhth$ such 
that $\msp\sigs = \rhth^{\omega}\msp$.

\noindent
A signature~$\sigs$ is \emph{eventually periodic} if there exists a 
rhythm~$\rhth$ such that there exist a finite sequence~$\rhth[t]$ of 
non-negative integers and a rhythm~$\rhth$ such that 
$\msp\sigs = \rhth[t]\xmd\rhth^{\omega}\msp$.
\end{definition}

Languages with periodic signatures were  considered and characterized 
in~\cite{MarsSaka17b} in the study of rational base 
numeration systems that we define as follows.

\begin{example}
\label{e.rat-bas}%
\textbf{The rational base numeration systems.}
Let~$\pq$ be a rational number, where $p > q \geq 1$ are two 
\emph{co-prime} integers.

In~\cite{AkiyEtAl08}, it has been shown how to define a numeration system 
with~$\pq$ as a base and where nevertheless integers have finite 
representations.
Let~$N$ be any positive integer; let us write~$N_0=N$ and, 
for~$i\geq0$, let
\begin{equation}
q\xmd N_i =p\xmd N_{i+1}+a_i
\label{q.rat-bas-1}
\end{equation}
where $a_i$ is the remainder of the division of~$q\xmd N_i$ by $p$,
and thus belongs to the digit-alphabet
$\msp\Ap=\{0,\ldots,\pmu\}\msp$. 
Since $N_{i+1}$ is less than $N_{i}$, the division \equnm{rat-bas-1}
can be repeated only a finite number of times, until eventually
$N_{k+1}=0$ for some~$k$.
This algorithm produces the digits $a_0$, $a_1$, \ldots, $a_k$, and:
\begin{equation}
    N= \sum_{i=0}^{k}\frac{a_i}{q}\xmd \left(\pq\right)^i
	\eqpnt
	\notag
\end{equation}
We will say that the word $\msp a_{k} \cdots a_0\msp$,
computed from $N$ from right to left, that is to say,
\emph{least significant digit first},
is a \emph{$\pq$-expansion} of $N$.
It is known that this representation is indeed unique and we denote 
it by~$\reprpq{N}$.
We define the language~$\Lpq$ of~$\Ape$ as the set of $\pq$-expansions 
of the integers:
\begin{equation}
    \Lpq= \Defi{\reprpq{n}}{n\in\N}
       \notag
\end{equation}
and accordingly, we denote by~$\Tpq$ the tree of the language~$\Lpq$. 
When $q=1$, we recover the usual numeration system in base~$p$
and~$\Lpq=\Lp$; in the following, $q\neq1$.
\figur{T32} shows the case~$p=3$ and $q=2$.
\end{example}

\begin{figure}[htbp]
\centering
\setlength{\lga}{1.4cm}%
\setlength{\lgb}{1.4cm}%
\VCDraw{%
\begin{VCPicture}{(0,0)(3\lga,7)}
\SmallState
\ChgStateLabelScale{.6}
\VCPut{(0,0)}{\State[8]{(0,0)}{A0}%
              \State[9]{(1\lga,0)}{A1}%
              \State[10]{(2\lga,0)}{A2}%
              \State[11]{(3\lga,0)}{A3}}%
\VCPut{(0,\lgb)}{\State[5]{(0,0)}{B0}%
                 \State[6]{(1.5\lga,0)}{B1}%
                 \State[7]{(3\lga,0)}{B2}}%
\VCPut{(0,2\lgb)}{\State[3]{(0,0)}{C0}
                  \State[4]{(2.25\lga,0)}{C1}}%
\VCPut{(1.125\lga,3\lgb)}{\State[2]{(0,0)}{D0}}%
\VCPut{(1.125\lga,4\lgb)}{\State[1]{(0,0)}{E0}}%
\VCPut{(1.125\lga,5\lgb)}{\State[0]{(0,0)}{F0}}%
\ChgEdgeLabelScale{.8}
%
\ChgEdgeArrowStyle{-}
\EdgeL{F0}{E0}{2}
\EdgeR{E0}{D0}{1}
\EdgeR{D0}{C0}{0}
\EdgeL{D0}{C1}{2}
\EdgeR{C0}{B0}{1}
\EdgeR{C1}{B1}{0}
\EdgeL{C1}{B2}{2}
\EdgeR{B0}{A0}{1}
\EdgeR{B1}{A1}{0}
\EdgeL{B1}{A2}{2}
\EdgeR{B2}{A3}{1}
%
\end{VCPicture}}
\caption{The first 6 levels of $\Ttd$.}
\label{f.T32}
\end{figure}

\begin{remark}
This definition \emph{is not} the one corresponding to
$\beta$-expansions with~$\beta=\pq$ (see~\secti{beta-num-gen}).
In particular, the digits \emph{are not} the integers less 
than~$\pq$ but rather the integers less than~$p$, hence those
\emph{whose quotient by~$q$} is less than~$\pq$.
\end{remark}

From the classical theory of formal languages point of view, the
language~$\Lpq$ is complex and difficult to understand.
It can be shown not to meet any kind of iteration property (and thus 
not to be rational nor context-free) and such that any two distinct 
subtrees  of~$\Tpq$ are never isomorphic (\cf~\cite{AkiyEtAl08}).
On the other hand, it is easy to verify the following property that 
expresses a certain kind of `regularity' (and that has indeed been the 
motivation for the definition of signatures).

\begin{proposition}[\cite{MarsSaka17b}]
    The signature of~$\Lpq$ is periodic and its period is a rhythm of 
    directing parameter~$(q,p)$.
\end{proposition}
  Note that having a periodic signature is even a characterization of
rational base numeration systems (possibly using a non-canonical
alphabet), according to \cite[Theorem~2]{MarsSaka17b}.
We now  can state the main result of this section.

\begin{theorem}
\label{t.per-sig}%
If a \pce language~$L$ has an eventually periodic signature with 
rhythm of parameter~$(q,p)$, then~$\CPL$ exists and
\begin{equation}
    \CPL=\frac{p}{p-q}
    \eqpnt
    \notag
\end{equation}
\end{theorem}

\begin{proof}
Let $\msp\sigs = \rhth[t]\xmd\rhth^{\omega}\msp$ be the signature 
of~$\TL$.
Let $\msp\rhth[t]=t_0 t_1\cdots t_{k}\msp$ and
$\msp  \sum_{\smash{i=0}}^{\smash{k}} t_i = P\msp$.
In the following, we always choose~$\ell$ larger than~$\ell_{0}$ such 
that~$\vL{\ell_{0}-1}>P$ and~$N$ larger than~$\vL{\ell_{0}}$, that is, 
we consider nodes and levels of~$\TL$ where the signature is in its 
periodic part.

We first observe that at any given level~$\ell$, the~$q$ leftmost
nodes have~$p$ children at level~$\ell+1$ and moreover for any~$k$
such that~$k\xmd q\leq\uLl$, the~$k\xmd q$ leftmost nodes have~$k\xmd
p$ children at level~$\ell+1$.
\emph{Conversely}, the~$p$ leftmost nodes at level~$\ell$ are the 
children of the~$q$ leftmost nodes at level~$\ell-1$ and for any~$k$ 
such that~$k\xmd p\leq\uLl$, 
the~$k\xmd p$ leftmost nodes  at level~$\ell$ are the children of 
the~$k\xmd q$ leftmost nodes at level~$\ell-1$.

The first observation implies that for every~$\ell$ (greater 
than~$\ell_{0}$), we have:
\begin{alignat}{2}
p\xmd\Inte{\frac{\uLl}{q}}
   &\leq \xmd&\uL{\ell+1}&\leq 
       p\xmd\Inte{\frac{\uLl}{q}} + (p-1)
       \EqVrgInt          
       \notag\\[.5ex]
\text{hence}\ee\e   
\frac{p}{q}\xmd\uLl - p
   &\leq&\uL{\ell+1}&\leq 
       \frac{p}{q}\xmd\uLl + (p-1)
       \eqpnt
       \eee
       \notag
\end{alignat}
And since
$\msp \lim_{\ell\to\infty}\uLl = +\infty\msp$, 
it follows that
\begin{equation}
    \lim_{\ell\to\infty}\frac{\uL{\ell+1}}{\uLl} = \frac{p}{q}
    \eqpnt
\label{q.sig-per-0}
\end{equation}

We then take the same notation as in the proof of \propo{int-bas}:
let~$N$ in~$\N$, $u=\repr{N}$, and~$\ell=|u|$; 
then $\msp\vL{\ell-1}\leq N<\vLl\msp$. 
As in \corol{sum-cp}:
\begin{equation}
\scp{N} = 
\sum_{i=0}^{\ell-1} \vL{i} + \jsCard{\LB{\Succ{u}}}
\eqpnt
\label{q.sig-per-1}
\end{equation}
\noindent 
From \equnm{sig-per-0} and \lemme{cal-cla}, it follows that
$\msp\lim_{\ell\to\infty}\frac{\vL{\ell+1}}{\vLl}=\frac{p}{q}\msp$
and then:
\begin{gather}
\lim_{\ell\to\infty}\frac{1}{\vL{\ell-1}}\sum_{i=0}^{\ell-1}\vL{i}
=\frac{p}{p-q}
\notag\\
\intertext{which can be written as:}
\sum_{i=0}^{\ell-1}\vL{i} =
\vL{\ell-1}\left(\frac{p}{p-q}+\epsilon(\ell)\right)
\ee\text{with }
\lim_{\ell\to\infty}\xmd\epsilon(\ell)=0
\eqpnt 
\label{q.sig-per-2}
\end{gather}

Let~$v=\Succ{u}$. 
The evaluation of~$\jsCard{\LB{v}}$ goes as follows:
the case where~$N=\vLl-1$ and~$|v|=\ell+1$
(which corresponds to $\scp{N} = \sum_{i=0}^{\ell} \vL{i}$)
is excluded; we write 
$N=\vL{\ell-1} + (M-1)$
with
$1\leq M<\uL{\ell}$ and: 
\begin{equation}
\jsCard{\LB{v}\cap A^{\ell}}=M
    \eqpnt
\notag
\end{equation}
The second observation above implies then the evaluation at 
level~$\ell-1$:
\begin{alignat}{2}
\frac{q}{p}\xmd M - q <
q\xmd\Inte{\frac{M}{p}}
&\leq\xmd&\jsCard{\LB{v}\cap A^{\ell-1}}&\leq 
q\xmd\Inte{\frac{M}{p}} + (q-1) < \frac{q}{p}\xmd M +q
\eqvrg
\notag\\
\intertext{at level~$\ell-2$:}
\left(\frac{q}{p}\right)^{\!2}\!M - q\xmd\frac{q}{p} - q
&<&\jsCard{\LB{v}\cap A^{\ell-2}}&<
\left(\frac{q}{p}\right)^{\!2}\!M + q\xmd\frac{q}{p} + q
\eqvrg
\notag\\
\intertext{and at level~$\ell-k$:}
\left(\frac{q}{p}\right)^{\!k}\!M 
   - q \sum_{i=1}^{k}\left(\frac{q}{p}\right)^{\!k-i}
&<&\jsCard{\LB{v}\cap A^{\ell-k}}&<
\left(\frac{q}{p}\right)^{\!k}\!M 
   + q \sum_{i=1}^{k}\left(\frac{q}{p}\right)^{\!k-i}
\eqpnt
\notag
\end{alignat}
Let us write~$B=\bigcup_{j=\ell_{0}}^{\ell}A^{j}$ 
and~$h=\ell-\ell_{0}$.
The summation of the above inequalities from~$k=0$ 
to~$k=h$ yields the following lower and upper bounds (after some 
simplifications):
%
%
\begin{equation}
\frac{p}{p-q}\left[
   M\xmd\left(1-\left(\frac{q}{p}\right)^{\!h+1}\right)- q\xmd h 
             \right] <
\jsCard{\LB{v}\cap B} <
\frac{p}{p-q}\left[M +q\xmd h \right] 
\eqpnt
\notag
\end{equation}
%
Let~$Q=\sum_{i=0}^{\ell_{0}-1} \vL{i}$. 
We bound~$\jsCard{\LB{v}\cap A^{<\ell_{0}}}$ from below by~$0$ and 
from above by~$Q$
and we get then:
\begin{equation}
\frac{p}{p-q}\left[
   M\xmd\left(1-\left(\frac{q}{p}\right)^{\!h+1}\right)- q\xmd h 
             \right] <
\jsCard{\LB{v}} <
\frac{p}{p-q}\left[M +q\xmd h \right] + Q
\eqpnt
\notag
\end{equation}
%
As in the proof of \propo{int-bas}, these inequalities together  
with~\equnm{sig-per-2} yields
\begin{multline}
(N+1)\frac{p}{p-q}+\bfvp{\ell-1}\xmd\epsilon(\ell)
-M\frac{q}{p-q}\left(\frac{q}{p}\right)^{\!h} 
-q\xmd h\frac{p}{p-q} <\\
\scp[p]{N}<
(N+1)\frac{p}{p-q}+\bfvp{\ell-1}\xmd\epsilon(\ell)
+ Q +q\xmd h\frac{p}{p-q}
\EqPnt
\notag
\end{multline}
If we divide by~$N$, both the lower and upper bounds tend 
to~$\frac{p}{p-q}$ when~$N$ tends to infinity, 
hence~$\frac{1}{N}\scp[p]{N}$ has a limit, and this limit 
is~$\frac{p}{p-q}$.
\end{proof}



\section{The carry propagation of rational languages:
        \protect\\ \eee \eee \e 
	 an algebraic point of view}
\label{s.car-pro-rat}%

Even in the case of rational (\pce) languages, the existence of the
local growth is not sufficient to insure the existence of the carry
propagation, but we could say it is `almost' sufficient.
Recall that if~$L$ is a language of~$\Ae$ and~$w$ a word of~$\Ae$, 
the \emph{quotient} of~$L$ by~$w$ is the language
$\msp w^{-1}L=\Defi{v\in\Ae}{w\xmd v\in L}\msp$
and that a language is rational if and only if it has a finite 
number of distinct quotients (see for instance~\cite{Saka09}, 
\cite{Rigo14} or any book on formal language theory).
The aim of this section is the proof of the following result.

\begin{theorem}
\label{t.cp-rat-lan}
Let~$L$ be a rational \pce language with local growth rate~$\igrr$.
If the local growth rate of every quotient 
of~$L$ exists, then the carry propagation~$\CAPR$ exists and is 
equal to $\displaystyle{\frac{\igrr}{\igrr-1}}$.
\end{theorem}

We prove indeed the existence of carry propagation for rational 
languages under somewhat more general hypotheses, the statement of 
which is more technical and requires some developments 
(\theor{a-dev-cp}). 
In any case, rationality \emph{does not} imply the existence of the
local growth rate, as seen with the example below.

\begin{example}
\label{e.sqr-4}
Let
$\msp K_{1}=(\{a\}\{a,b,c,d\})^*\{a,\varepsilon\}\msp$
be the rational \pce language of~$\{a,b,c,d\}^{*}$
accepted by the automaton~$\Ac_{1}$ in \figur{sqr-4a}.

We have: $\wlg[K_{1}]{0}=1$,
$\wlg[K_{1}]{2\ell+1}=\wlg[K_{1}]{2\ell}$ and
$\wlg[K_{1}]{2\ell+2}=4\xmd\wlg[K_{1}]{2\ell+1}$, hence~$\igrr[K_{1}]$
does not exist.
\end{example}

\begin{figure}[h]
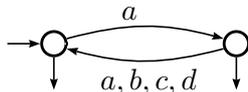

\setlength{\lga}{4cm}
\centering
\VCDraw{%
\begin{VCPicture}{(0,-1)(\lga,1)}
\SmallState
\State{(0,0)}{A}\State{(\lga,0)}{B}
\Initial[w]{A}\Final[s]{A}\Final[s]{B}
\ArcL{A}{B}{a}
\ArcL[.48]{B}{A}{a,b,c,d}
\end{VCPicture}}%
\caption{The minimal automaton~$\Ac_{1}$ of 
$K_{1}=(\{a\}\{a,b,c,d\})^*\{a,\epsilon\}$.}
\label{f.sqr-4a}
\end{figure}

The proof of \theor{cp-rat-lan} and of other results of the same kind 
goes in two main steps.
We first prove with \theor{lgr-a-dev} in \secti{lgr-dom} that the
local growth rate of a rational language exists if and only if the
language has an \emph{almost dominating eigenvalue}, as defined in
\secti{gen-fun-dom}.
In \secti{dom-cp}, we prove that if~$L$ has an almost dominating eigenvalue,
then the carry propagation of~$L$ exists under some additional
hypotheses on the eigenvalues of the quotients of~$L$
(\theor{a-dev-cp}). \theor{cp-rat-lan} is just a corollary of \theor{a-dev-cp}.

In \secti{beta-num-gen} we will see that the hypothesis of
\theor{cp-rat-lan} are fulfilled in the case of the so-called Parry
beta-numeration (\corol{bet-num}).

\subsection{Generating functions and dominating eigenvalues} 
\label{s.gen-fun-dom}

Let~$L$ be a language of~$\Ae$.
The \emph{generating function} of~$L$, $\genz$, is the (formal power)
series in one indeterminate whose $\ell$-th coefficient is the number
of words of~$L$ of length~$\ell$, that is, with our notation:
\begin{equation}
    \genz = \sum_{\ell=0}^{\infty} \uLl\xmd z^{\ell}
\eqpnt
\notag
\end{equation}

Let~$L$ be a rational language of~$\Ae$ 
and~$\Ac=\auta$ a 
deterministic automaton of `dimension'~$Q$  
that accepts~$L$: $L=\CompAuto{\Ac}$.
We identify~$I$ and~$T$, subsets of~$Q$, with their characteristic 
functions in~$\N$, and we write them as vectors of 
dimension~$Q$, respectively row- and column-vectors:
\begin{equation}
    \fa p\in Q\quantsp
    I_{p} = 
\begin{cases}
        \; 1\EqVrg & \text{if\ $\msp p\msp$ is initial}; \\
        \; 0\EqVrg & \text{otherwise};
\end{cases}
\eee
    T_{p} = 
\begin{cases}
        \; 1\EqVrg & \text{if\ $\msp p\msp$ is final}; \\
        \; 0\EqVrg & \text{otherwise}.
\end{cases}
\notag
\end{equation}

The \emph{adjacency matrix}~$\adj$ of~$\Ac$ is the $Q\x Q$-matrix the
$(p,q)$-entry of which is the \emph{number} of transitions in~$\Ac$
that go from state~$p$ to state~$q$ (that is, the entries of~$\adj$
are in~$\N$):
\begin{equation}
    \fa p,q\in Q\quantsp
    (\adj)_{p,q} = \jsCard{\Defi{a\in A}{(p,a,q)\in E}}
\eqpnt
\notag
\end{equation}
Since~$\Ac$ is deterministic (the hypothesis `unambiguous' would
indeed be sufficient), the adjacency matrix allows the computation 
of~$\uLl$, the number of words of~$L$ of length~$\ell$, as:
\begin{equation}
    \fa \ell\in\N\quantsp
    \uLl = I\matmul(\adj)^{\ell}\matmul T
\eqpnt
\eee
\notag
\end{equation}
That is, $\genz$ is an \emph{$\N$-rational series} since the above
equation precisely states that it is realized by the representation
$\ImuT$, with~$\mu(z)=\adj$.
The $\N$-rationality of~$\genz$ implies a number of properties which
eventually allow us to establish \theor{a-dev-cp} and then 
\theor{cp-rat-lan}. 

The semiring~$\N$ is embedded in the \emph{field}~$\Q$ (and, further 
on, in the \emph{algebraically closed field}~$\C$) and in the 
remaining of the subsection, we essentially derive an expression of 
the coefficients~$\uLl$ from the fact that~$\genz$ is a 
$\Q$-rational series (or even a $\C$-rational series).
The very special properties of rational series with non-negative
coefficients come into play in the next subsection.
We rely on the treatise~\cite{BersReut11} of
Berstel--Reutenauer (Sec.~6.1,~6.2,~8.1, and~8.3) for this exposition.

The Cayley--Hamilton Theorem implies that the sequence~$\sequ{\uLl}$ 
satisfies the \emph{linear recurrence relation} defined by the 
\emph{characteristic polynomial}~$\pol[\Ac]$ of~$\adj$, the zeroes 
of which are the \emph{eigenvalues} of~$\adj$.
The sequence~$\sequ{\uLl}$ satisfies indeed a \emph{shortest} linear 
recurrence relation associated with a polynomial~$\pol$, the 
\emph{minimal polynomial} of~$\genz$, which is a divisor 
of~$\pol[\Ac]$.
The zeroes of~$\pol$ are called \emph{the eigenvalues 
of}~$\genz$ and \emph{of~$L$}.
The \emph{multiplicities} of these eigenvalues are those of these 
zeroes.

\begin{definition}
\label{d.lan-dom-eig}
We call the maximum of the moduli of the eigenvalues of a rational 
language~$L$ the \emph{modulus} of~$L$.
It is the multiplicative inverse of the radius of convergence of the 
series~$\genz$. 

A rational language $L$ is said to have a \emph{dominating
eigenvalue}, or, for short, \emph{to be \dev}, if there is, among the
eigenvalues of~$L$, a unique eigenvalue of maximal modulus, which is
called \emph{the} dominating eigenvalue of~$L$.
\end{definition}

With the next two examples, we stress the difference between
the eigenvalues of the adjacency matrix of an automaton~$\Ac$ that
recognizes~$L$ and the eigenvalues of $L$.

\begin{example}[\examp{sqr-4} continued]
\label{e.dom-mod-2}%
The adjacency matrix of~$\Ac_{1}$
shown in \figur{sqr-4a} is 
$\msp\adj[\Ac_{1}]=\begin{pmatrix} 0 & 1 \\ 4 & 0\end{pmatrix}\msp$,
its characteristic polynomial is 
$\msp\pol[\Ac_{1}]=X^{2}-4\msp$,
the zeroes of which are~$2$ and~$-2$.
This polynomial is also the minimal polynomial of the linear
recurrence satisfied by the coefficients of~$\gen[K_{1}]{z}$:
\begin{gather}
  \wlg[K_{1}]{0}=1\EqVrgInt
  \wlg[K_{1}]{1}=1\EqVrgInt
  \wlg[K_{1}]{\ell+2}=4\xmd\wlg[K_{1}]{\ell}
\eqvrg
\notag\\
\text{hence}
\ee\eee
\fa\ell\geq 0\quantsp  
  \wlg[K_{1}]{\ell}= \frac{3}{4}\xmd2^{\ell} + \frac{1}{4}\xmd(-2)^{\ell}
\eqvrg
\eee\eee\eee
\notag
\end{gather}
and~$K_{1}$ is thus not~\dev.
\end{example}

\begin{example}
\label{e.dom-eig-2}
The adjacency matrix of the automaton~$\Ac_{2}$ in \figur{sqr-4} 
is\linebreak 
$\msp\adj[\Ac_{2}]=\begin{pmatrix} 0 & 2 \\ 2 & 0 \end{pmatrix}\msp$,
its characteristic polynomial is
$\msp \pol[\Ac_{2}]=X^{2}-4\msp$ as above.
But in this case, the \emph{minimal polynomial of the linear 
recurrence} 
satisfied by the coefficients of~$\gen[K_{2}]{z}$,
$\wlg[K_{2}]{\ell}= 2^{\ell}$,
is~$\pol[K_{2}]=X-2$, with~$2$ as a unique zero and~$K_{2}$ is~\dev. 
\end{example}

\begin{figure}[h]
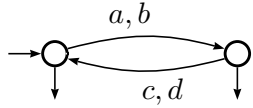

\setlength{\lga}{4cm}
\centering
\VCDraw{%
\begin{VCPicture}{(0,-1)(\lga,1)}
\SmallState
\State{(0,0)}{A}\State{(\lga,0)}{B}
\Initial[w]{A}\Final[s]{A}\Final[s]{B}
\ArcL{A}{B}{a,b}
\ArcL{B}{A}{c,d}
\end{VCPicture}}%
\caption{The minimal automaton~$\Ac_{2}$ of 
$K_{2}=(\{a,b\}\{c,d\})^*\{a,b,\epsilon\}$.} 
\label{f.sqr-4}
\end{figure}

As stated in~\cite{BersReut11}, the rational function~$\genz$ may be 
written, in a unique way,~as: 
\begin{equation}
    \genz = T(z) + \frac{R(z)}{S(z)}
    \notag
\end{equation}
where~$T(z)$, $R(z)$ and~$S(z)$ are polynomials in~$\QPz$, 
$\deg R < \deg S$ and~$S(0)\not=0$.
It can be shown that~$\pol$ is the \emph{reciprocal polynomial}
of~$S$: $\msp\pol(z)=S(\frac{1}{z})\xmd z^{\deg S}\msp$.
It follows that if~$\lambda_{1}$, 
$\lambda_{2}$,\ldots,~$\lambda_{t}$ are the zeroes of~$\pol$, the 
coefficients~$\uLl$ of~$\genz$ can be written as:
\begin{equation}
    \fa\ell\in\N\quantsp
    \uLl = \sum_{j=1}^{t} \lambda_{j}^{\ell}\xmd P_{j}(\ell)
    \eqvrg
    \eee
    \label{q.lin-rec}
\end{equation}
where every~$P_{j}$ is a polynomial (which depends on~$L$ even though
it does not appear in the writing) 
\emph{whose degree is equal to the multiplicity of the 
zero~$\lambda_{j}$ in~$\pol$ minus~$1$}, and is
determined by the first values of the sequence~$\msp\sequ{\uLl}\msp$.

\subsection{From local growth rate to dominating eigenvalue}
\label{s.lgr-dom}%

The properties of rational series with positive coefficients allow us
to characterize the generating functions of rational languages with
local growth rate.
Let us first recall the theorem due to Berstel on such series.

\begin{theorem}[Theorem~8.1.1 and Lemma~8.1.2 in \cite{BersReut11}]%
\label{t.pos-rat-eig}%
~

\noindent 
Let~$f(z)$ be an $\R_{+}$-rational function which is not a polynomial 
and~$\lambda$ the maximum of the moduli of its eigenvalues.
Then:

\tha $\lambda$ is an eigenvalue of~$f(z)$ (hence an eigenvalue 
in~$\R_{+}$).

\thb Every eigenvalue
of~$f(z)$ of modulus~$\lambda$ is of the
form~$\msp\lambda\xmd\eulnum^{i\xmd\theta}\msp$
where~$\eulnum^{i\xmd\theta}$ is a root of unity.

\thc The multiplicity of any eigenvalue of modulus~$\lambda$ is at 
most that of~$\lambda$.
\end{theorem}

We express the consequences of this result (and of~\equat{lin-rec}) 
in the following way.
Let~$L$ be a rational language and~$\genz$ its generating function, 
an $\R_{+}$-rational function.
Let~$\lambda=\lambda_{1}$, $\lambda_{2}$,\ldots, $\lambda_{k}$ be the
eigenvalues of maximal modulus~$\lambda$ of~$\genz$ and~$d$ the degree
of the polynomial~$P_{1}$ in~\equnm{lin-rec}.
For $j=\nobreak 2,\ldots, k$, we write
$\msp\displaystyle{\lambda_{j}=
   \lambda\xmd\eulnum^{i\xmd\theta_{j}}}\msp$;
$\displaystyle{\eulnum^{i\xmd\theta_{j}}}$ is a root of unity, 
hence~$\theta_{j}= 2\xmd\pi/h_{j}$ where~$h_{j}$ is an integer and 
let~$r$ be the least common multiple of all the~$h_{j}$.

There exist~$k$ (possibly complex) numbers~$\delta_{1}$,
$\delta_{2}$,\ldots,~$\delta_{k}$, with~$\delta_{1}$ in~$\R$ and
different from~$0$, such that~\equnm{lin-rec} can be given the
following asymptotic form:
\begin{equation}
\fa\ell \text{ large enough}\quantsp
\uLl =
\lambda^{\ell}\xmd\ell^{d}\xmd 
\left(\delta_1 +
   \sum_{j=2}^k \delta_j\xmd\eulnum^{i\xmd\ell\xmd\theta_j}\right) + 
   \petito{\lambda^{\ell}\xmd\ell^{d}}
\eqpnt
\label{q.asy-eva}
\end{equation}
(It is understood that~$\delta_j$ is not zero if the 
polynomial~$P_{j}$ in~\equnm{lin-rec} is of degree~$d$, it is equal 
to~$0$ if this polynomial is of degree less than~$d$ --- and by 
\theor{pos-rat-eig}~(c) no such polynomial has degree greater 
than~$d$.)

One can say that the description of~$\uLl$ given in~\equnm{lin-rec} 
is \emph{ordered by eigenvalues} whereas  the description 
in~\equnm{asy-eva} is \emph{ordered by moduli} of eigenvalues. 
Since for every $j=\nobreak 2,\ldots, k$ and every~$p$ in~$\N$ we have
$\displaystyle{\eulnum^{i\xmd{p\xmd r}\theta_{j}}=1}$, it follows that
\begin{equation}
\lim_{p\rightarrow\infty}\xmd
      \frac{\uL{p\xmd r}}{\lambda^{p\xmd r}\xmd(p\xmd r)^{d}}\xmd 
        = \sum_{j=1}^k \delta_j
\eqpnt
\notag
\end{equation}
%

\begin{definition}
\label{d.a-dev}%
Let~$f(z)$ be an $\R_{+}$-rational function which is not a polynomial 
and~$\lambda$ the maximum of the moduli of its eigenvalues.
We say that~$f(z)$ has an \emph{almost dominating eigenvalue}, or 
is~\adev, if the \emph{multiplicity} of any non-real eigenvalue of 
modulus~$\lambda$ is \emph{strictly less} than that of the 
eigenvalue~$\lambda$. 
\end{definition}

Accordingly, we say that a rational language~$L$ is~\adev if~$\genz$
is~\adev.
Using the above notation, $L$~is \adev if and only if 
all the~$\delta_{j}$, $j=\nobreak2,\ldots, k$,  but~$\delta_{1}$ 
in~\equnm{asy-eva} are equal to~$0$, that is,  
if and only if~\equnm{asy-eva} takes the following form:
\begin{equation}
\fa\ell \text{ large enough}\quantsp
\uLl =
\lambda^{\ell}\xmd\ell^{d}\xmd\delta_1 + 
     \petito{\lambda^{\ell}\xmd\ell^{d}}
\eqpnt
\eee\eee
\label{q.asy-eva-2}
\end{equation}

\begin{example}
\label{e.a-dev}
The rational \pce language~$K_{3}$ accepted by the 
automaton~$\Ac_{3}$ shown in \figur{a-dev} is~\adev without 
being~\dev.  
The zeroes of the characteristic polynomial of~$\adj[\Ac_{3}]$ 
are~$2$ with multiplicity~$2$ and~$-2$ (with multiplicity~$1$).
The zeroes of~$\pol[K_{3}]=(X^{2}-4)(2-X)$ are the same as we have 
\begin{equation}
\fa\ell\in\N\quantsp
\uLl =
\left(\frac{1}{4}\xmd\ell + \frac{7}{8}\right)\xmd 2^{\ell} 
+ \frac{1}{8}\xmd(-2)^{\ell}
\eqpnt
\eee
\label{q.a-dev}
\end{equation}
On the other hand, the language~$K_{1}$ from \exemp{sqr-4}, which is
not~\dev, is not~\adev either.
\end{example}

\begin{figure}[h]
\setlength{\lga}{4cm}
\ifelsevier\else\medskipneg\fi
\centering
\VCDraw{%
\begin{VCPicture}{(-\lga,-1)(\lga,1)}
\SmallState
\State{(0,0)}{A}\State{(\lga,0)}{B}\State{(-\lga,0)}{C}
\Initial[n]{A}\Final[s]{A}\Final[s]{B}\Final[s]{C}
\ArcL{A}{B}{a,b}
\ArcL{B}{A}{c,d}
\EdgeR{A}{C}{c}
\LoopW{C}{a,b}
\end{VCPicture}}%
\caption{The minimal automaton~$\Ac_{3}$ of 
$K_{3}=(\{a,b\}\{c,d\})^*\{a,b,\epsilon\}
\xmd\cup\xmd c\xmd\{a,b\}^*$.} 
\label{f.a-dev}
\end{figure}
\ifelsevier\else\medskipneg\fi

\begin{theorem}
\label{t.lgr-a-dev}
A rational language~$L$ is \adev if and only if the local growth 
rate~$\igrr$ exists. 
In this case, the modulus of~$L$ is equal to~$\igrr$.
\end{theorem}

\begin{proof}
If~$L$ is~\adev, 
the asymptotic expression~\equnm{asy-eva-2} shows that the condition 
is sufficient since
\begin{equation}
\frac{\uL{\ell+1}}{\uLl} = 
\lambda\xmd\left(\frac{\ell+1}{\ell}\right)^{d}\xmd(1+\petito{1})
\notag
\end{equation}
implies that
\begin{equation}
    \lim_{\ell\rightarrow\infty}\frac{\uL{\ell+1}}{\uLl} =
    \lambda
    \eqvrg
    \label{q.dom-eig-3}
\end{equation}
which states both that~$\igrr$ exists and is equal to~$\lambda$.

Conversely, let us suppose that the limit of
$\displaystyle{\frac{\uL{\ell+1}}{\uLl}}$
exists and is equal to~$\igrr$ when~$\ell$ tends to infinity.
For the ease of writing, and in view of the use of~\equnm{asy-eva}, 
let us set:
\begin{equation}
\wea{\ell} = \delta_1 +
   \sum_{j=2}^k \delta_{j}\xmd\eulnum^{i\xmd\ell\xmd\theta_{j}}
\eqpnt
\notag
\end{equation}
The function~$\wea{\ell}$ is periodic of period~$r$, and for every 
integer~$s$, $0\leq s<r$, the hypothesis, and~\equnm{asy-eva}, imply
\begin{multline}
\lim_{p\rightarrow\infty} \frac{\uL{p\xmd r+s+1}}{\uL{p\xmd r+s}}=\\
\lim_{p\rightarrow\infty} 
\left(\lambda\left(\frac{p\xmd r+s+1}{p\xmd r+s}\right)^{d}
    \frac{\wea{s+1}}{\wea{s}}\xmd(1+\petito{1})\right) =
\lambda\xmd\frac{\wea{s+1}}{\wea{s}}=\igrr
\eqpnt
\notag
\end{multline}
Hence, there exists an~$x$ in~$\R_{+}$ such that
\begin{equation}
\fa s\in\N\quantvrg 0\leq s<r\quantsp
\frac{\wea{s+1}}{\wea{s}} =x
\eqpnt
\eee\eee
\notag
\end{equation}
Moreover, since 
$\wea{0}=\wea{r}$ and
\begin{equation}
\frac{\wea{0}}{\wea{0}} =
\frac{\wea{1}}{\wea{0}}\xmd\frac{\wea{2}}{\wea{1}}\cdots
\frac{\wea{0}}{\wea{r-1}} = x^{r}=1
\eqvrg 
\notag
\end{equation}
it follows that $\msp x=1\msp$,
$\msp\lambda=\igrr\msp$ and
\begin{equation}
\fa s\in\N\quantvrg 0\leq s<r\quantsp
\wea{s} = \wea{0} = \delta_{1} =
\delta_1 + \sum_{j=2}^k \delta_{j}\xmd\eulnum^{i\xmd s\xmd\theta_{j}}
\eqpnt
\eee\eee
\notag
\end{equation}
We conclude that the vector
$\begin{pmatrix}
0 & \delta_2 & \cdots & \delta_k
\end{pmatrix}$
is a solution of the Vandermonde linear system (for the sake of 
completeness, we set~$\theta_1=0$): 
\begin{equation}
\begin{pmatrix}
1&1&\cdots &1\\
\eulnum^{i\xmd\theta_1}& \eulnum^{i\xmd\theta_2}& 
                     \cdots & \eulnum^{i\xmd\theta_k}\\
\eulnum^{i\xmd2\xmd\theta_1}& \eulnum^{i\xmd2\xmd\theta_2}& 
                     \cdots & \eulnum^{i\xmd2\xmd\theta_k}\\
\vdots & \vdots & &\vdots \\
\eulnum^{i\xmd(k-1)\xmd\theta_1}& \eulnum^{i\xmd(k-1)\xmd\theta_2}& 
                     \cdots & \eulnum^{i\xmd(k-1)\xmd\theta_k}
\end{pmatrix}
\begin{pmatrix}
\zeta_1\\ \zeta_2\\ \vdots\\ \zeta_k
\end{pmatrix}=
\begin{pmatrix}
    0\\ 0\\ \vdots\\ 0
\end{pmatrix}
\label{q.lgr-dev-3}
\end{equation}
hence identically zero: 
all~$\delta_{j}$, $j=\nobreak2,\ldots, k$, are equal to~$0$ and~$L$ 
is~\adev.
\end{proof} 

\subsection{From dominating eigenvalue to the carry propagation}
\label{s.dom-cp}

With the notions of modulus and of dominating eigenvalue of a rational 
language, we can now state a result that is more general and of which 
\theor{cp-rat-lan} is an obvious corollary.
\begin{theorem}
\label{t.a-dev-cp}%
Let~$L$ be an \adev rational \pce language and~$\lambda$ its modulus. 
If every quotient of~$L$ whose modulus is equal to~$\lambda$ is \adev, 
then~$L$ has a carry propagation and 
$\msp\displaystyle{\CPL=\frac{\lambda}{\lambda-1}}\msp$.
\end{theorem}

Indeed, previous results  show that if a rational \pce language~$L$ 
is~\adev and of modulus~$\lambda$, then~$\igrr$ exists 
and~$\msp\igrr=\lambda\msp$ and if the carry propagation~$\CPL$ exists, 
then~$\msp\displaystyle{\CPL=\frac{\lambda}{\lambda-1}}\msp$. 
The hypothesis on the quotients of~$L$ is necessary
as shown by the following example.

\begin{example}
\label{e.ctr-exp-js}%
Let~$K_{4}$ be the language accepted by the 
automaton~$\Ac_{4}$ shown in 
\figur{ctr-exp-js}.
It has first the property of being a \dev (and not only an~\adev) 
language, of modulus~$2$.

On the other hand, 
$\msp K_{4}=\epsilon\cup a\xmd K_{1}\cup b\xmd K_{1}\cup c\xmd 
K'_{1}\msp$
where~$K_{1}$ is the language of \examp{sqr-4} and~$K'_{1}$ 
the one accepted by the automaton~$\Ac'_{1}$ obtained from the 
automaton~$\Ac_{1}$ of \figur{sqr-4a} by changing the initial state.
The language~$K_{1}$ is a quotient of~$K_{4}$: $K_{1}=a^{-1}K_{4}$; it
has modulus~$2$ and \emph{is not} \adev.

We have seen that
$\bfu[K_{1}]{\ell}= \frac{3}{4}\xmd2^{\ell} + 
\frac{1}{4}\xmd(-2)^{\ell}$;
similarly
$\bfu[K'_{1}]{\ell}= \frac{3}{2}\xmd2^{\ell} - 
\frac{1}{2}\xmd(-2)^{\ell}$.
Hence
$\msp\bfu[K_{4}]{0}=1\msp$ and
$\msp\bfu[K_{4}]{\ell+1}=2\xmd\bfu[K_{1}]{\ell}+\bfu[K'_{1}]{\ell}=
3\cdot 2^{\ell}\msp$.
From which one deduces
$\msp\bfv[K_{4}]{0}=1\msp$ and 
$\msp\bfv[K_{4}]{\ell+1} = 3\cdot 2^{\ell+1} -2\msp$ and
\begin{equation}
\bfw[K_{4}]{\ell+1}= 1 + \sum_{j=0}^{\ell} \bfv[K_{4}]{j+1}
   = 3\cdot 2^{\ell+2} -2\xmd\ell -7
\eqpnt
\notag
\end{equation}

\begin{figure}[h]
\setlength{\lga}{4cm}
\centering
\VCDraw{%
\begin{VCPicture}{(-1\lga,-1.1)(2\lga,1.6)}
\VCPut{(-1.2\lga,0)}{\State[p]{(0,0)}{A}\State[q]{(\lga,0)}{B}}
\VCPut{(1.2\lga,0)}{\State[r]{(0,0)}{C}\State[s]{(\lga,0)}{D}}
\State[i]{(.5\lga,.3\lga)}{O}
\Initial[n]{O}\Final[s]{O}
\Final[s]{A}\Final[s]{B}
\Final[s]{C}\Final[s]{D}
\ArcR{O}{A}{a,b}\ArcL{O}{C}{c}
\ArcL[.7]{A}{B}{a}\ArcL{B}{A}{a,b,c,d}
\ArcL{D}{C}{a}\ArcL{C}{D}{a,b,c,d}
\end{VCPicture}}%
\caption{The automaton~$\Ac_{4}$}
\label{f.ctr-exp-js}
\end{figure}

We show that $\CPL[K_{4}]$ does not exist with the same argument as
the one developed in \examp{un-bal}.
We choose a sequence of words~$u_{\ell}$ and then a sequence of 
numbers~$N(\ell)=\val[K_{4}]{u_{\ell}}$ and show that the ratio
$\scp[K_{4}]{N(\ell)}/N(\ell)$ does not have~$2$ as limit (and even 
that it has no limit).

We choose the words in
$\msp b\xmd\Maxlg[b^{-1}K_{4}]= \Maxlg[b\xmd K_{1}]\msp$.
It follows that
\begin{gather}
N(\ell+2)=\bfv[K_{4}]{\ell+1}+2\xmd\bfu[K_{4}]{\ell+1}
\ee\text{and}
\notag\\
\begin{split}
\scp[K_{4}]{N(\ell+2)}&=\bfw[K_{4}]{\ell+1}+
    \jsCard{\LB[K_{4}]{\Succ[K_{4}]{\repr[K_{4}]{N(\ell+2)}}}}\\
	&=\bfw[K_{4}]{\ell+1}+2\xmd\bfv[K_{4}]{\ell+1}
\eqpnt
\notag
\end{split}
\end{gather}

\noindent
We then have:
\begin{gather}
	 N(2\xmd k+2)= 2\cdot 2^{2\xmd k+2} -2
\ee\text{and}\ee
	  N(2\xmd k+3)= \frac{5}{2}\xmd 2^{2\xmd k+3} -2
\eqvrg
\notag\\
\scp[K_{4}]{N(2\xmd k+2)} \sim  \frac{13}{3}\xmd 2^{2\xmd k+2}
\ee\text{and}\ee
\scp[K_{4}]{N(2\xmd k+3)} \sim  \frac{14}{3}\xmd 2^{2\xmd k+3}
\eqvrg
\notag\\
\text{hence}\e
\lim_{\ell\to+\infty}\frac{\scp{N(2\ell)}}{N(2\ell)}=\frac{13}{6}
\ee\text{and}\ee
\lim_{\ell\to+\infty}\frac{\scp{N(2\ell+1)}}{N(2\ell+1)}=\frac{28}{15}
\ ,
\notag
\end{gather}
which complete the proof of the non-existence of~$\CPL[K_{4}]$.
\end{example}


The proof of \theor{a-dev-cp} requires a new description of the `left
bank' of a word (recall the definition p.\xmd\pageref{p.car-pro-agg}),
a description for which we introduce some notation.
In order to keep these new symbols readable, we simplify some of those
already in use.

\renewcommand{\uL}[1]{\bfu[]{#1}}%
\renewcommand{\vL}[1]{\bfv[]{#1}}%
\renewcommand{\pol}[1][]{\polop_{#1}}
For the remaining of the section, the \adev rational \pce language~$L$ is 
fixed and kept implicit in most cases:
the number of words of~$L$ of length~$\ell$ (resp.  of length less
than or equal to~$\ell$) is now denoted by~$\uLl$ (resp.  by~$\vLl$)
and the minimal polynomial of~$L$ is now denoted by $\pol$.

Let~$\msp\Ac=\aut{A,Q,q_{0},\delta,T}\msp$ be a deterministic
finite automaton that accepts~$L$ (and which is also kept implicit in 
what follows).
For every~$q$ in~$Q$ and~$w$ in~$\Ae$, we write~$q\act w$ for the
state reached by the computation of~$\Ac$ starting in~$q$ and labeled
with~$w$.
For every state~$q$ in~$Q$, we denote by~$\Lq$ the language accepted 
by the automaton
$\msp\Ac_{q}=\aut{A,Q,q,\delta,T}\msp$,
that is,
$\msp\Lq=\Defi{w\in\Ae}{q\act w\in T}\msp$
and, for every~$\ell$ in~$\N$, by~$\bfu[q]{\ell}$
the number of words of~$\Lq$ of length~$\ell$ and by~$\bfv[q]{\ell}$ 
the number of words of~$\Lq$ of length less than or equal to~$\ell$
(in particular, $\msp\uLl=\bfu[q_{0}]{\ell}\msp$ 
and~$\msp\vLl=\bfv[q_{0}]{\ell}\msp$).\footnote{%
  These definitions hide some technicalities: for~$q\act w$ to be 
  defined for all~$q$ and~$w$, $\Ac$ needs to be not necessarily trim 
  but possibly endowed with a sink state~$s$; then~$L_{s}$ will be empty 
  and~$\bfu[s]{\ell}$ equal to~$0$ for every~$\ell$.}

For every~$q$ in~$Q$, there exists a word~$w_{q}$ 
of length~$\ell_{q}$ such that
$\msp q_{0}\act w_{q} = q\msp$
and then
$\msp L_{q} = w_{q}^{-1}L\msp$: $L_{q}$ is a \emph{quotient} of~$L$.
It then follows
\begin{equation}
\fa q\in Q\quantvrg
\fa\ell\in\N\quantvrg\ell\geq\ell_{q} 
\quantsp
\uL{\ell+\ell_{q}} \geq \bfu[q]{\ell} 
\eqpnt
\eee\eee
\label{q.dev-cp-0}
\end{equation}

If $\msp w=a_1 a_{2}\cdots a_{\ell+1}\msp$ is a word of~$\Ae$, we denote 
by~$\prew{j}$ \emph{the prefix of length~$j$} of~$w$: 
$\msp\prew{j}=a_1a_{2}\cdots a_{j}\msp$;
$\prew{0}=\epsilon$ and~$\prew{\ell+1}=w$. 
(The formulas to come are simpler if the length of~$w$ is written 
as~$\ell+1$ rather than~$\ell$.)
Suppose~$w$ is in~$L$.
The left bank of~$w$, $\LB{w}$, is, by definition, for each
length~$j$, $1\leq j\leq\ell+1$, the set of words of~$L$ of length~$j$
that are less than~$\prew{j}$ in the lexicographic order.
This is a description by \emph{horizontal layers}.
The same set can be given a decomposition by \emph{subtrees}
of~$\TcL$.
For every~$j$, $0\leq j\leq\ell$, the prefix of~$w$ of length~$j+1$
is~$\prew{j+1}=\prew{j}a_{j+1}$.
Then, for every~$j$, $0\leq j\leq\ell$ and for every~$a$, $a<a_{j+1}$,
$\LB{w}$ contains all words of~$L_{q_{0}\act\prew{j}a}$ of length less
than, or equal to,~$\ell-j$ concatenated on the left with~$\prew{j}a$.
Moreover, these subsets form a \emph{partition} of~$\LB{w}$:
\begin{equation}
\LB{w} = \bigcup_{j=0}^{\ell} \left[\bigcup_{a<a_{j+1}} 
             \prew{j}a\xmd\left(L_{q_{0}\act\prew{j}a}\cap A^{\leq\ell-j}
	     \right)\right]
\eqvrg
\label{q.dev-cp-1}
\end{equation}
where the unions are pairwise disjoint and can then be used for 
counting the elements of~$\LB{w}$.

\begin{proof}[Proof of \theor{a-dev-cp}]
Let~$\lambda$ be the modulus of~$L$.
Let~$N$ be an integer and
$\msp\repr{N}=\nobreak w=\nobreak a_1a_{2}\cdots a_{\ell+1}\msp$ its 
$L$-representation (see Sec.~\ref{s.lan-ans}).
By definition, $N$ is equal to the number of words of~$L$ that are 
less than~$w$ in the radix order, that is, in the line\footnote{%
   This is a reformulation of Lemma~3 in \cite{LecoRigo01}.} 
of the 
decomposition~\equnm{dev-cp-1}:
\begin{equation}
N = \vL{\ell} + 
    \sum_{j=0}^{\ell} \left[\sum_{a<a_{j+1}} 
             \bfu[q_{0}\act\prew{j}a]{\ell-j}\right]
\eqpnt
\label{q.dev-cp-2}
\end{equation}
On the other hand, \corol{sum-cp} and~\equnm{dev-cp-1} yield the 
following expression for the sum of the carry propagations at the 
first~$N$ words of~$L$:
\begin{equation}
\scp{N} = \sum_{j=0}^{\ell}\vL{j} + 
    \sum_{j=0}^{\ell} \left[\sum_{a<a_{j+1}} 
             \bfv[q_{0}\act\prew{j}a]{\ell-j}\right]
\eqpnt
\label{q.dev-cp-3}
\end{equation}
By \propo{car-pro-fil} and~\theor{lgr-a-dev}, if the carry 
propagation
\begin{equation}
\CPL=\lim_{N\to\infty}\frac{1}{N}\scp{N}\notag
\end{equation}
of~$L$ exists, it must be equal 
to~$\displaystyle{\frac{\lambda}{\lambda-1}}$. 
We thus evaluate
\begin{equation}
\lim_{N\to\infty}\left(\frac{1}{N}\scp{N}
            -\frac{\lambda}{\lambda-1}\right)
\notag
\end{equation}
and show it exists and is equal to~$0$, using 
both~\equnm{dev-cp-2} and~\equnm{dev-cp-3}.
We write:
\begin{multline}
\frac{1}{N}\left(\scp{N}-\frac{\lambda}{\lambda-1}\xmd N\right)=
\frac{1}{N}\left(\sum_{j=0}^{\ell}\vL{j}
            -\frac{\lambda}{\lambda-1}\vL{\ell}\right)\\
+\frac{1}{N}\left(\sum_{j=0}^{\ell} \left(\sum_{a<a_{j+1}} 
             \left(\bfv[q_{0}\act\prew{j}a]{\ell-j}
            -\frac{\lambda}{\lambda-1}\bfu[q_{0}\act\prew{j}a]{\ell-j}
	    \right)\right)\right)
\eqpnt
\notag
\end{multline}
The two parts of the right-hand side of the equation are evaluated 
separately. 
We first note that~$\ell$ tends to infinity with~$N$ and recall that, 
by~\equnm{dom-eig-3}, 
\begin{equation}
\lim_{\ell\to\infty}\frac{\uL{\ell+1}}{\uLl}= \lambda
\eqpnt
\notag
\end{equation}
Since~$N\ge\vL{\ell}$, we have:
\begin{equation}
\frac{1}{N}\left|\sum_{j=0}^{\ell}\vL{j}
            -\frac{\lambda}{\lambda-1}\vL{\ell}\right| \leq
\left|\frac{1}{\vL{\ell}}\left(\sum_{j=0}^{\ell}\vL{j}\right)
            -\frac{\lambda}{\lambda-1}\right|	    
\notag
\end{equation}
which, by \lemme{cal-cla}, tends to~$0$ when~$\ell$ tends to infinity.

The second term requires some more work.
For the ease of writing, let us set, for every~$q$ in~$Q$ and 
every~$\ell$ in~$\N$,
\begin{equation}
\bfz{\ell} = 
   \left(\bfv[q]{\ell}-\frac{\lambda}{\lambda-1}\bfu[q]{\ell}\right)
\eqpnt
\notag
\end{equation}
The term we have to evaluate reads then
\begin{equation}
\frac{1}{N}\left|\sum_{j=0}^{\ell} \left(\sum_{a<a_{j+1}} 
          \left(\bfz[q_{0}\act\prew{j}a]{\ell-j}\right)\right)\right|
\label{q.dev-cp-4}
\end{equation}
and is (loosely) bounded by
\begin{equation}
    (\jsCard{A}-1)\xmd \frac{1}{\vL{\ell}}\xmd     
    \sum_{j=0}^{\ell} \left(\sum_{q\in Q}\left|\bfz{j}\right|\right)
    \eqpnt 
    \label{q.dev-cp-5}   
\end{equation}
Indeed, the range of every inner sum in~\equnm{dev-cp-4} is a subset of
the alphabet~$A$ made of letters less than a given one, hence every
such sum contains at most $(\jsCard{A}-1)$ terms.
Moreover, we have replaced every term
\begin{equation}
\bfz[q_{0}\act\prew{j}a]{\ell-j}
\ee\text{by the sum}\ee
\sum_{q\in Q}\left|\bfz[q]{\ell-j}\right|
\eqpnt 
\notag
\end{equation}
This is of course a loose upper bound but it allows us to get rid 
of the problem of taking the limit of quantities that are different 
when~$N$ tends to infinity.
Finally, we permute the two summations in~\equnm{dev-cp-5} and it
remains to show that, for every state~$q$ in~$Q$,
\begin{equation}
\lim_{\ell\to\infty}
\msp\xmd\xmd
\frac{1}{\vL{\ell}}\xmd    
\left(\sum_{j=0}^{\ell} \left|\bfz[q]{j}\right|\right)
    =0
\eqvrg 
\label{q.dev-cp-6}   
\end{equation}
and we need to go more into details for that purpose.

For every~$q$ in~$Q$, $L_{q}$ is accepted by~$\Ac_{q}$, the 
accessible part of which is a subautomaton of~$\Ac$.
Hence, the sequence~$\sequ{\bfu[q]{\ell}}$ satisfies a linear
recurrence relation whose minimal polynomial~$\pol[q]$ is, as
is~$\pol$, a factor of~$\pol[\Ac]$.
However, the zeroes of~$\pol[q]$ are not necessarily a subset of those
of~$\pol$ and we base the comparison between~$\sequ{\bfu[q]{\ell}}$
and~$\sequ{\uLl}$ on the asymptotic behaviour.

The series
$\msp\gen[q]{z}=\sum_{\ell=0}^{\infty}\bfu[q]{\ell}\xmd z^{\ell}\msp$
is an $\N$-rational series and, for the same reasons as above, it has 
a real eigenvalue of maximal modulus~$\mu_{q}$ and of 
multiplicity~$d_{q}+1$, and one can write:
\begin{equation}
\fa\ell \text{ large enough}\quantsp
\bfu[q]{\ell} =
\mu_{q}^{\ell}\xmd\ell^{d_{q}}\xmd 
\left(\delta_{q,1} +
   \sum_{j=2}^k \delta_{q,j}\xmd\eulnum^{i\xmd\ell\xmd\theta_{q,j}}\right) + 
   \petito{\mu_{q}^{\ell}\xmd\ell^{d_{q}}}
\eqvrg
\label{q.asy-eva-q}
\end{equation}
where~$d_{q}$, the~$\delta_{q,j}$'s and the~$\theta_{q,j}$'s play the 
same role as~$d$, the~$\delta_{j}$'s and the~$\theta_{j}$'s play in 
\equat{asy-eva}. 

There are two cases: either~$\mu_{q}$ is less  
than~$\lambda$, or~$\mu_{q}$ is equal to~$\lambda$.
It cannot be larger than~$\lambda$ for otherwise~$\bfu[q]{\ell}$ 
would not be bounded by~$\uLl$ (\equat{dev-cp-0}).
In the first case, the quantity
$\msp\displaystyle{\left|\bfz[q]{j}\right| = 
\left|\bfv[q]{j}-\frac{\lambda}{\lambda-1}\bfu[q]{j}\right|}\msp$
is of the order of~$\mu_{q}^{\ell}$, in the second case, of the order 
of~$\petito{\lambda^{\ell}\ell^{d}}$, hence, in both 
cases,~\equnm{dev-cp-6} holds.
More precisely, the computations go as follows. 
The reader will see that the hypothesis on the quotient plays its role
in the second case only.

\medskip
\noindent
Case 1: $\msp\mu_{q}<\lambda\msp$. 
The case $\mu_q=1$ corresponds to sequences $\bfu[q]{\ell}$ and thus
$\bfv[q]{\ell}$ having a polynomial growth.  
In which case, \equnm{dev-cp-6} directly holds.  
In the following, we assume that~$\mu_q>1$.
 
The quantity 
$\msp \wop_q(\ell) = 
        \delta_{q,1} +
        \sum_{j=2}^k 
            \delta_{q,j}\xmd\eulnum^{i\xmd\ell\xmd\theta_{q,j}}
\msp$
is periodic (with some period~$h_{q}$) and, \emph{since the 
sequence~$\sequ{\bfu[q]{\ell}}$ is monotonically increasing}, there 
exist bounds $\alpha_{q}$ and~$\beta_{q}$, 
$\msp0<\alpha_{q}\leq\beta_{q}\msp$,  such that 
\begin{equation}
\fa\ell\in\N \quantsp
\alpha_{q}\leq\wop_{q}(\ell)\leq \beta_{q}
\eqpnt
\eee
\notag
\end{equation}
It follows that
\begin{alignat}{3}
&\fa\ell\in\N \quantsp&
\mu_{q}^{\ell}\xmd\ell^{d_{q}}\xmd\alpha_{q}
+\petito{\mu_{q}^{\ell}\xmd\ell^{d_{q}}}
&\leq&\bfu[q]{\ell}&\leq 
\mu_{q}^{\ell}\xmd\ell^{d_{q}}\xmd\beta_{q}
+\petito{\mu_{q}^{\ell}\xmd\ell^{d_{q}}}
\eqvrg
\eee
\notag\\
\intertext{and}
&\fa\ell\in\N \quantsp&
\frac{\mu_{q}}{\mu_{q}-1}\xmd
\mu_{q}^{\ell}\xmd\ell^{d_{q}}\xmd\alpha_{q}
+\petito{\mu_{q}^{\ell}\xmd\ell^{d_{q}}}
&\leq&\bfv[q]{\ell}&\leq 
\frac{\mu_{q}}{\mu_{q}-1}\xmd
\mu_{q}^{\ell}\xmd\ell^{d_{q}}\xmd\beta_{q}
+\petito{\mu_{q}^{\ell}\xmd\ell^{d_{q}}}
\eqvrg
\eee
\notag
\end{alignat}
hence
\begin{equation}
\fa\ell\in\N \quantsp
\mu_{q}^{\ell}\xmd\ell^{d_{q}}\xmd\alpha'_{q}
+\petito{\mu_{q}^{\ell}\xmd\ell^{d_{q}}}
\leq\left|\bfv[q]{\ell}
-\frac{\lambda}{\lambda-1}\bfu[q]{\ell}\right|\leq 
\mu_{q}^{\ell}\xmd\ell^{d_{q}}\xmd\beta'_{q}
+\petito{\mu_{q}^{\ell}\xmd\ell^{d_{q}}}
\eqvrg
\eee
\notag
\end{equation}
where
\begin{equation}
\alpha'_{q} = \frac{\mu_{q}}{\mu_{q}-1}\xmd
\left(\alpha_{q}-\frac{\lambda}{\lambda-1}\beta_{q}\right)
\ee\text{and}\ee
\beta'_{q} = \frac{\mu_{q}}{\mu_{q}-1}\xmd
\left(\beta_{q}-\frac{\lambda}{\lambda-1}\alpha_{q}\right)
\eqpnt
\notag
\end{equation}
It follows that the quantity 
\begin{equation}
    \sum_{j=0}^{\ell} 
    \left|\bfv[q]{j}
    -\frac{\lambda}{\lambda-1}\bfu[q]{j}\right|
\notag
\end{equation}
is also of the order of~$\mu_{q}^{\ell}$ and since~$\vLl$ is of the 
order of~$\lambda^{\ell}$, \equnm{dev-cp-6} holds.

\medskip
\noindent
Case 2: $\msp\mu_{q}=\lambda\msp$.
In this case, and since by hypothesis, $L_{q}$ is \adev, 
every~$\delta_{q,j}=0$, $2\leq j\leq k$, and it holds:
\begin{equation}
\fa\ell \text{ large enough}\quantsp
\bfu[q]{\ell} =
\lambda^{\ell}\xmd\ell^{d_{q}}\xmd 
\delta_{q,1} +
   \petito{\lambda^{\ell}\xmd\ell^{d_{q}}}
\eqpnt
\eee
\notag
\end{equation}
It follows that
$\msp\displaystyle{\left|\bfv[q]{\ell}
-\frac{\lambda}{\lambda-1}\bfu[q]{\ell}\right|}\msp$
is a~$\petito{\lambda^{\ell}\xmd\ell^{d_{q}}}$ 
with $d_q \le d$
and 
\equnm{dev-cp-6} holds again, which completes the proof.
\end{proof}

 

\section{The carry propagation of a language: 
 \protect\\ \eee \eee \e 
 an ergodic point of view}
\label{s.erg-poi-vie}%

The definition of the carry propagation of a language:
\begin{equation}
\CPL = \lim_{N \to \infty} \frac{1}{N}\sum_{i=0}^{N-1}\cpL{i}
\eqvrg
\label{q.cp-erg-sum}
\end{equation}
especially if we write it as:
\begin{equation}
\CPL = \lim_{N \to \infty} 
       \frac{1}{N}
	   \sum_{i=0}^{N-1}
	   \cpL{\Succop_{L}^{i}\!\left(\varepsilon\right)}
\eqvrg
\notag
\end{equation}
inevitably reminds one of the \emph{Ergodic Theorem} (that we recall 
right below).
In this section, we explain how to set the carry propagation problem
in terms relevant to ergodic theory and we study under which
conditions and to what extent the Ergodic Theorem allows us to
conclude the existence of the carry propagation.
We begin with a very brief account of ergodic theory; for more 
detailed definitions, see~\cite{Pete83} for instance.

\subsection{Birkhoff's Ergodic Theorem}

A \emph{dynamical system}~$(\Kc,\tau)$ is a compact set~$\Kc$,
equipped with a map~$\tau$ from~$\Kc$ into itself,
called the \emph{action} of the system.
A probability measure~$\mu$ on~$\Kc$ is \emph{$\tau$-invariant} 
if~$\tau$ is measurable and if
$\msp\mu(\tau^{-1}(B))=\mu(B)\msp$ 
for every measurable set~$B$.
The dynamical system~$(\Kc,\tau)$ is said to be \emph{ergodic} 
if~$\tau^{-1}(B)=B$ implies~$\mu(B)=0$ or~$1$, for every 
$\tau$-invariant measure~$\mu$.
It is \emph{uniquely ergodic} if it admits a unique $\tau$-invariant
measure (if there exists only one $\tau$-invariant measure, then it is
ergodic).
The Ergodic Theorem then reads.

\begin{theorem}
\label{t.erg}
Let~$(\Kc,\tau)$ be a dynamical system, $\mu$~a $\tau$-invariant 
measure on~$\Kc$ and 
$\msp f\colon\Kc\rightarrow\R\msp$ a function\footnote{%
   That is, $f$ is absolutely (Lebesgue) $\mu$-integrable.} 
in~$\Lone{\mu}$.
If $(\Kc,\tau)$ is \emph{ergodic} then, \emph{for $\mu$-almost all~$s$
in}~$\Kc$,
\begin{equation}
\lim_{N \to\infty}\frac{1}{N}\sum_{i=0}^{N-1} f(\tau^i(s))=
     \int_{\KcG} f d\mu
\eqpnt
\label{q.cp-erg-int}
\end{equation}
Moreover, if $(\Kc,\tau)$ is \emph{uniquely ergodic} and if~$f$ 
and~$\tau$ are \emph{continuous}, then~\equnm{cp-erg-int} holds 
\emph{for all~$s$ in}~$\Kc$. 
\end{theorem}

This theorem states indeed two results: it says first that the limit
of the left hand-side of~\equnm{cp-erg-int} \emph{exists}, and,
second, it gives the value of this limit.
What is really of interest for us is \emph{the existence of the limit}
since, in most cases, if we know that~$\CP$ exists, we already have
other means to compute it.

We have thus to explain how to turn the language~$L$ into a compact 
set, and how to transform the successor function into a map of this 
compact set into itself.
The hypotheses of the classic formulation of the Ergodic Theorem are
rather restrictive for our case of study.
We shall rely on more recent and technical works \cite{BaraGrab16}
which significantly widen the scope of this theorem.

\subsection{Turning a numeration system into a dynamical system}

Let~$L\subseteq\Ae$ be a numeration system, that is, once again, 
the set of representations of the natural integers.
The purpose is the embedding of~$L$ into a compact set, and 
extending the successor function into an action on that set.
Since we use the \emph{radix order} on words in order to map~$L$ onto
the set of integers or, which amounts to the same thing, since we use
the \emph{most significant digit first} (MSDF) convention for the
representation of integers (assuming a left-to-right reading), we
build the compact set by considering \emph{left} infinite words.

Note that the authors from whom we borrow the results rather use the
\emph{least significant digit first} (LSDF) representation of
integers, and \emph{right} infinite words to build the same compact
set (\cite{GrabEtAl95,BaraGrab16}).
Going from one convention to the other is routine and requires just 
suppleness of mind (or a mirror).

\subsubsection{Compactification of the numeration system}

The set of left infinite words over $A$ is denoted by $\lomA$.
As we did for the words of~$\Ae$, the left infinite words are indexed
from right to left (fortunately):
if~$s$ is in~$\lomA$, we write $s= \cdots s_{2} s_{1} s_{0}$, and 
for $0\le j \le \ell \le +\infty$,
we denote by~$s_{[\ell,j]}$ the 
word~$s_{[\ell,j]}=s_{\ell}s_{\ell-1}\cdots s_{j}$.

As assumed since the beginning of this paper, the alphabet~$A$ is an
alphabet of digits, starting with~$0$: 
$\msp A=\{0,1,\ldots,r-1\}\msp$.
In order to embed finite words into (left) infinite ones, we need the
following assumption.

\begin{assumption} 
\label{ass.ini-dig}
No word of~$L$ begins with~$0$, that is, 
$\msp L\subseteq(A\!\bk\!\{0\})\xmd\Ae\msp$. 
\end{assumption}

This assumption is naturally fulfilled by the classical numeration 
systems in integer bases, the Fibonacci system, \etc
Under this assumption, the map from~$\Ae$ to~$\lomA$ defined by
$\msp w\mapsto\lomz w\msp$ is \emph{a bijection} between~$L$ 
and~$\lomz L$, that is, $L$~embeds in~$\lomA$ and can be 
identified with~$\lomz L$.

\ifelsevier\else\bigskip\fi

The set~$\lomA$ is classically equipped with the \emph{product
topology} or topology of \emph{simple convergence}, that is, the
topology induced by the \emph{distance} between elements defined by
$\dist{s,t} = 2^{-\ecar{s,t}}$ where~$\ecar{s,t}$ is the length of the
longest common right-factor of~$s$ and~$t$.
Under this topology, $\lomA$ is a \emph{compact set}, and so is any 
closed subset of~$\lomA$.

\begin{definition}
The \emph{compactification} of~$L$ is the closure of~~$\lomz L$ 
under the topology of~$\lomA$ and is denoted by~$\KcL$:
\begin{equation}
    \KcL=\overline{\lomz L}=
    \Defi{s\in\lomA}%
       {\forall j\in\N \quantsmsp
        \exists w^{(j)} \in 0^*L \quantsp 
	s_{[j,0]} \text{ is a right-factor of } w^{(j)}}
    \eqpnt
    \notag
\end{equation}
\end{definition}

The topology on~$\KcL$ is the one induced by the topology on~$\lomA$.
For every word~$w$ in~$\Ae$, 
we call the set of elements~$s$ in~$\KcL$ of which~$w$ is a 
right-factor the \emph{cylinder generated by~$w$}, and denote it 
by~$\cyliw$:\footnote{%
   It should be noted that although the \emph{notation}~$\cyliw$ does
   not bear any reference to~$L$, the \emph{set}~$\cyliw$ does depend
   on~$L$.}
\begin{equation}
    \cyliw=\Defi{s\in\KcL}{s_{[|w|-1,0]}=w}
	      = \lomA w \cap \KcL
    \eqpnt
    \notag
\end{equation}
The set of cylinders is a base of open sets of~$\KcL$.
Given any two words~$u$ and~$v$ in~$\Ae$, with~$|u|\leq|v|$,
either
$\msp\cyli{v}\subseteq\cyli{u}\msp$, a case that holds if and only 
if~$u$ is a right-factor of~$v$, or
$\msp\cyli{u}\cap\cyli{v}=\emptyset\msp$.

\subsubsection{Definition of the odometer}

Let~$L\subseteq\Ae$ be a language which satisfies \assum{ini-dig}. 
The successor function $\Succf\colon L\rightarrow L$ is naturally 
transformed into a function 
$\Succf\colon \lomz L\rightarrow \lomz L$
by setting $\Succ{\lomz w} = \lomz\Succ{w}$.

\begin{definition}
We call a function from~$\KcL$ into itself that extends~$\Succf$ an 
\emph{odometer} on~$L$, and we denote it by~$\odof$.
\end{definition}

This definition silently implies that the uniqueness of the odometer
is not guaranteed.
A particular odometer is chosen in the case where~$\Succf$
is continuous:

\begin{definition}
\label{d.con-suc}%
Let~$L$ be a language with the property that 
$\Succf\colon \lomz L\rightarrow \lomz L$ is \emph{continuous}.
Then the odometer~$\odof$ is the \emph{unique} continuous function 
from~$\KcL$ into itself that extends~$\Succf$.
\end{definition}

Uniqueness in the above definition follows from the fact that~$\lomz L$ 
is dense in~$\KcL$. 
If~$\Succf$ is \emph{not continuous}, one has to find other means to
define~$\odof$, and they depend on the cases, and on the authors (see
for instance~\cite{GrabEtAl95}, \cite{BertRigo07}).
We give such a construction, following~\cite{GrabEtAl95},
in~\secti{erg-gre-num}.

\subsubsection{Extension of the carry propagation}

We extend the map~$\rdiffop$ defined in \secti{car-pro-num} on pairs 
of finite words to pairs of elements of~$\lomA$.  
Let~$s$ and~$t$ in~$\lomA$; then:
\begin{equation}
\Delta(s,t) = \left\{
\begin{array}{ll}
\min \Defi{j\in\N}{s_{[\infty ,j]} = t_{[\infty ,j]}}\e &
\text{if}\e
\Defi{j\in\N}{s_{[\infty,j]}=t_{[\infty,j]}}\neq\emptyset\eqvrg
\\
+\infty & \text{otherwise}\eqpnt
\end{array}
\right.
\notag
\end{equation}

Conversely, the definition of~$\rdiffop$ 
on~$\big((A\!\bk\!\{0\})\xmd\Ae\big)^{2}$ 
can be deduced 
from the one on~$(\lomA)^{2}$ which is simpler and we have, 
for~$u$ and~$v$ in~$(A\!\bk\!\{0\})\xmd\Ae$,
\begin{equation}
\rdiff{\lomz u, \lomz v} = \rdiff{u,v}
    \eqpnt
    \eee
    \notag
\end{equation}

When the odometer~$\odof$ will be defined, we shall set, as in
\defin{car-pro}:
\begin{equation}
    \fa s\in\lomA\quantsp
    \capr{s}=\Delta(s,\tau_L(s))
    \eqpnt
    \eee
    \notag
\end{equation}

\begin{proposition}
\label{p.odo-cont}
If~$\odof$ is continuous, then~$\cpf$ is continuous at any point 
where it takes finite values. 
\end{proposition}

\begin{proof}
Indeed, let~$s$ in~$\KcL$ with $\capr{s}< + \infty$ and
$(s_{n})_{n\in\N}$ be a sequence of elements of $\KcL$ such that
$\dist{s,s_n}$ tends to $0$.
It means that $\ecar{s,s_n}$, the length of the longest common right
factor of~$s$ and~$s_n$, takes arbitrarily large values.  Since
$\odof$ is continuous, $\ecar{\tau_L(s),\tau_L(s_n)}$ is arbitrarily
large as well.
Let $j>\capr{s}$.
For large enough~$n$, $s_{[j,0]}=(s_n)_{[j,0]}$ and similarly
$\tau_L(s)_{[j,0]}=\tau_L(s_n)_{[j,0]}$.  
Thus $\capr{s}=\capr{s_n}$.
\end{proof}
 
If we write~$0= \lomz$, $\CP$, the carry propagation of~$L$, can be
written, if the limit exists, as:
\begin{equation}
\CP = \lim_{N\rightarrow\infty} \frac{1}{N} \sum_{i=0}^{N-1}
      \capr{\odof^{i}(0)}
      \eqvrg
\label{q.cp-erg-odo}
\notag
\end{equation}
which is the transformation of~\equnm{cp-erg-sum} we were aiming at in 
order to engage with the use of the Ergodic Theorem.

\begin{example}
Let~$p$ be an integer.
The completion~$\Kc_{p}$ of~$\lomz \Lp$ is the ring $\Z_p$ of the
\emph{$p$-adic integers} (a non-integral one if~$p$ is not a prime).

The ring~$\Z_p$ is a topological group, and the odometer~$\odof[p]$ is
the addition of~$1$, and thus a group rotation and a continuous
function.
By Proposition~\ref{p.odo-cont} the carry propagation~$\cpf[p]$ is 
continuous. 
The system $(\Kc_{p},\odof[p])$ is uniquely ergodic,
see~\cite{PythFogg02} for instance.
By applying the (second part of) Ergodic Theorem, we get an `ergodic
proof' of \propo{int-bas}.
\end{example}

\subsection{The dynamics of greedy numeration systems}
\label{s.dyn-gre-num}

In this section, we consider a case where the odometer is not defined
by continuity but rather by a combinatorial property of the
representation languages.
The greedy numeration systems have indeed the property that the
language of the representations of the natural integers is closed
under right-factors, and this will allow a meaningful definition of
the odometer, even when it is not continuous.
Our study is based on recent results due to Barat and
Grabner~\cite{BaraGrab16}.

\subsubsection{Greedy algorithm and greedy numeration systems}
\label{s.gre-num-sys}

Greedy numeration systems (GNS, for short) are a generalization of the
integer base numeration systems.
The base is replaced by a \emph{basis} (also called \emph{scale})
which is an infinite sequence of positive integers and which plays the
role of the sequence of the powers of the integer base.
The classical example is the Fibonacci numeration system where the 
basis consists of the sequence of Fibonacci numbers.
These systems have been first defined and studied in full generality 
by A.~Fraenkel~\cite{Frae85} and we have given large accounts on 
this subject in some previous works of 
ours~\cite{FrouSaka10,Rigo14}.

A \emph{basis} is a strictly increasing sequence of integers 
$\msp G = (G_\ell)_{\ell\in\N}$ with $G_0 =\nobreak 1$.
The \emph{greedy $G$-expansion} of a natural integer is the result of
a so-called \emph{greedy algorithm} --- described in this context
in~\cite{Frae85} --- for the definition of which we take a new
notation.
Given two integers~$m$ and~$p$, we write
$\msp \quot{m}{p} \msp$
and
$\msp \rest{m\,}{p} \msp$
for the \emph{quotient} and the \emph{remainder} of the Euclidean
division of~$m$ by~$p$ respectively.

\begin{definition}
\label{d.gre-alg}%
The \emph{greedy algorithm} goes as follows: given~$N$ in~$\N$, 

\thi let~$k$ be defined by the condition 
$\msp G_{k}\le N < G_{k + 1}\msp$.

\thii let
$\xmd x_k =\quot{N}{G_{k}}\xmd$ and 
$\xmd r_k =\rest{N}{G_{k}}\xmd$;

\thiii
for every~$i$, from~$i=k-1$ to~$i=0$, 
let~$\xmd x_i= \quot{r_{i+1}}{G_{i}}\xmd$ and 
$\xmd r_i = \rest{r_{i+1}}{G_{i}}\xmd$.

\smallskip
\noindent
We then have:
$\msp N = x_k\xmd G_k + x_{k-1}\xmd G_{k-1} + \cdots + x_0\xmd G_0$.  
\end{definition}

The sequence of digits 
$\msp x_{k}\xmd x_{k-1} \cdots x_{1}\xmd x_{0}\msp$ is called  
the \emph{(greedy) $G$-expansion} of~$N$ and is denoted 
by~$\reprG{N}$. 
The set of $G$-expansions is denoted by~$\LG$:
\begin{equation}
\LG=\Defi{\reprG{N}}{N\in\N}      
\eqpnt
\notag
\end{equation}
The language~$\LG$ is characterized by the following:
\begin{multline}
\e x_{k}\xmd x_{k-1} \cdots x_{0}\in\LG
\e\Longleftrightarrow\e\\
\forall i\quantvrg 0 \le i \le k\quantvrg \quantsmsp
x_{i}\xmd G_{i} + x_{i-1}\xmd G_{i-1} + \cdots + x_{0}\xmd G_{0}
< G_{i+1}
\eqpnt
\ee
\label{q.gre-exp-1}
\end{multline}
The $G$-expansion maps the natural order on~$\N$ onto the \emph{radix 
order} on~$\LG$, that is, 
$\msp N\le M\msp$ holds if and only if 
$\msp\reprG{N}\rad\reprG{M}\msp$ holds
and~$\LG$ may also be considered as an~ANS. 
\equat{gre-exp-1} becomes:
\begin{equation}
x_{k}\xmd x_{k-1} \cdots x_{0}\in\LG
\hspace{-.1em}
\e\Longleftrightarrow\e
\hspace{-.1em}
\forall i\quantvrg 0 \le i \le k\quantvrg \quantsmsp
\hspace{-.1em}
x_{i}\xmd x_{i-1} \cdots x_{0} \rad \reprG{G_{i+1}-1}
\, .
\label{q.gre-exp-2}
\end{equation}
It follows that for every~$\ell$ in~$\N$, we have:
\begin{equation}
G_{\ell} = \bfv[\LG]{\ell} 
\eqpnt
\notag
\end{equation}
For readability, we write $\msp g_{\ell} = \reprG{G_{\ell}-1}\msp$, 
and it follows from~\equnm{gre-exp-2} that
\begin{equation}
\Maxlg[\LG]=  \Defi{g_{\ell}}{\ell\in\N}
\eqpnt
\notag
\end{equation}

By construction, the language~$\LG$ satisfies \assum{ini-dig} and the
language~$0^{*}\LG$ is \emph{closed under right-factor}.
Note that~$\LG$ is not \pce in general (\cf \remar{gns-pce} below).

If the sequence of the quotients~$G_{\ell+1}/G_{\ell}$ of successive 
terms of~$G$ is bounded, with 
$\msp r=\limsup\lceil\frac{G_{\ell+1}}{G_{\ell}}\rceil\msp$,
then all $G$-expansions are words over the alphabet
$\msp A_{G}=\{0,1,\ldots,r-1\}\msp$.
In the following, we silently assume that this condition holds and 
that~$\LG$ is thus a language over the finite alphabet~$A_{G}$.
(For instance, we exclude GNS such as 
$\msp G=(\ell !)_{\ell\in\N}$.)

\subsubsection{Ergodicity of greedy numeration systems}
\label{s.erg-gre-num}%

Let~$G$ be a GNS.
We denote the successor function on~$\LG$ by~$\Succf[G]$ (rather 
than~$\Succf[\LG]$). 
The definition of the compactification of~$\LG$, which we denote
by~$\KcG$ (rather than~$\Kc_{\LG}$), takes a simpler form
since~$0^{*}\LG$ is closed under right-factor:
\begin{equation}
\KcG=\overline{\lomz\LG}=
\Defi{s\in\lomA}{\forall j\in\N \quantsp s_{[j,0]}\in 0^{*}\LG}
\eqpnt
\notag
\end{equation}
The same closure property by right-factor yields the definition of an
odometer.

\begin{theorem}[\cite{BaraGrab16,GrabEtAl95}]
\label{d.odo-gre}
Let~$G$ be a GNS.
For every~$s$ in~$\KcG$,
$\msp\lim_{j\to\infty} \Succ[G]{s_{[j,0]}}\msp$
exists and defines the odometer
$\msp\odof[G]\colon\KcG\rightarrow\KcG\msp$:
\begin{equation}
\fa s\in\KcG\quantsp
\odoG{s} =   \lim_{j\to\infty} \Succ[G]{s_{[j,0]}}
\eqpnt
\eee
\notag
\end{equation}
\end{theorem}

The carry propagation at~$s$ in~$\KcG$ is defined 
by~$\msp\cpG{s}=\rdiff{s,\odoG{s}}\msp$ as announced above,
and is denoted by~$\msp\cpG{s}\msp$ (rather than~$\capr[\LG]{s}$).
The carry propagation of~$\LG$, which we denote
by~$\CPG$ (rather than~$\CP[\LG]$) is then defined, when it exists, by 
(\cf \equnm{cp-erg-odo}):  
\begin{equation}
\CPG = \lim_{N\to\infty} 
           \frac{1}{N} \sum_{i=0}^{N-1} \cpG{\odofG^{i}(0)}
\eqpnt
\notag
\end{equation}

\begin{remark}
The language~$\LG$ is not necessarily a regular language.
But when it is, then

\thi the basis $G$ is a linear recurrent sequence, a result due to 
Shallit \cite{Shal94};

\thii the odometer $\odof[G]$ is continuous if and only if $\Succf[G]$
is realizable by a finite right sequential transducer~\cite{Frou97}.
We come back to this result and the definition of finite right
sequential transducers in the hopefully forthcoming sequel of this
work~\cite{CCSF-Part2}.
\end{remark}

The odometer~$\odofG$ may be continuous or not, as 
shown in Examples~\ref{e.Fib-erg} and~\ref{e.Fina-erg} below.
Resorting to the Ergodic Theorem requires some further hypothesis as well
as some technical developments.

\begin{definition}
A GNS~$G$ is said to be \emph{exponential} if 
it is equivalent to a sequence which is homothetic to a geometric 
progression, that is, if 
there exist two real constants~$\alpha >1$ and~$C>0$ such that
$\msp G_{\ell}\sim C\alpha^{\ell}\msp$ when~$\ell$ tends to infinity.
\end{definition}
 
Exponential GNS are of interest to us because of the following result. 

\begin{theorem}[\cite{BaraEtAl02}, Theorem~8]%
\label{t.exp-uni-erg}
If~$G$ is an exponential GNS, then the dynamical 
system~$(\KcG,\odofG)$ is \emph{uniquely ergodic}.
\end{theorem}

If~$G$ is an exponential GNS, the unique $\odofG$-invariant measure 
on~$\KcG$ is denoted by~$\mu_{G}$.
Since~$\odofG$ is not necessarily continuous, and even though the 
system~$(\KcG,\odofG)$ is uniquely ergodic, we only have the first 
part of Ergodic \theor{erg} at hand.
The following definitions and results, again borrowed
from~\cite{BaraGrab16}, are used in the proof of \theor{cp-erg} we are
aiming at and which amounts indeed to the proof that~$0$ is contained
in the set of $\mu_{G}$-almost all points for which~\equnm{cp-erg-int}
holds.

We first have an evaluation of the measure of the cylinders generated
by the maximal words, that holds without the assumption of 
exponentiality.

\begin{proposition}[\cite{BaraGrab16}, Eq.~4.8]%
\label{p.cyl-siz}%
Let~$G$ be a GNS and~$\msp\mu\msp$ a $\odofG$-invariant measure.
Then, for every~$\ell$ in~$\N$, we have:
\begin{equation}
\mu(\cyli{g_{\ell}}) \leq 1/G_{\ell}
\eqpnt
\notag
\end{equation}
\end{proposition}

\begin{notation*}
For every~$i$ in~$\N$, 
we denote by~$\msp\delta_{i}\msp$ the \emph{Dirac measure} 
at point $\reprG{i}=\nobreak{\odofG^i(0)}$ of~$\KcG$ and, 
for every~$N$ in~$\N$, by~$\msp\nu_{N}\msp$ 
the mean of these measures on the `first'~$N$ points of~$\KcG$:
\begin{equation}
    \nu_{N}=\frac{1}{N} \sum_{i=0}^{N-1} \delta_{i}
    \eqpnt 
    \notag
\end{equation}
We write~$\Charf{S}$ for the \emph{characteristic function} of a 
subset~$S$ of~$\KcG$. 
We thus have, for any~$w\in\LG$:
\begin{equation}
\nu_{N}(\cyliw)=
    \frac{1}{N} \sum_{i=0}^{N-1}\Char{\cyliw}{\odofG^i(0)}
\eqpnt
\notag
\end{equation}
\end{notation*}

The main result of~\cite{BaraGrab16} we rely on is the expression 
of~$\mu_{G}$ in terms of the~$\nu_{N}$. 

\begin{theorem}[\cite{BaraGrab16}, Theorem~2]
\label{t.gen-0}
Let~$G$ be a GNS and $(\KcG,\odofG)$ the associated dynamical system.
If~$(\KcG,\odofG)$ is uniquely ergodic with measure~$\mu_{G}$, then
for every~$w$ in~$\LG$, we have:
\begin{equation}
\lim_{N\rightarrow\infty}\xmd\nu_{N}(\cyliw)=\mu_G(\cyliw)
\eqpnt
\notag
\end{equation}
\end{theorem}

\begin{remark}
The determined reader who refers himself to \cite{BaraGrab16} will 
hardly recognize \theor{gen-0} there.
Indeed, Theorem 2 in~\cite{BaraGrab16} is much more general and 
says, roughly, that \emph{any invariant measure} on~$\KcG$ is a 
\emph{cluster point of sequences} of convex combinations of 
the~$\nu_{N}$'s.
If~$\KcG$ is uniquely ergodic, then~$\mu_{G}$ is the only possible 
cluster point, and, on the other hand, the limit of the sequence of 
the~$\nu_{N}$'s is a cluster point.
\end{remark}

\begin{remark}
One may say that \theor{gen-0}, or the convergence of the $\nu_{N}$'s 
toward~$\mu_{G}$, expresses the property that~$0$ is a \emph{generic 
point} of the system~$(\KcG, \odofG)$ in the sense that
the measure of any cylinder~$\cyliw$ is obtained as 
the limit of the \emph{statistics} induced by the orbit of~$0$ 
under~$\odofG$, or, to put it in a way that is more congruent with 
the Ergodic Theorem, we have:
\begin{equation}
	    \fa w \in A_{G}^{\,*}\quantsp
     \int_{\KcG} \Charf{\cyliw} d\mu_{G} =
     \lim_{N\rightarrow\infty}\xmd
     \frac{1}{N} \sum_{i=0}^{N-1}\Char{\cyliw}{\odofG^i(0)}
\eqvrg
\eee
\notag
\end{equation}
%
This \emph{does not imply} that
\begin{equation}
    \fa f\in\Lone{\mu_{G}}\quantsp
     \int_{\KcG} f d\mu_{G} =
     \lim_{N\rightarrow\infty}\xmd
     \frac{1}{N} \sum_{i=0}^{N-1}f(\odofG^i(0))
\eqpnt
\eee
\label{q.pro-leg-1}
\end{equation}
It suffices to take 
$\msp f = \Charf{\Omega(0)}\msp$
the characteristic function of the orbit of~$0$:
the left-hand side is~$0$ since the domain of~$\Omega(0)$ is 
denumerable and the mean of the sum in the right-hand side is 
uniformly equal to~$1$, hence has limit~$1$.

In order to prove \theor{cp-erg} below, we have to prove that the 
function~$\msp\cpfG\msp$ is somehow \emph{regular enough} to 
guarantee~\equnm{pro-leg-1}.
\equat{pro-leg-1} holds for any \emph{Riemann-integrable} 
function~$f$.
On the other hand, $\Charf{\Omega(0)}$ is typical of a function that 
is Lebesgue-integrable but not Riemann-integrable.
The function~$\cpfG$ is not Riemann-integrable either, since it is 
\emph{unbounded}.
It should be treated as an \emph{improper integral}.
\end{remark}

\subsubsection{Carry propagation in greedy numeration systems}

We are now in a position to give an ergodic proof of the existence of 
the carry propagation for a family of greedy numeration systems.

\begin{theorem}
\label{t.cp-erg}%
Let $G$ be a GNS that meets the following two conditions:

\thi $(\KcG,\odofG)$ is uniquely ergodic;

\thii $\displaystyle{\sum_{k=0}^{+\infty} \frac{k}{G_{k}}}$
is bounded.

\noindent
Then, the carry propagation
$\msp\displaystyle{\CPG = 
        \lim_{N\rightarrow\infty} \meanEi{\cpfG}{N}}\msp$
exists.
\end{theorem}

As a consequence of \theor{exp-uni-erg}, and since condition~(ii) is 
obviously satisfied by an exponential GNS, we have:

\begin{corollary}
\label{c.exp-uni-erg-2}%
If $G$ is an exponential GNS, then~$\CPG$ exists.  
\end{corollary}

Before proving \theor{cp-erg} at \secti{pro-cp-erg}, we establish
that~$\cpf[G]$ is an integrable function (\propo{cp-int-ext}).
To that end and for the sake of further developments, we first define
subsets of~$\KcG$ according to the values taken by~$\cpfG$:
\begin{equation}
\fa k\in\N\quantsp
D_{k} =\Defi{s\in\KcG}{\cpG{s}=k+1}
\eqpnt 
\eee
\notag
\end{equation}
The subsets~$D_{k}$ are Boolean combinations of cylinders.
Indeed, $\msp\cpG{s}=k+1\msp$ if and only if:

\tha $g_{k}$ is a right-factor of~$s$;

\thb no~$g_{m}$, $m\ge k+1$, is a right-factor of~$s$.

\noindent 
(Remember that for every integer~$\ell$, $g_{\ell}$ is the
maximal word of~$\LG$ of length~$\ell$ and 
$\msp g_{\ell} =\reprG{G_{\ell}-1}\msp$.)
If~$g_{k}$ is not a right-factor of~$g_{m}$, then
$\msp\cyli{g_{k}}\cap\cyli{g_{m}}=\emptyset\msp$,
hence we can write:
\begin{equation}
\fa k\in\N\quantsp
D_{k} = \cyli{g_{k}} \setminus \bigcup_{m\ge k+1} \cyli{g_{m}} 
\eqvrg 
\eee
\label{q.erg-cp-1}
\end{equation}
and the~$D_{k}$ are measurable.

One can be more precise and give an expression of the~$D_{k}$ that is
unambiguous and that will be used in actual computations.
Consider the (strict) ordering relation `being a right-factor' 
(on~$A_{G}^{*}$):
$\msp h_{1}\rfog h_{2} \msp$ if~$h_{2}$ is a right-factor of~$h_{1}$ 
(and~$h_{1}\not=h_{2}$).
Let
\begin{equation}
    \fa k\in\N\quantsp
    T(k) = \Defi{g_{m}\in\Maxlg[\LG]}{g_{m}\rfog g_{k}}
    \eee
    \notag
\end{equation}
and
\begin{equation}
T'(k) = \min\xmd T(k) =
\Defi{g_{m}\in T(k)}%
  {\text{there exists no $g_{n}$ in $T(k)$}\quantsmsp g_{m}\rfog g_{n}}
\eqpnt
\notag
\end{equation}
For the ease of writing, we define:
\begin{equation}
    I(k) = \Defi{m\in\N}{g_{m}\in T'(k)}
    \eqvrg
    \notag
\end{equation}
and we have
\begin{equation}
    D_{k} = \cyli{g_{k}} \setminus \biguplus_{m\in I(k)} \cyli{g_{m}}
    \eqvrg
\label{q.erg-cp-2}
\end{equation}
where~$\biguplus$ is the \emph{disjoint} union.
Finally, let us give some more notation.
\begin{itemize}
\item  
For every~$k$ in~$\N$, let
$\displaystyle{M_{k}=\sum_{j=k+1}^{\infty}\frac{j+1}{G_{j}}}\msp$.
\e 
By hypothesis,
$\msp\displaystyle{\lim_{k\rightarrow\infty}M_{k}=0}\msp$.
	
\item  
We write
$\msp \displaystyle{F_{k} = \bigcup_{j=0}^{j=k} D_{j}} \msp$.  

\item  
We denote by $\msp f_{k}\msp$ the function that is equal to~$\cpfG$ 
on~$F_{k}$ and to~$0$ everywhere else.

\end{itemize}

For every~$k$ in~$\N$, the function~$f_{k}$ is a \emph{step function},
that is, a linear combination of characteristic functions of
measurable sets, in this case, of the characteristic functions of
the~$D_{j}$, $0\leq j\leq k$ and of the one of 
$\KcG \setminus F_{k}$.
As a direct consequence of \theor{gen-0}, we then have:

\begin{proposition}
\label{p.erg-the-fk}%
\begin{equation}
\fa k\in\N\quantsp
\int_{\KcG} f_{k}\xmd  d\mu_{G} =
     \lim_{N\rightarrow\infty}\xmd\meanEi{f_{k}}{N}
\eqpnt
\eee
\notag
\end{equation}
\end{proposition}

From the definition of the~$D_{k}$'s, we now derive that $\cpf[G]$ is 
an integrable function:

\begin{proposition}
\label{p.cp-int-ext}%
Under the conditions of \theor{cp-erg}, $\cpf[G]$ is in $\Lone{\mu_G}$,
that is,
$\msp\displaystyle{\int_{\KcG} \cpfG\xmd d\mu_G}\msp$
exists.
\end{proposition}

\ifelsevier
\begin{proof}
\else
\begin{proofss}
\fi
Since all terms are positive, we have:
\begin{equation}
\int_{\KcG} \cpfG\xmd d\mu_G =
\lim_{k\rightarrow\infty}
   \sum_{j=0}^{k} (j+1)\xmd\mu_{G}(D_{j})
\eqpnt
\label{q.cp-int-ext}
\end{equation}
From~\equnm{erg-cp-2} and \propo{cyl-siz},
follows
\begin{equation}
    \fa k\in\N\quantsp
    \sum_{j=k+1}^{\infty} (j+1)\xmd\mu_{G}(D_{j})
    \leq
    \sum_{j=k+1}^{\infty} (j+1)\xmd\mu_{G}(\cyli{g_{j}})
    \leq
    \sum_{j=k+1}^{\infty} \frac{j+1}{G_{j}}
    = M_{k}
\eqpnt
\eee
\notag
\end{equation}
Since
$\msp\displaystyle{\int_{\KcG} f_{k}\xmd  d\mu_{G} =
   \sum_{j=0}^{k} (j+1)\xmd\mu_{G}(D_{j})}\msp$ 
we have:
\begin{equation}
\int_{\KcG} f_{k}\xmd  d\mu_{G} \leq
\int_{\KcG} \cpfG\xmd d\mu_G \leq
\int_{\KcG} f_{k}\xmd  d\mu_{G} + M_{k}
\eqvrg
\label{q.erg-cp-4}
\end{equation}
which shows not only that
$\msp \displaystyle{\int_{\KcG} \cpfG\xmd d\mu_G} \msp$ exists but 
also that 
\begin{equation}
    \lim_{k\rightarrow\infty}
\int_{\KcG} f_{k}\xmd  d\mu_{G} =
\int_{\KcG} \cpfG\xmd d\mu_G 
\eqpnt
\ifelsevier\else\tag*{\EoP}\fi
\end{equation}
\ifelsevier
\end{proof}
\else
\end{proofss}
\fi

\noindent
Of course, we have:
\begin{equation}
\fa k\in\N\quantvrg\fa N\in\N\quantsp
\meanEi{f_{k}}{N} \leq \meanEi{\cpfG}{N}
\eqvrg
\ee\e
\label{q.erg-fk-1}
\end{equation}
but it is not enough that the left-hand side has a limit for the 
right-hand side to have also one.
We need to find a bounding interval as in~\equnm{erg-cp-4} in order 
to insure that the quantity
$\msp\displaystyle{\meanEi{\cpfG}{N}}\msp$
converges when~$N$ tends to infinity.
And this is what is done in the next subsection.

\subsubsection{Proof of \theor{cp-erg}}
\label{s.pro-cp-erg}%

We begin with some more notation.
First, for every integer~$N$ in~$\N$, we write $\msp\deghN\msp$ for 
\NeM{je ne suis sžr ni du nom, ni du symbole, mais je pense que le 
concept est utile.}%
the \emph{degree}, or \emph{height}, with respect to the basis (or 
scale)~$G$, that is, the integer~$k=\deghN$ is such that 
$\msp G_{k}\leq N<G_{k+1}\msp$. 
In particular,
$\msp\lght{\reprG{N}}=\deghN+1\msp$.

Second, for every integer~$n$ in~$\N$, and for 
every~$k\geq\lght{\reprG{n}}=\ell$, we write 
$\msp\repr[G,k]{n}\msp$ for 
$\msp\repr[G,k]{n}=0^{k-\ell}\reprG{n}\msp$,
that is, $\repr[G,k]{n}$ is the \emph{unique} word 
in~$A_{G}^{k} \cap 0^{*}\reprG{n}$.
Finally, in the same way as we write $\msp\cpG{n}\msp$ for
$\msp\cpG{\tau^{n}(0)}$, we write $\msp f_{k}(n)\msp$  
for $\msp f_{k}\left(\tau^{n}(0)\right)$.
The greedy algorithm (\defin{gre-alg}) may then
equivalently be rewritten as follows.

\begin{lemma}
\label{r.gre-alg}
The $G$-expansions of integers, that is, the greedy algorithm for the 
basis~$G$, is described by the following recurrence formula:

\thi $\msp\reprG{0} = \epsilon\msp$;

\thii $\msp\fa N\in\N\msp$, if $\msp k = \deghN\msp$, then 
$\msp \reprG{N} = d\xmd\repr[G,k]{r}\msp$ \\
\PushLine 
with $\msp d = \quot{N}{G_{k}}\msp$ and
$\msp r = \rest{N}{G_{k}}\msp$.
\ee
\end{lemma}

\begin{corollary}
\label{c.gre-alg}~

Let~$N$ in~$\N$ and $k = \deghN$.
If $\msp N<G_{k+1}-1\msp$, then:
$\msp\cpG{N}=\cpG{\rest{N}{G_{k}}}\msp$.
\end{corollary}

\begin{proof}
Let~$\msp\reprG{N}=d\xmd w\msp$; 
then~$\msp\reprG{\rest{N}{G_{k}}}=w\msp$ 
with~$\msp\lght{w}=k$.
There are two possibilities: either
$\msp\cpG{N}=k+1\msp$ or $\msp\cpG{N}\leq k\msp$.
Indeed, the possibility that $\msp N+1=G_{k+1}\msp$ and 
$\msp\cpG{N}= k+2\msp$ is ruled out by the hypothesis 
$\msp N<G_{k+1}-1\msp$.

If $\msp\cpG{N}=k+1\msp$, then 
$\msp\reprG{N+1}=(\dpu)\xmd 0^{k}\msp$ and
then~$\msp\reprG{\rest{N}{G_{k}}+1}=1\xmd 0^{k}\msp$ and
$\msp\cpG{\rest{N}{G_{k}}}=k+1\msp$.

If $\msp\cpG{N}\leq k\msp$, then 
$\msp\reprG{N+1}=d\xmd w'$ and
then~$\msp\reprG{\rest{N}{G_{k}}+1}=w'\msp$ and
$\msp\cpG{\rest{N}{G_{k}}}=\cpG{N}\msp$ again.
\end{proof}

The first step toward \theor{cp-erg} is the description of the
relationship between the functions~$f_{k-1}$ and~$f_{k}$ for all
numbers less than~$G_{k+1}$, as expressed by the following.

\begin{proposition}
\label{p.cp-erg-1}
For every~$k$ in~$\N$, we have:
\begin{equation}
\fa N\in\N\quantvrg
0<N< G_{k+1}\quantsmsp
\sumTEi{f_{k}}{N} =
\sumTEi{f_{k-1}}{N} + \Inte{\frac{N}{G_{k}}}\xmd(k+1)
\eqpnt
\e
\label{q.cp-erg-1}
\end{equation}
\end{proposition}

\begin{proof}
Let~$k$ in~$\N$ and~$d_{k}$ be the largest digit that appears at 
index~$k$ (remember that the rightmost index is~$0$), that is: 
\begin{equation}
d_{k}\xmd G_{k} < G_{k+1} \leq (d_{k}+1)\xmd G_{k} 
\eqpnt 
\notag
\end{equation}
(An integer base is the case where the equality on the right holds 
for the same digit for every~$k$.)
Let us consider the integers in the interval~$[0,G_{k+1}[$ and the 
functions~$f_{k-1}$ and~$f_{k}$:
\begin{alignat}{3}
\text{If}\ee\ \ \,
0\leq n &< G_{k}-1   & \text{then } \cpG{n}&\leq k  &
\text{and}\e f_{k-1}(n)=&f_{k}(n)
\notag\\
 n &=G_{k}-1   &\text{then }\cpG{n}&= k+1  &
\ \ \text{and}\ \ f_{k-1}(n)=0,  \  &f_{k}(n)=k+1
\notag\\
G_{k}-1\leq n &< 2\xmd G_{k}-1 &  \text{then } \cpG{n}&\leq k &
\text{and}\e f_{k-1}(n)=&f_{k}(n)
\notag\\
 n &=2\xmd G_{k}-1   &  \text{then } \cpG{n}&= k+1  &
\ \ \text{and}\ \  f_{k-1}(n)=0, \ &f_{k}(n)=k+1
\notag\\[-3ex]
\intertext{\eee\eee.........}
d_{k}\xmd G_{k}-1\leq n &<G_{k+1}-1 &\ \,\text{then } \cpG{n}&\leq k &
\text{and}\e f_{k-1}(n)=&f_{k}(n)
\notag
\end{alignat}
Taking advantage that all summations go to~$N-1$, 
these~$2\xmd d_{k}+1$ lines of equalities imply~\equnm{cp-erg-1}.
\end{proof}

The aim is to obtain an equation of the same kind as~\equnm{cp-erg-1} 
but which holds for all~$N$ in~$\N$.
\corol{gre-alg} leads to the definition of \emph{$(H,L)$-extensions},
\propo{cp-erg-1} gives us a hint for the elementary arithmetic
\lemme{mean} that will pave the way to the solution
(\propo{cp-erg-2}).

\begin{definition}
\label{d.HL-ext}
Let~$H$ and~$L$ in~$\N$, with~$H<L$.
Let $\msp\alpha\colon[0,H[\rightarrow\N\msp$ be a function.
Let 
$\msp\alpha'\colon[0,L[\rightarrow\N\msp$ 
be the function defined by 
\begin{equation}
    \fa m\in\xmd [0,L[\quantsp
 \alpha'(m) = \alpha(\rest{m}{H})  
\eqpnt
\eee
\label{q.HL-ext}
\end{equation}
We call~$\alpha'$ the \emph{$(H,L)$-extension} of~$\alpha$. 
\end{definition}

The graph of the $(H,L)$-extension of~$\alpha$ consists then of the 
repetition of the graph of~$\alpha$ translated by the 
quantities~$H$, $2\xmd H$, \etc, along the $x$-axis, until~$k\xmd H$, 
where $\msp k\xmd H< L\leq(k+1)\xmd H\msp$, the last piece being cut 
off at the abscissa~$L$.

\begin{lemma}
\label{l.mean}
Let~$H$ in~$\N$ and
$\msp\alpha,\beta\colon[0,H[\rightarrow\N\msp$
be two functions with the property that there exists a~$K$ in~$\N$ 
(presumably~$K<H$) and a constant~$C$ in~$\N$ such that
\begin{equation}
\fa n\in\xmd ]0,H]\quantsp
\sumTEi{\beta}{n}\leq
\sumTEi{\alpha}{n} + C\xmd\Inte{\frac{n}{K}}
\eqpnt
\eee
\label{q.mean-lem-1}
\end{equation}
Let~$L$ in~$\N$ ($L>H$), and
$\msp\alpha',\beta'\colon[0,L[\rightarrow\N\msp$
be the $(H,L)$-extensions of~$\alpha$ and~$\beta$ respectively.
Then we have:
\begin{equation}
    \fa m\in\xmd ]0,L]\quantsp
    \sumTEj{\beta'}{m}\leq
\sumTEj{\alpha'}{m} + C\xmd\Inte{\frac{m}{K}}
\eqpnt
\eee
\label{q.mean-lem-3}
\end{equation}
\end{lemma}

\begin{proof}
From \equnm{mean-lem-1} follows in particular
\begin{equation}
    \sumTEi{\beta}{H}\leq
\sumTEi{\alpha}{H} + C\xmd\Inte{\frac{H}{K}}
\eqpnt
\notag
\end{equation}
For $m$ in~$\N$, let us write
$\msp d=\quot{m}{H} \msp$ and
$\msp n=\rest{m}{H} \msp$
(hence $m=d\xmd H+n$).
Then, using the definition of $(H,L)$-extension, one writes
\begin{equation}
    \sumTEj{\beta'}{m} =
d\xmd\left(\sumTEi{\beta}{H}\right) + \sumTEi{\beta}{n}
\eqvrg
\notag
\end{equation}
with the convention that
$\msp\displaystyle{\sum_{i=0}^{i=-1}\!\beta(i) = 0}\msp$.
We then have:
\begin{align}
    \sumTEj{\beta'}{m} & \leq
d\xmd\left(\sumTEi{\alpha}{H} + C\xmd\Inte{\frac{H}{K}}\right) + 
\sumTEi{\alpha}{n}+ C\xmd\Inte{\frac{n}{K}}
\notag\\
   & \leq
\sumTEj{\alpha'}{m} + 
   C\xmd \left(d\xmd\Inte{\frac{H}{K}} + \Inte{\frac{n}{K}}\right)
\notag\\
   & \leq
\sumTEj{\alpha'}{m} + C\xmd \Inte{\frac{m}{K}}  
\eqvrg
\notag
\end{align}
from the obvious inequality
$\msp\displaystyle{d\xmd\Inte{\frac{H}{K}} + 
\Inte{\frac{n}{K}} \leq \Inte{\frac{d\xmd H+n}{K}}}\msp$.
\end{proof}

The key statement for the proof of \theor{cp-erg} reads as follows.

\begin{proposition}
\label{p.cp-erg-2}
Let~$k$ be a fixed integer greater than~$1$.	
Then, for every~$N$ in~$\N$:
\begin{equation}
\sumTEi{f_{k}}{N}\leq
\sumTEi{f_{k-1}}{N} + \Inte{\frac{N}{G_{k}}}\xmd(k+1)
\eqpnt
\label{q.cp-erg-2}
\end{equation}
\end{proposition}

\begin{proof}
Let us establish by induction that for every~$h$, 
$h\geq k$, we have:
\begin{equation}
\fa N\in\xmd ]0,G_{h+1}]\quantsp
\sumTEi{f_{k}}{N}\leq
\sumTEi{f_{k-1}}{N} + \Inte{\frac{N}{G_{k}}}\xmd(k+1)
\eqpnt
\eee
\label{q.cp-erg-3}
\end{equation}

\propo{cp-erg-1} asserts that indeed \emph{equality} holds
in~\equnm{cp-erg-2} for all~$N$ in~$]0,G_{k+1}[$.
\ifelsevier
\else 
If~$\msp N=G_{k+1}\msp$, the summations in~\equnm{cp-erg-2} go up 
to~$\msp n=G_{k+1}-1\msp$, and then $\cpG{n}=k+2\msp$ 
and~$\msp f_{k-1}(n)=f_{k}(n)=0\msp$. 
\fi
Hence the equality still holds, but for the case 
where~$G_{k+1}=(d_{k}+1)\xmd G_{k}$ (and 
then~$\Inte{\frac{N}{G_{k}}}=d_{k+1}$) in which case the 
\emph{inequality} holds and~\equnm{cp-erg-3} is established for~$h=k$.

Let us call~$\alpha$ and~$\beta$ the restrictions to~$[0,G_{h+1}[$ 
of~$f_{k-1}$ and~$f_{k}$ respectively.
Let $\msp L=G_{h+2} - 1\msp$; from \corol{gre-alg} follows that the 
$(G_{h+1},L)$-expansions of~$\alpha$ and~$\beta$ are the restrictions 
to~$[0,G_{h+2}-1[$ of~$f_{k-1}$ and~$f_{k}$ respectively.

From \lemme{mean} we deduce that the inequality~\equnm{cp-erg-2} 
holds for every~$N$ in~$]0,G_{h+2}-1]$.

Since $\msp\cpG{G_{h+2}-1}=h+3\msp$, then
$\msp f_{k-1}(G_{h+2}-1)=f_{k}(G_{h+2}-1)=0\msp$
and~\equnm{cp-erg-2} also holds for~$N=G_{h+2}$,
which completes the induction step.
\end{proof}

\begin{proof}[Proof of \theor{cp-erg}]
	
Let~$k$ be a fixed integer.
For every~$N$ in~$\N$, there exists an
$\msp h=\sup_{n\in[0,N]}\cpG{n} -1\msp$ such that
$\msp \cpG{n}=f_{h}(n) \msp$ for all~$n$ in~$[0,N[$.
(Note that we cannot exchange the quantifiers and state: `there 
exists an~$h$ such that for every~$N$ \etc')
We then have
\begin{equation}
\meanEi{f_{h}}{N}= \meanEi{\cpfG}{N}
\eqpnt
\notag
\end{equation}
From~\equnm{erg-fk-1} and \propo{cp-erg-2} follows
\begin{multline}
\meanEi{f_{k}}{N}\leq \meanEi{\cpfG}{N} \\
\leq
\meanEi{f_{k}}{N} + 
\sum_{j=k+1}^{j=h}\frac{1}{N}\Inte{\frac{N}{G_{j}}}\xmd(k+1)
\eqpnt
\notag
\end{multline}
Two obvious majorizations give	
\begin{equation}
\meanEi{f_{k}}{N}\leq \meanEi{\cpfG}{N} 
\leq
\meanEi{f_{k}}{N} + M_{k+1}
\eqvrg
\notag
\end{equation}
which yields, when~$N$ tends to infinity, and taking 
\propo{erg-the-fk} into account:
\begin{equation}
\fa k\in\N\quantsp
\int_{\KcG} f_{k}\xmd  d\mu_{G}\leq 
\lim_{N\rightarrow\infty}\meanEi{\cpfG}{N} 
\leq \int_{\KcG} f_{k}\xmd d\mu_{G} + M_{k+1}
\eqpnt
\eee\ee
\notag
\end{equation}
If we make now~$k$ tend to infinity we get both that the 
limit

\noindent
$\msp\displaystyle{\lim_{N\rightarrow\infty}\meanEi{\cpfG}{N}}\msp$
exists, and that this limit is 
$\msp\displaystyle{\int_{\KcG} \cpfG\xmd  d\mu_{G}}\msp$.
\end{proof}

In the case where the language~$\LG$ of the exponential greedy 
numeration system~$G$ is \pce, we can use the results of 
\secti{car-pro-num} and give the value of the carry propagation.  
Since~$\msp G_\ell\sim C\xmd\alpha^\ell\msp$ implies that the local 
growth rate of~$\LG$ is equal to~$\alpha$, we have, by \corol{car-pro-fil}:

\begin{theorem}
\label{t.cp-erg-pce}
If~$G$ is an exponential GNS with 
\NeM{Est-ce un thŽorme ou un corollaire?}%
$\msp G_\ell \sim C\xmd\alpha^\ell$ and if $\LG$ is \pce, then~$\CPG$ 
exists and
\begin{equation}
\CPG =\frac{\alpha}{\alpha-1}
\eqpnt
\notag
\end{equation}
\end{theorem}

The next section deals with a family of greedy numeration systems
which have \pce languages, namely $\beta$-numeration systems.

\begin{remark}
\label{r.gns-pce}%
It is somewhat unsatisfactory to have to put an hypothesis on~$\LG$ 
directly.
It would be more natural to have a condition on the basis~$G$ itself
which would insure that~$\LG$ be \pce. 

A necessary condition for~$\LG$ to be \pce is that the sequence 
$\displaystyle{\left\lfloor\frac{G_{n+1}}{G_{n}}\right\rfloor}$ be 
\emph{non-increasing}.
But it is not a sufficient condition, as shown by the sequence
$\msp G= 1, 2, 3, 5, 9, 14, 23, \ldots\msp$:
the representation of~$8$ is~$1\xmd1\xmd0\xmd0$ but 
neither~$1\xmd1\xmd0$ nor~$1\xmd1$ are in~$\LG$.
\end{remark}

\begin{remark}
\label{r.gns-cp-cal}%
The above computations also open the way for the computation 
of~$\CPG$ that would be independent from the \pce hypothesis.
From~\equnm{erg-cp-2} and~\equnm{cp-int-ext} follows:
\begin{equation}
\int_{\KcG} \!\cpfG\xmd d\muG =\!
\sum_{k\geq0} (k+1)\xmd\muG(D_{k})=\! 
\sum_{k\geq0} 
(k+1)\left(\muG(\cyli{g_{k}})-
           \!\!\!\sum_{m\in I(k)}\!\!\muG(\cyli{g_{m}})\right)
.
\label{q.gns-cp-cal-1}
\end{equation}
Instead of using the measure of~$D_{k}$ for every~$k$, it is more 
efficient to compute the sum in~\equnm{gns-cp-cal-1} `layer by layer' 
so to speak.
If we invert the relation~$I$, that is, if we write
$\msp J(m)= \Defi{k}{m\in I(k)}$,
$J(m)$ is a singleton for every~$m$ since~$g_{\ShiftInd{J(m)}}$ is the 
\emph{longest} right-factor of~$g_{m}$ in~$\Maxlg[\LG]$.
The contribution to the sum of the `layer'~$\cyli{g_{k}}$ will be
$(k+1)-\left(J(k)+1\right)=k-J(k)$.
Since~$g_{0}=\varepsilon$ and~$\muG(\cyli{\varepsilon})=1$, we then 
have:
\begin{equation}
\int_{\KcG} \!\cpfG\xmd d\muG = 1 +
\sum_{k\geq1} 
\left(k-J(k)\right)\xmd\muG(\cyli{g_{k}})
\eqpnt
\label{q.gns-cp-cal-2}
\end{equation}
In~\cite{BaraGrab16}, a machinery has been developed for computing the
measure of the cylinders~$\cyli{g_{k}}$ which then would allow one to
obtain the value of the carry propagation without the \pce hypothesis
and the results of \secti{car-pro-num}.
Some examples of the usage of~\equnm{gns-cp-cal-2} are given 
below.
\end{remark}

\subsection{Beta-numeration}
\label{s.beta-num-gen}

Let $\beta>1$ be a real number.
The definition of a GNS associated with~$\beta$ --- due to
Bertrand--Mathis~\cite{BertMath89} --- goes in three steps: the
definition of the \emph{$\beta$-expansion} of real numbers, the one of
\emph{quasi-greedy} $\beta$-expansion of~$1$, and finally the one of
the basis~$\Gb$.

For any real number~$x$, let us denote by~$\frct{x}$ the 
\emph{fractional part} of~$x$, that is, $\frct{x}= x-\ntgr{x}$.
In~\cite{Renyi57}, R\'enyi proposed the following greedy algorithm 
for any~$x\in[0,1]$:
let~$r_{0}=x$ and, for every ~$i\ge1$, let 
$ x_{i}= \ntgr{\beta\xmd r_{i-1}}$
and $ r_{i}=\frct{\beta\xmd r_{i-1}}$.
Then,
\begin{equation}
x=\sum_{i=1}^{+\infty} x_{i}\xmd \beta^{-i}	
\ee\text{with}\ee
\fa i \geq1\quantsmsp
x_{i}\in\Ab=\{0,\ldots,\Digmu{\lceil\beta\rceil}\}
\eqpnt
\label{q.beta-exp-1}
\end{equation}
The sequence $\bge(x)= (x_i)_{i\ge 1}$ is called the 
\emph{$\beta$-expansion} of~$x$.
Seen as a \emph{right} infinite word of~$\Ab^{\omega}$, it is the 
\emph{greatest} in the lexicographic ordering of~$\Ab^{\omega}$ for 
which~\equnm{beta-exp-1} holds.
When the expansion ends in infinitely many~$0$'s, it is said to be 
\emph{finite} (and the $0$'s are omitted).
If~$x$ is greater than~$1$, the same algorithm is used 
for~$x\xmd\beta^{-k}$ such that~$x\xmd\beta^{-k}\in[0,1]$ and then 
the radix point is placed after the $k$-th digit; we thus obtain the 
$\beta$-expansion of~$x$ for any~$x$ in~$\R_{+}$.
 
Let $\bge(1)=(t_n)_{n\ge1}$ be the $\beta$-expansion of~$1$.
We define the \emph{infinite word}~$\qge$, called
the \emph{quasi-greedy expansion} of~$1$, in the following 
way.
If~$\bge(1)$ is infinite, then~$\qge=\bge(1)$.
If~$\bge(1)$ is finite, of the form $\bge(1)=t_1 \cdots t_m$, 
$t_m\neq 0$,  then~$\qge=(t_1 \cdots t_{m-1}(t_m-1))^\omega$.
It is easy to see that~$\qge$ is a $\beta$-representation 
of~$1$, that is, it satisfies~\equnm{beta-exp-1} for~$x=1$.

\begin{definition}[\cite{BertMath89}]
\label{d.beta-GNS}
Let~$\beta>1$ be a real number and $\qge=(d_i)_{i\geq1}$ the 
quasi-greedy expansion of~$1$.
The \emph{canonical greedy numeration system associated with~$\beta$} 
is defined by the basis $\Gb=(G_{\ell})_{\ell\in\N}$ inductively 
defined by: 
\begin{equation}
G_{0}=1
\ee\text{and}\ee
\fa\ell\geq1\quantsp
G_{\ell}=d_{1}\xmd G_{\ell-1}+ d_{2}\xmd G_{\ell-2}+
             \cdots+d_{\ell}\xmd G_{0} + 1
\eqpnt
\notag
\end{equation}
\end{definition}

The canonical GNS associated with~$\beta$ is exponential as asserted by 
the following.

\begin{proposition}[\cite{BertMath89}]
\label{p.beta-GNS-exp}
Let~$\beta>1$ be a real number and $\Gb=(G_{\ell})_{\ell\in\N}$ 
the canonical GNS associated with~$\beta$.
There exists a real constant~$K>0$ such that
$\msp G_\ell \sim K\xmd\beta^{\ell}\msp$.
\end{proposition}

It is easy to verify that~$\Ab=A_{G_{\beta}}$, but, of course, the 
$\Gb$-representation of an integer~$n$ is not the same as the 
$\beta$-expansion of~$n$.
With a slight abuse, we nevertheless write~$\Lb$ (rather 
than~$L_{G_{\beta}}$) for the representation language of~$\Gb$.
The language~$\Lb$ is characterized by the following.

\begin{proposition}[\cite{Parr60}]
\label{p.beta-GNS-lan}
Let~$\beta>1$ be a real number and $\qge=(d_i)_{i\geq1}$ the 
quasi-greedy expansion of~$1$.
A word $\msp w=w_k \cdots w_0\msp$ is in~$\Lb$ if and only if for 
every~$i$, $0\leq i\leq k$,
$\msp w_i \cdots w_0 \lex d_1 \cdots d_{i+1}\msp$.
\end{proposition}

A comprehensive survey on $\beta$- and $\Gb$-numeration systems can be
found in~\cite{FrouSaka10}.
We now study the carry propagation in these numeration systems.
We write~$\CP[\beta]$ rather 
than~$\CP[L_{\beta}]$.
The last two propositions and the results of \secti{pro-cp-erg} 
immediately imply the following.

\begin{corollary}
\label{c.beta-lan-pce}~

Let~$\beta>1$ be a real number.
Then, the language~$\Lb$ is a \pce language.
\end{corollary}

\begin{proof}
\thi $\Lb$ is prefix-closed since, by definition of the lexicographic 
order, we have
$\msp d_1\cdots d_{j+1}\lex d_1\cdots d_{i+1}\msp$ for every $j\le i$.

\thii $\Lb$ is extendable since if~$w$ is in~$\Lb$, then~$w\xmd0$ is 
in~$\Lb$ as well since we have
$\msp d_1 \cdots d_{i+1}\xmd0 \lex d_1 \cdots d_{i+1}\xmd d_{i+2}\msp$.
\end{proof}

\begin{corollary}
\label{c.beta-cp}
Let~$\beta>1$ be a real number.
Then, the carry propagation of the language~$\Lb$ exists and
is equal to:
\begin{equation}
\CP[\beta]=\frac{\beta}{\beta-1}
\eqpnt
\notag
\end{equation}
\end{corollary}

The value of the carry propagation can also be computed directly
from~\equnm{gns-cp-cal-2} according to the self-overlapping properties
of the quasi-greedy expansion of~$1$ $\qge=(d_i)_{i \ge 1}$.
We develop below two examples borrowed from~\cite{BaraGrab16}.

\begin{example}
\label{e.dom-ergo}%
This first example is the case where the quasi-greedy expansion 
of~$1$ is such that no left factor $d_1\cdots d_m$ of $\qge$ 
has a right-factor of the form $d_1\cdots d_k$, $k <m$ ---
this is the case for instance when $d_1 > d_j$ for every $j \ge 2$.
This implies in particular, with the notation of \remar{gns-cp-cal},
that $J(k)=0$ for all $k \ge 1$.
In~\cite[Example~5]{BaraGrab16}, the measure of cylinders is computed  
for this case and expressed by the following:
\begin{equation}
	\fa k\geq1\quantsp
	\mub(\cyli{d_1 \cdots d_k})=(\beta-1)\xmd\beta^{-k-1}
\eqpnt
\eee
\label{q.dom-erg}
\end{equation} 
Since the derivation of the series expansion of
\begin{equation}
\frac{1}{\beta-1}=\sum_{k\geq1}\frac{1}{\beta^{k}}
\ee\text{yields}\ee
\frac{1}{(\beta-1)^{2}}=\sum_{k\geq1}\frac{k}{\beta^{k+1}}
\eqvrg
\notag
\end{equation} 
Equations~\equnm{dom-erg} and~\equnm{gns-cp-cal-2}together gives: 
\begin{equation}
\CP[\beta] = 1 + \sum_{ k \ge 1}k\xmd\mub(\cyli{d_1 \cdots d_k})
           = 1 + (\beta-1)\sum_{ k \ge 1}\frac{k}{\beta^{k+1}}
		   = \frac{\beta}{\beta-1}
\eqpnt
\notag
\end{equation}
\end{example}

\begin{example}
\label{e.tri-bon-erg}%
\textbf{The Tribonacci numeration system}. 
Let~$\psi$ be the zero greater than~$1$ of the 
polynomial~$X^3-X^2-X-1$. 
Then $\bge[\psi](1)=111$ and $\qge[\psi]=(110)^\omega$ ($\psi$ is what 
is called below a \emph{simple Parry number}).
It follows that in this case,
$J(1)=0$, $J(2)=1$ and $J(k)=k-3$ for all $k \ge 3$.

In~\cite[Example~2]{BaraGrab16}, the measure of cylinders is computed  
for this case and given by the following:
\begin{equation}
\mup(\cyli{d_1})=1-\psi^{-1}
\ee\text{and}\ee
\fa k\geq 2\quantsmsp
\mup(\cyli{d_1\cdots d_k})=\psi^{-k-1}
\eqpnt
\notag
\end{equation}
\equat{gns-cp-cal-2} then becomes: 
\begin{align}
\CP[\psi]  & = 1 
   + \mup(\cyli{d_1}) + \mup(\cyli{d_1d_2})
   + 3 \sum_{k\ge 3} \mup(\cyli{d_1 \cdots d_k})
\notag\\
  & = 1 + (1-\frac{1}{\psi})+\frac{1}{\psi^3}
        + 3 \sum_{k\ge 3}\frac{1}{\psi^{k+1}}
    = 2 - \frac{1}{\psi}+\frac{1}{\psi^3}
	    + \frac{3}{\psi^3(\psi-1)}
\notag\\
& = \frac{2\xmd\psi^{4}-3\psi^{3}+\psi^{2}+\psi+2}{\psi^3(\psi-1)}
  = \frac{\psi^{4}+(\psi-2)(\psi^{3}-\psi^{2}-\psi-1)}{\psi^3(\psi-1)}
  = \frac{\psi}{\psi-1}
  \e
\notag
\end{align}
since~$\psi^{3}-\psi^{2}-\psi-1=0$.
\end{example}

It may seem frustrating that  these computations of~$\CPG$ are 
conducted precisely in cases where the result is already known but, 
on the other hand, it is interesting to consider these cases where 
\emph{two completely different computation methods} may be conducted 
(and luckily give the same result).
Along the same line, we finally say a word on numeration systems which
are relevant to both methods of~\secti{car-pro-rat} and that
of~\secti{erg-poi-vie}.
The latter put into light, of course, the continuity, or 
non-continuity, of the odometer.
We begin with some more definitions and results.

\begin{definition}[\cite{FrouSaka10}]
A real number~$\beta$ greater than~$1$ is called a \emph{Parry 
number} if the  
$\beta$-expansion of~$1$ is finite or infinite eventually periodic.
If the $\beta$-expansion of~$1$ is finite, then~$\beta$ is called a
\emph{simple Parry number}.
\end{definition}

\begin{proposition}
\label{p.Par-num}%
Let~$\beta$ be a Parry number.

\thi \emph{\cite{FrouSolo96}}
The language~$\Lb$ is \emph{rational}.

\thii \emph{\cite{FrouSolo96}}
The automaton~$\Ac_{\beta}$ which recognizes the 
language~$0^{*}\xmd\Lb$ is \emph{strongly connected}.

\thiii \emph{\cite[Proposition~7.2.21]{Loth02}} $\beta$ is \emph{the dominant
root} of the characteristic polynomial of the adjacency matrix 
of~$\Ac_\beta$. 

\thiv \emph{\cite{GrabEtAl95}}
The odometer~$\odof[\!\beta]$ is \emph{continuous} if and only if~$\beta$ is 
a \emph{simple} Parry number.
\end{proposition}

It follows from (i)--(iii) that if~$\beta$ is a Parry number, 
then~$\Lb$ and all its quotients are (rational, \pce, and) \dev 
languages with~$\beta$ as local growth rate and \theor{cp-rat-lan} 
yields an algebraic proof of the following particular case of 
\corol{beta-cp}.

\begin{corollary}
\label{c.bet-num}%
If~$\beta$ is a Parry number, then~$\CP[\beta]$ exists and
$\msp\displaystyle{\CP[\beta]= \frac{\beta}{\beta-1}}\msp$.
\end{corollary}

If~$\beta$ is a simple Parry number, it follows from~(iv) that, 
conversely, the Ergodic Theorem directly implies the existence of the 
carry propagation of~$\Lb$ (\via \propo{odo-cont}).

\begin{example}
[\examp{Fib-1} continued]
\label{e.Fib-erg}%
The Fibonacci numeration system is the canonical GNS
associated with the golden mean~$\varphi$.
Since $\bge[\varphi](1)=1\xmd1$, $\varphi$ is a simple Parry number. 
The set of greedy expansions of the natural integers is 
$L_{\varphi}=1\{0,1\}^* \setminus\{0,1\}^*11\{0,1\}^*
\cup\{\varepsilon\}$. 
The automaton below recognizes~$0^*L_{\varphi}$.

The compactification of~$\lomz L_{\varphi}$ is
\begin{equation}
\Kc_{\varphi}=\Kc_{L_{\varphi}}=
\Defi{s=(\cdots s_2 s_1 s_0)\in\lom{\{0,1\}}}%
{\forall j\quantsmsp s_{[j,0]}=s_{j}\cdots s_{0}\prec(10)^\omega}
\eqpnt
\notag
\end{equation}
For instance:
$\odo[\varphi]{\lom{(0\xmd1)}\xmd0(0\xmd1)^n}
=\lom{(0\xmd1)}0\xmd1\xmd0^{2n-1}$. 
On the other hand,
\begin{equation}
\odo[\varphi]{\lom{(0\xmd1)}} 
       = \lim_{n\to\infty}\Succ[\varphi]{(0\xmd1)^n} 
       = \lim_{n\to \infty}1\xmd0^{2n-1}= \lomz
\eqpnt
\notag
\end{equation}
This illustrates the fact that the 
odometer~$\odof[\varphi]$ is
continuous.

\begin{figure}[h]
\centering
\VCDraw{%
\begin{VCPicture}{(-1,-1.5)(6,1.5)}
\State[]{(5,0)}{1}
\State[]{(0,0)}{0}
\Final[s]{0}
\Final[s]{2}
\Final[s]{1}
\Initial[w]{0}
\ArcR{0}{1}{1}
\ArcR{1}{0}{0}
\LoopN{0}{0}
\end{VCPicture}%
}
\label{f.autom-fib}
\end{figure}

\end{example}

\begin{example}
[\examp{Fi-na} continued]
\label{e.Fina-erg}%
The Fina numeration system is
the canonical GNS associated with
$\theta=\frac{3+\sqrt{5}}{2}$. 
Since  $\qge[\theta]=2\xmd 1^\omega$,
$\theta$ is a Parry number which is not simple.
We have:
\begin{equation}
\odo[\theta]{\lom{1}} = \lim_{n\to\infty} \Succ[\theta]{1^{n}}
                 = \lim_{n\to\infty} 1^{n-1}2 = \lom{1}\xmd 2
\eqpnt
\notag
\end{equation}
On the other hand, let us consider
the sequence $\msp(w^{(n)})_n=(\lomz2\xmd1^n)_n\msp$.
We have: 
$\msp\lim_{n\to \infty}(w^{(n)})_n=\lom{1}\msp$, and, for each~$n$, 
$\odo[\theta]{w^{(n)}}=\lom{1}\xmd0^{n+1}$, which tends
to $\lomz$, thus the odometer $\odof[\theta]$ is not continuous.
\end{example}


%
\addcontentsline{toc}{section}{References}  
%


\end{document}